\documentclass[a4paper,oneside,titlepage,12pt]{report}

\usepackage{graphicx,subfigure,psfrag}
\usepackage{amsmath}
\usepackage{amssymb}
\usepackage{amsthm}
\usepackage{color}
\usepackage{mathabx}
\usepackage[Glenn]{fncychap}
\usepackage{baththesis}
\usepackage{epsfig}
\usepackage[latin1]{inputenc}
\usepackage{amsmath,amssymb,amsthm,amsxtra}
\usepackage{enumerate}
\usepackage{verbatim}
\usepackage{amsfonts}
\usepackage{graphicx}
\usepackage{fancyhdr}  
\usepackage{algorithm}	
\usepackage{algorithmic}
\usepackage{mathrsfs} 
\usepackage{subfigure}
\usepackage{geometry}
\usepackage{color}
\usepackage{mathabx}
\definecolor{green}{rgb}{0,1,0}
\usepackage{cancel}
\usepackage{soul}
\usepackage{float}
\usepackage{url}
\usepackage{multicol}
\usepackage{tikz}
\usepackage{float}
\usepackage[margin=0pt,font=small,labelfont=bf,textfont=it]{caption}

\theoremstyle{plain}
\newtheorem{Theorem}{Theorem}[chapter]		

\newtheorem{Lemma}[Theorem]{Lemma}

\newtheorem{Remark}[Theorem]{Remark}

\newtheorem{Proposition}[Theorem]{Proposition}

\newtheorem{corrollary}[Theorem]{Corollary}

\theoremstyle{definition}
\newtheorem{defn}[Theorem]{Definition}
\fontsize{12}{20}

\newcommand{\properideal}{%
  \mathrel{\ooalign{$\lneq$\cr\raise.22ex\hbox{$\lhd$}\cr}}}

\title{\sffamily\bfseries Nilpotent Symplectic Alternating Algebras} 
\author{\sffamily Layla Hamad Elnil Mugbil Sorkatti}
\degree{\sffamily Doctor of Philosophy}
\department{\sffamily Department of Mathematical Sciences}
\degreemonthyear{\sffamily June 2015} \norestrictions

\begin{document}
\maketitle
\pagenumbering{arabic}

\thispagestyle{empty}
\vspace{-2.1cm}
\hspace{500cm}

\hspace{500cm}
\hspace{500cm}
\begin{center}
.. To my dad 
 \& 
my mum, who surprised me by becoming an expert in using the communication technology to keep in touch with me .. 
\end{center}
\begin{center}
.. To Paul Vaderlind ..
\end{center}
\begin{center}
..  To Mohsin Hashim  ..
\end{center}
\begin{center}
.. To my `$2^{nd}$ mum' Zalima Md. Drus ..
\end{center}
\begin{center}
.. To anyone who reads it and/or uses it as a reference ..
 \end{center}
\begin{center}
* A special dedication to my nephew Ahmed Mohammed Omer *
\end{center}

\clearpage

\thispagestyle{empty}
\chapter*{SUMMARY}
\vspace{-2.1cm}
\hspace{500cm}
\noindent

We develop a structure theory for nilpotent symplectic alternating algebras.
We then give a classification of all nilpotent symplectic alternating algebras of dimension up to $10$ over any field $\mathbb{F}$.
The study reveals a new subclasses of powerful groups that we call \emph{powerfully nilpotent groups}, \emph{perfect nilpotent groups} and \emph{powerfully soluble groups}. 

\clearpage

\thispagestyle{empty}
\chapter*{ACKNOWLEDGEMENTS}
\vspace{-2.1cm}
`O mankind! We created you from a single (pair) of a male and a female, and made you into nations and tribes, that ye may know each other (not that ye may despise (each other). Verily the most honoured of you in the sight of Allah is (he who is) the most righteous of you. And Allah has full knowledge and is well acquainted (with all things)'.  The Quran 49:13 (Surah Al-Hujurat) \\\\
First and foremost, I am ever grateful to Allah of having the opportunity to continue my studies in the beautiful City of Bath.  My short stay here has been extremely memorable. I met interesting people of different nationalities and backgrounds. They have enriched my life beyond my expectation.\\\\
My acknowledgment will begin with my Supervisor, Gunnar Traustason for all his help and encouragement during the last three years. I believe that his supervision and positive attitude have made my PhD a very enriching experience. I wish him all the best for the future.\\\\
I extend my acknowledgment to the entire Geometry Group at Bath University, especially Geoff Smith, Gregory Sankaran, Alster King, David Calderbank, Francis Burstall and Xinping Su.  Also, to the Postgraduate Geometry \& Algebra Seminar Group, especially Acyri, Matthew and Nathan. My special thanks to Atika Ahmed for all her help, generosity and sincere friendship. Not forgetting, the Department of Mathematics and all the people on Level 2, especially Ann Linfield and Mary Banies. 
I would also like to thank my office mates in $4W\,2.21$; Horacio, Alex, Sarah, Istvan, Joshua, Tom and Fancies.
It is not least because of Prof. Paull Vaderlind, Dr. Mohsin Hashim and Eilaf Sorkatti. Their help, constant support and advice have been invaluable.\\\\
My acknowledgment also goes to my bother Mugbil who was always there for me and made time to travel with me.
 Many thanks go to my friends Abdulraheem, Ayoub Bashier, Amy, Ammarah, Asiah, Fatima, Khadra, Maren, 
Meruyert \& Aizhan, Alnada, Naseerah, Natalya, Oscar Xi Wang, Passawan, Rougya, Sian, Soumya \& Sanaa, Sozy and Zaynub for their friendship, empathy and great sense of humour during my sojourn in Bath. 
Special thanks to my friend Maren for the fruitful discussions. 
Warm thanks to Amy, Naaserah and Zalima for the numerous dinner invitations.
My thank you, too to my relatives in Sudan, Cardiff and Birmingham. I also would like to mention especially the kindest couple in the world, Osman and Nafisa for their worry and care for me.
I am so grateful to most of the people mentioned for making their homes very welcoming to me all the time. My apology to anyone I have left out.\\\\
 I must not forget the staff at the English Language Center, namely Tim Ratcliffe. I also have been fortunate to cross paths with a lovely, helpful lady Mrs Del Davies from the University Accommodation Center. My thanks also go to Anne Ellis, Santander Bank, University of Bath branch for her patience and honesty in advising me on banking matters. \\\\
A big thank you to Khartoum University and University of Bath for their financial support that made my research for this thesis possible and for funding several visits and conferences, including Groups St Andrews $2009$ and $2013$.\\\\
Last but not least, to my parents for their long distance unwavering support, endurance and great patience.  My brother Mohammed and Sister Eilaf for taking care of them in my absence. 

\thispagestyle{empty}

\tableofcontents
\addtocontents{toc}{\def\protect\@chapapp{}}

\chapter{Introduction}
In this thesis we study certain algebraic structures called symplectic alternating algebras. Symplectic alternating algebras originate in a study of powerful $2$-Engel groups [\cite{3},\cite{4}] although here we will study them purely 
as structures that are interesting in their own right with many beautiful properties. This thesis continues the development of the theory of symplectic alternating algebras that was started in \cite{1}. Some general theory was also developed in \cite{2}. The aim is to explore these algebraic structures and in particular
to develop a theory for nilpotent symplectic alternating algebras.\\ \\
Let $\mathbb{F}$ be a field. A \emph{Symplectic Alternating Algebra} over $\mathbb{F}$ is a triple $(L,\,(\ ,\ ),\ \cdot\,)$ where
$L$ is a symplectic vector space over $\mathbb{F}$ with respect to a non-degenerate alternating form $(\ ,\ )$ and $\cdot$ is a
bilinear and alternating binary operation on $L$ such that 
$(u\cdot v,w)=(v\cdot w,u) $ for all $u,v,w\in L$. We often denote Symplectic Alternating Algebra more shortly by SAA.
\section{Connection with Engel groups and the Burnside Problem}
As we said above SAA's originate in some work on powerful $2$-Engel groups. We will not be exploring this connection in this thesis
but will be primarily looking at SAA's as structures interesting in their own right. In this section we however briefly describe the origin as a background to our work.
As a starting point we first mention the famous Burnside problems from which Engel groups originate.
These were posed in $1902$ by William Burnside \cite{Burnside02}.\\ \\
{\bf The General Burnside Problem} Is a finitely generated periodic group necessarily finite?\\\\
{\bf The Burnside Problem} 
If $B(r,n)$ is the largest $r$-generator group of exponent $n$. For what values of $r$ and $n$ is 
$B(r, n)$ finite?\\\\
{\bf The Restricted Burnside Problem} For what values of $r$ and $n$ is there an upper bound on the orders of finite $r$-generator groups of exponent $n$?\\\\
In $1964$ Golod \cite{golod1964nil} answered the general Burnside 
problem by constructing a counter example that is a finitely generated infinite $p$-group.
For the restricted Burnside problem the answer turns out to be that such an upper bound exists for all $r$ and $n$. 
P. Hall and Higman \cite{hall1956p} reduced this problem
to the case when $n$ is a prime power exponent. 
The solution was then completed in 1989 by Zel'manov [\cite{zel1991-odd},\cite{zel1991}] .\\\\
We next turn to the Burnside problem, For $n=2, 3, 4, 6$ it is known that $B(r,n)$ is finite. 
In general the answer is however negative and $B(r,n)$ is known to be infinite when $r \geq 2$
and $n$ is large enough [\cite{adeiian1979burnside}, \cite{ivanov1992burnside}, \cite{lysenok1996infinite}]. Surprisingly until now it is unknown whether $B(2,5)$ or $B(2,8)$ is finite or not. \\ \\
The Burnside problems lead naturally to the Engel-identities. Engel groups have their origin in William Burnside's paper \cite{Burnside02}. Recall that the commutator of two elements $x$ and $y$ in a group is defined as 
$[x, y]= x^{-1}y^{-1}xy$. We adopt the left normed convention for commutators of more than two elements. Thus 
$[x_1, x_2, \ldots, x_n] = [\ldots[[x_1, x_2], \ldots], x_n]$. We also define $[y,_{n} x]$ recursively by $[y,_{0} x] = y$
and $[y,_{n+1} x] = [[y,_{n} x], x]$ .
\begin{defn}
Let $G$ be a group. $G$ is said to be an \emph{Engel} group if for each pair $(x, y) \in G \times G$ there exists an integer 
			$m = m(x, y)$ such that \[ [y,_{m} x] = 1 .\] 
A group is said to be an $n$-Engel group if 
			$m(x, y) \leq n$
for all $x, y \in G$.
\end{defn}
\begin{Remark} Obviously $G$ is an $0$-Engel group if and only if $G = \{e\}$ and $1$-Engel groups are the abelian groups.
\end{Remark}
\noindent
It is well know that groups of exponent $2$ are abelian that implies that any $r$-generator group of exponent $2$ is of order at most $2^r$. Burnside showed in \cite{Burnside02} that it is also true that groups of exponent $3$ are locally finite. It was Burnside who observed that in groups of exponent $3$, any two conjugates $a, a^b$ commute but this property is equivalent to the $2$-Engel identity $[[b,a],a]=1$ and thus these groups are $2$-Engel groups. It has been shown later in \cite{Burnside01} by Burnside that the $2$-Engel groups also satisfy the laws
\begin{eqnarray*}
[x, y, z]  &=& [ y, z, x] \\
{[x, y, z]}^3 &=& 1 
\end{eqnarray*}
and thus any $2$-Engel group where there is no element of order $3$ would be nilpotent of class at most $2$. One can see furthermore that the following identity holds
\begin{eqnarray*}
[x, y, z, t] &=& 1
\end{eqnarray*}
which apparently was noticed first by Hopkins \cite{1929}.
The fact that these laws characterize $2$-Engel groups is however usually attributed to Levi \cite{1942}. The last identity shows that 
$2$-Engel groups are nilpotent of class at most $3$. Further details can be found in \cite{2011}. That we have a complete understanding of $2$-Engel groups is however no more true then saying that
the law $xy=yx$ tells us all about abelian groups. There are still a number of open question regarding $2$-Engel 
groups. For example, the following well known problems were raised by Caranti \cite{2006unsolved1}.\\ \\
{\bf Problem 1}.
(a) Let $G$ be a group of which every element commutes with all its endomorphic images. Is $G$ nilpotent of class at most $2$?\\\\
(b) Does there exist a finite $2$-Engel $3$-group of class three such that 
		$\mbox{Aut\ } G = \mbox{Aut}_c G \cdot \mbox{Inn\ }G$ 
where $\mbox{Aut}_c G$ is the group of central automorphisms of $G$?\\ \\
We have a positive answer for question (b). In 2010 Abdollahi, A., et al constructed such an automorphism group \cite{abdollahi2010finite} using GAP and Magma computations.\\\\
As we mentioned before SAA's originate from a study of powerful $2$-Engel groups.
\begin{defn} A finite $p$-group, $p$ odd, is said to be \emph{powerful} if $[G, G] \leq G^p$. Or equivalently if $G/G^p$ is abelian. If $p=2$ then $G$ is powerful if $[G, G] \leq G^4$. \end{defn}
\noindent
Questions about $p$-groups can often be reduced to powerful $p$-groups.
In $2008$ some work was done by Moravec and Traustason in \cite{3} on powerful $2$-Engel groups. 
			Now any powerful $3$-group of exponent $3$ is abelian 
								and
one might therefore expect that the class of powerful $2$-Engel $3$-groups would be smaller than $3$.\\\\
\noindent
In fact Moravec and Traustason showed that this is the case for $3$-generator groups.
\begin{Proposition} [\cite{3}]
Every $3$-generator powerful $2$-Engel group is nilpotent of class at most $2$. \end{Proposition}  \noindent
We know that the class of $2$-Engel groups is at most $3$. Notice that as the class of powerful $2$-Engel groups is not closed under taking subgroups, it does not follows from the last proposition that 
when the class of $2$-Engel groups is furthermore powerful then the class will be reduced to $2$.\\\\
In fact it turns out that there is a rich family of powerful $2$-Engel groups of class $3$.
\begin{defn} A powerful $2$-Engel $3$-group of class $3$ is \emph{minimal} if all the proper powerful sections have class at most $2$. \end{defn} \noindent
Moravec and Traustason classified powerful $2$-Engel groups of class three that are minimal. The study reveals that there are infinitely many minimal groups of rank $5$ and also of any even rank $\geq 4$.\\ \\
Symplectic vector spaces play a role in the classification. One of the families considered has a richer structure which lead to related algebraic structures that we call "symplectic alternating algebras". \\ \\
\noindent
In fact the study in \cite{4} reveals there is a one-to-one correspondence between symplectic alternating algebras over the field $\mbox{GF}(3)$ and a certain class of powerful $2$-Engel $3$-group of exponent $27$. 
\noindent
The groups form a class $\mathcal{C}$ that consist of powerful $2$-Engel $3$-groups $G$ with the following extra properties :\\\\
(a)  $G = \langle x , H \rangle$ where $H =\{ g \in G :\ g^9 = 1\}$ and $Z(G) = \langle x \rangle$ with $O(x) = 27$.\\ \\
(b) $G$ is of rank $2r+1$ and has order $3^{3+4r}$.\\\\
The associated symplectic alternating algebra $L(G)$ is constructed as follows. First we consider 
				$L(G) = H / G^3$ as a vector space over $\mbox{GF}(3)$. 
To this we associate a bilinear alternating form (,) and a alternating binary multiplication as follows
 		\[    [a, b]^3 = x ^{9 (  \bar{a}, \bar{b}  )  }   \]
		$$\bar{a} \cdot \bar{b} = \bar{c} \mbox{\ where\ } [a,b ] Z(G)= c^3 Z(G) \mbox{\ and\ } \bar{y}=yG^3.$$
One can show that these are well defined and turn $L(G)$ into a SAA.
Suppose that $\bar{a} \cdot \bar{b} = \bar{d}$ and $\bar{b} \cdot \bar{c} = \bar{e}$. Then
\begin{align*}
x^{9( \bar{a} \cdot \bar{b} , \bar{c} )}  &= x^{9(\bar{d} ,\bar{c} )} = [d, c]^3 = [d^3 , c] = [a,b,c] = [b, c, a] = [e^3, a] =  [e, a]^3 = x^{9 (\bar{e}, \bar{a})}\\
& = x^{9 ( \bar{b} \cdot \bar{c} , \bar{a} )}.
\end{align*}
Notice that $L(G)$ is abelian if and only if $G$ is nilpotent of class at most $2$. Traustason showed that $L(G) \cong L(K)$ if and only if $G \cong K$.\\\\
Suppose $G =\langle x, h_1, \ldots, h_{2r} \rangle$ and that $L(G) = \langle u_1, \ldots, u_{2r} \rangle$ is the corresponding SAA where $u_i = h_i G^3$. Form the structure coefficients for $L(G)$
\begin{eqnarray*}
		u_i \cdot u_j &=& \alpha_{ij}(1) u_1 + \cdots + \alpha_{ij}(2r) u_{2r} \\
			(u_i, u_j) &=& \beta_{ij} 
	\end{eqnarray*}			
where $1 \leq i < j \leq 2r$. We get the following powerful commutator relations for $G$
\begin{eqnarray*}  
		[h_i, h_j] &=&  h_1^{3 \alpha_{ij}(1)} \cdots h_{2r}^{3\alpha_{ij}(2r)} x^{3\beta_{ij}} .
 \end{eqnarray*}	
\noindent
\begin{Remark} Recall that a subgroup $K$ of a $3$-group $G$ is said to be powerfully embedded if $[K,G] \leq K^3$. We then
notice that powerful subgroups of $G$ correspond to subalgebras and powerful embedded subgroups correspond to ideals.
\end{Remark}
\noindent
{\bf Example.} Consider the symplectic alternating algebra $L = \mathbb{F} x_1 + \mathbb{F} y_1 + \mathbb{F} x_2 + \mathbb{F} y_2$ of dimension $4$, over $\mathbb{F} = \mbox{GF}(3)$ where $(x_1, y_1) = (x_2, y_2)=1$ and $(x_i, x_j) = (y_i, y_j) = (x_i, y_j)=0$ otherwise for $i,j \in \{1, 2\}$.
\begin{eqnarray*}
x_1x_2 &=& 0\\
y_1y_2 &=& -y_1 \\
x_1y_1 &=& x_2\\
x_1y_2 &=& -x_1\\
x_2y_1 &=& 0\\
x_2y_2 &=& 0
\end{eqnarray*}
The corresponding group is $G(L) = \langle x_1, y_1, x_2, y_2, x \rangle$ with the relations
\begin{eqnarray*}
[x_1, x_2] &=& 1\\
\mbox{} [y_1,y_2 ] &=& y_1^{-3} \\
\mbox{} [x_1,y_1] &=& x_2^3 x^3\\
\mbox{} [x_1,y_2] &=& x_1^{-3}\\
\mbox{} [x_2, y_{1}] & = & 1 \\
\mbox{} [x_2,y_2] & = & x^3
\end{eqnarray*}
\section{An overview of this thesis}
We now give a detailed summary of the thesis. The thesis is divided into two parts. 
In the first part we develop a structure theory for nilpotent symplectic alternating algebras. We will first discuss some background material in Chapter $2$, which we will need throughout the thesis. 
We then begin Chapter $3$ by describing some results that in particular lead to specific type of presentations that we call nilpotent presentations. All algebras with a nilpotent presentation are nilpotent and conversely any nilpotent algebra will have a nilpotent presentation. In particular we will focus on the algebras that are of maximal class and we will see that their structure is very rigid. 
In Chapter $4$ we also illustrate the theory by classifying all nilpotent SAA's of dimension up to $8$ over an arbitrary field $\mathbb{F}$. 
The classification of the nilpotent symplectic alternating algebras of dimension up to $8$ is implicit in \cite{3} although this is not done explicitly and the context there is a more general setting. 
There are three algebras and one family of nilpotent Symplectic alternating algebras of dimension $8$ over any field.\\ \\
The second part of the thesis is to deal with the challenging classification of nilpotent symplectic alternating algebras of dimension $10$ over any field. 
It turns out that the classification of algebras with a center that is not isotropic can be easily reduced to the classification 
of algebras of dimension $8$. The main bulk of work is thus about algebras with isotropic center. The dimension of the center lies between $2$ and $5$ and we deal with these cases in turn.
At some points there are interesting geometrical situations that arise, like when dealing with algebras that have an isotropic center of dimension $4$.
There are $22$ such algebras when the field is algebraically closed. Over the field $\mbox{GF\,}(3)$, where there is a $1$-$1$ correspondence with a class of powerful $2$-Engel $3$-groups, there are $25$ algebras. In general the classifications depends strongly on the field.
\section{Publication details}
Part I of this thesis is joint work with Gunnar Traustason. This has been published in the Journal of Algebra, and forms reference \cite{6}.\\\\
Part II is is also joint work with Gunnar Traustason and the majority of this is currently being prepared for publication \cite{7,8}.
\part{General Theory}
\chapter{Background Material }
\section{Symplectic Alternating Algebras}
\begin{defn}
Let $\mathbb{F}$ be a field. A \emph{Symplectic Alternating Algebra} over $\mathbb{F}$ is a triple $(L,\,(\ ,\ ),\ \cdot\,)$ where
$L$ is a symplectic vector space over $\mathbb{F}$ with respect to a non-degenerate alternating form $(\ ,\ )$ and $\cdot$ is a
bilinear and alternating binary operation on $L$ such that 
$$(u\cdot v,w)=(v\cdot w,u)  $$ for all $u,v,w\in L.$ 
\end{defn}
\noindent
Notice that $(u\cdot x, v)= ( x \cdot v, u) = - (v \cdot x, u) = (u, v\cdot  x)$ and thus the multiplication from the right by $x$ is self-adjoint with respect to the alternating form. As the alternating form is non-degenerate, $L$ is of even dimension and we can pick a basis $x_1, y_1, \ldots, x_n, y_n$ with the property that $(x_i, x_j) = (y_i, y_j) =0$ and $(x_i, y_j) = \delta_{ij}$ for $1 \leq i \leq j \leq n$. We refer to a basis of this type as a  \emph{standard basis}.\\ \\
Suppose we have any basis $u_1, \ldots, u_{2n}$ for $L$. The structure of $L$ is then determined from
\[ (u_i u_j, u_k)= \gamma_{ijk}, \quad 1 \leq i <  j  <  k \leq 2n.\]
We refer to such data as a presentation for $L$. The convention is to only list those triple values that are non-zero.
Alternatively we can describe $L$ as follows, if we take the two isotropic subspaces $\mathbb{F} x_1+ \cdots + \mathbb{F} x_n$ and 
 $\mathbb{F} y_1+ \cdots + \mathbb{F} y_n$ with respect to a given standard basis, then it is suffices to write down only the
 products $x_ix_j, y_iy_j,$ $1 \leq i < j \leq n$. The reason for this is that having
 determined these products we have determined all the triples $ (u_i u_j, u_k)$ where
 $ 1 \leq i < j < k \leq 2n$, since two of those are either some $x_i, x_j$ or
 some $y_i, y_j$ in which case the triple is determined from $x_ix_j$ or $y_iy_j$. Since
 $(x_i x_j, x_k) = (x_j x_k, x_i) = (x_k x_i, x_j)$ and $(y_i y_j, y_k) = (y_j y_k, y_i) = 
(y_k y_i, y_j)$, this put some more conditions on the products $x_ix_j$ and $y_iy_j$. \\ \\
We adopt the left-normed convention for multiple products. 
Thus $x_1x_2 \cdots x_n$ = $(\cdots(x_1$
$x_2)$
$\cdots)x_n $.
Many of the terms that will be used are analogous to the corresponding terms for related structures. Thus a subspace $I$ of a SAA $L$ is an ideal if $IL \leq I$. Also $U \leq V$ stands for `$U$ is a subspace of $V$'. \\ \\
Now let $L$ be a SAA of dimension $2n$. We next look at some general properties that hold for $L$.
\begin{Lemma}[\cite{1}] \label{lma115} 
If $I$ is an ideal of $L$ then $I^\perp$ is also an ideal. Furthermore any isotropic ideal $I$ of a SAA $L$ is abelian.
\end{Lemma}
\begin{proof}
As $(I^\perp \cdot L, I) = - (I \cdot L, I^\perp) = 0$, $I^\perp$ is an ideal in $L$. For the latter suppose that $I$ is an isotropic ideal of $L$. Thus
$I \leq I^\perp$ and so $(I \cdot I, L) = (I \cdot L, I) = 0$
implies that $I$ is abelian as required.
\end{proof}
\noindent
We define the lower central series in an analogous way to related structures like associative algebras and Lie algebras. Thus we define the lower central series recursively by 
\[ L^1=L  \text{ and } L^{n+1}=L^n L, \]
\noindent
and the upper central series by 
\[ Z_0(L)=\{0\} \text{ and } Z_{n+1}(L)=\{ x \in L : xL \in Z_n(L) \}.\]
It is readily seen that the terms of the lower and the upper central series are all ideals of $L$. The following beautiful property that shows relation between the upper and the lower central series will be used frequently.
\begin{Lemma}[\cite{1}]  \label{beautiful_relation} 
$Z_n(L)=(L^{n+1})^\perp.$
\end{Lemma}
\begin{proof}
We have
$ a \in Z_n(L) \Leftrightarrow a\underbrace{L \cdots L}_{n} = 0 \Leftrightarrow  0 = (a \underbrace{L \cdots L}_{n}, L) = ( a, L^{n+1}) \Leftrightarrow a \in (L^{n+1})^\perp$.
\end{proof}
\noindent
Notice also that $\mbox{dim\,} Z_n(L) + \mbox{dim\,} L^{n+1} = \mbox{dim\,} L$.
\begin{Lemma}[\cite{1}]
 Any one-dimensional ideal of $L$ is contained in $Z(L)$ and any two-dimensional ideal is abelian and contains a non-trivial element from $Z(L)$. 
\end{Lemma}
\begin{proof}
Let $I$ be an ideal of dimension one. Then $L/I^\perp$ is an alternating algebra of dimension one and hence abelian. It follows that $L^2 \leq I^\perp$ and hence $I \leq (L^2)^\perp = Z(L)$.\\ \\
Now let $I$ be an ideal of dimension two. Then $L/I^\perp$ is an alternating algebra of dimension two. Thus $dim\ (L/I^\perp)^2 \leq 1$ and so $(L/I^\perp)^2 < L/I^\perp$, that is $L > L^2 + I^\perp = ((L^2)^\perp \cap I )^\perp = (Z(L) \cap I)^\perp.$ Hence $Z(L) \cap I > \{0\}$.
\end{proof}
\noindent
We define simplicity for SAA in the natural way. 
$L \neq \{0\}$ is simple if it has no proper nontrivial ideals. 
\begin{defn} Suppose that $L$ is a SAA with ideals $I_1 , \ldots, I_n$ which all are SAA's and where
\[ L = I_1 \operp I_2 \operp \cdots \operp I_n. \] 
\end{defn}
\noindent
We then say that $L$ is the direct sum of the SAA's $I_1, \ldots, I_n$. 
\begin{Remark} It follows that $I_i I_j \leq I_i \cap I_j = \{ 0 \}$ when $i \neq j$. \end{Remark}
\begin{defn}
We say that a SAA is semi-simple if it is a direct sum of simple SAA's.
\end{defn}
\noindent
\noindent
As we said before, some general theory was developed in \cite{1} and \cite{2}. In particular a well-known dichotomy property for Lie algebras also holds for SAA's. Thus a SAA is either semi-simple or has a non-trivial abelian ideal. 
\begin{Theorem}[\cite{1}]
 Either $L$ has a non-trivial abelian ideal or $L$ is semi-simple. In the latter case the direct summands are uniquely determined as the minimal ideals of $L$ .
\end{Theorem}
\noindent
It is however still unknown whether there exist a non-trivial SAA $L$ where $\mbox{Aut\,} L =\{\mbox{id\,}\}$.
\noindent
We are interested in the classification of SAA. Suppose that $V$ is a symplectic vector space with a non-degenerate alternating form $(,)$. Let $x_1, \ldots, x_{2n}$ be a standard basis of $V$. Thus $(x_{2i}, x_{2i-1}) = 1$ but $(x_j, x_i) = 0$ otherwise for any $1 \leq i < j \leq 2n$. The next proposition show that in fact we can turn this into a SAA. 
\begin{Proposition} [\cite {1}] Let $n \geq 2$ and for each $(i, j, k),$ $1 \leq i < j < k \leq 2n$, choose a number $\alpha(i, j, k)$ in the field $\mathbb{F}$. There is a unique SAA of dimension $2n$ over $\mathbb{F}$ satisfying 
\[ (x_i x_j, x_k) = \alpha(i, j, k)\]
for $1 \leq i < j < k \leq 2n$.
\end{Proposition}
\begin{proof}
We only need to define a bilinear alternating product on $V$ that interacts with the alternating form in such a way that $(uv,w) = (vw, u)$ for all $u, v, w \in V$. 
As the product and the alternating form are both bilinear, everything reduces to working with the basis vectors. Now extend the function 
$\alpha$ to all triples of pairwise disjoint numbers $1 \leq i,j,k \leq 2n$ such that $\alpha(i, j,k ) = \alpha(j, k , i) = \alpha(k,i,j)$, $\alpha(j,i,k) = \alpha(k,j,i)=\alpha(i,k,j) = - \alpha(i, j, k)$ and we let
every product $(x_i x_j, x_k ) =0$ if there is a repeated occurrence of a basis vector.
A bilinear alternating product is determined completely from $x_jx_i$, $i< j$.  Now let
$$x_j x_i = -\alpha(j, i, 2) x_1 + \alpha(j, i, 1) x_2 - \cdots - \alpha(j, i, 2n) x_{2n-1} + \alpha(j, i , 2n-1) x_{2n}$$
where $i < j$. We thus have that $V$ is the unique SAA satisfying the stated conditions.
\end{proof}
\noindent
We next get from this some information about the growth of SAA's.
The map $L^3 \to \mathbb{F}$, $(u, v, w) \mapsto (u \cdot v, w)$ is an alternating ternary form and each alternating ternary form on a given symplectic vector space, with a non-degenerate alternating form, defines a unique SAA. Classifying SAA's of dimension $2n$ over a field $\mathbb{F}$ is then equivalent to finding all the ${\mathcal \mbox{Sp}}(V)$-orbits of $(\wedge^3 V)^*$ under the natural action, where $V$ is a symplectic vector space of dimension $2n$ with a non-degenerate alternating form. Suppose that $\mathbb{F}$ is a finite field and suppose that the disjoint ${\mathcal \mbox{Sp}}(V)$-orbits of $(\wedge^3 V)^*$ are ${u}^{{\mathcal \mbox{Sp}}(V)}_1, 
\ldots, {u}^{{\mathcal\mbox{Sp}}(V)}_m$. Then
\[ m \leq |\mathbb{F}|^{{2n \choose 3}} = |(\wedge^{3}V)^*| \leq m|{\mathcal \mbox{Sp}}(V)| \leq m|\mathbb{F}|^{{2n+1 \choose 2}}. \]
It follows that $m = |\mathbb{F}|^{\frac{4n^3}{3} + O(n^2)}$. Because of the sheer
 growth, a general classification of SAA's seems impossible.\\\\
We end this section by looking at all the SAA's of dimension up to $4$.
It is clear that the only SAA of dimension $2$ is the abelian one. Furthermore, it is easily
 seen that apart from the abelian one there is only one SAA of dimension $4$ that can be
 described by the following multiplication table. (see \cite{1}).
\begin{eqnarray*}
x_1x_2 &=& 0\\
y_1y_2 &=& -y_1\\
L: \quad x_1y_1 &=& x_2\\
x_1y_2 &=& - x_1\\
x_2 y_1 &=& 0\\
x_2y_2 &=& 0
\end{eqnarray*} 
The presentation is thus $(x_1y_1, y_2) = 1$. As we said before there is a close connection between SAA's over the field $\mbox{GF}(3)$ of three elements and a certain class of $2$-Engel groups, and in \cite{1} the SAA's over $\mbox{GF}(3)$ of dimension $6$ were classified. There are $31$ such algebras of dimension $6$ of which $15$ are simple. None of the $31$ has a trivial automorphism group. 
We would like to mention here also the work of Atkinson \cite{5} who in his thesis looked at alternating ternary forms over $\mbox{GF}(3)$ in order to study a certain class of groups of exponent $3$.
\section{Nilpotent Symplectic Alternating Algebras}
\begin{defn} 
A symplectic alternating algebra $L$ is \emph{nilpotent} if there exists an ascending chain of ideals $I_0, \ldots, I_n$ such that
\[ \{0\} = I_0 \leq I_1 \leq \cdots \leq I_n = L \]
and $I_sL \leq I_{s-1}$ for $s=1, \ldots, n$. The smallest possible $n$ is then called the \emph{nilpotence class} of $L$.
\end{defn}
\begin{defn} 
More generally, if $I_0 \leq I_1 \leq \cdots \leq I_n$ is any chain of ideals of $L$ then we say that this chain is central in $L$ if $I_s L \leq I_{s-1}$ for $s=1, \ldots, n$.
\end{defn}
\begin{Remark} Equivalently we have that $L$ is nilpotent of class $n \geq 0$ if $n$ is the smallest non-negative integer such that $L^{n+1}=\{0\}$ or equivalently $Z_n(L)=L$. \end{Remark}\noindent
Another interesting property is that any SAA that is abelian-by-nilpotent must be nilpotent.
\noindent
\begin{Proposition}[\cite{2}]
Let $L$ be a SAA. If $L$ is abelian-by-(nilpotent of class $\leq n$) then it is nilpotent of class at most $2n + 1$.
\end{Proposition}
\begin{proof}
Let $I$ be an abelian ideal of $L$ such that $L/I$ is nilpotent of class at most $n$. Then $L^{n+1} \leq I$. But this is equivalent to saying that
$ 0 = (L^{n+1} \cdot L^{n+1}, L) = (L^{n+1}, L \cdot  \underbrace{L \cdots L}_{n+1}) = (L , L^{2n+2})$ and $L$ is nilpotent of class at most $2n+1$.
\end{proof}\noindent
Notice that this result is however not true if we assume that our algebra is nilpotent-by-abelian. The the non-abelian SAA $L$ of dimension $4$ still provides a counter example\\\\
\textbf{Example.} \cite{2}
Consider
$$\begin{array}{ll}
\begin{tabular}{c} 
  \mbox{}  $L:$  \\
\end{tabular} &  
$$\begin{array}{ll}
x_1x_2 = 0 \\
y_1y_2 = - y_1,
\end{array}$$
\end{array}$$
the only nonabelian SAA of dimension $4$ over a field $\mathbb{F}$. Notice that 
\[ Z(L) = \mathbb{F} x_2 ,\  L^2 = Z(L)^\perp = \mathbb{F} x_1 + \mathbb{F} x_2 + \mathbb{F} y_1.\]
Notice that $(L^2)^3=\{0\}$ and thus $L$ is nilpotent-by-abelian. However 
$L$ is not nilpotent as $y_1 y_2^n = (-1)^n y_1 $ for any integer $n \geq 1$.
\chapter{General Theory}
\section{Introduction}
Here we develop a structure theory for nilpotent SAA's. As we said before some general theory was developed in \cite{1} and \cite{2}. We will describe some general results that in particular lead to specific type of presentations that we call later nilpotent presentations. All algebras with a nilpotent presentation are nilpotent and conversely any nilpotent algebra will have a nilpotent presentation. We will also focus on the algebras that are of maximal class and we will see that their structure is very rigid.\\\\
The lack of the Jacobi identity means that many properties that hold for Lie algebras do not hold for SAA's. As the following example shows, it is not true in general that the product of two ideals is an ideal. That example also shows that the formula $L^iL^j \leq L^{i+j}$ does not hold in general.\\\\
\textbf{Example. } 
Consider the $12$-dimensional SAA which has a standard basis $x_1, y_1, x_2, y_2,$
$ x_3, y_3,$
$ x_4, y_4, x_5, y_5, x_6, y_6$ where
\[ (x_3y_5, y_6) = (x_2y_4, y_6) = (x_1y_4, y_5) = (y_1y_2, y_3) = 1 \]
and $(uv, w) = 0$ if $u, v, w$ are basis elements where 
$\{u, v, w\} \not \in \{ \{x_3, y_5, y_6\},$
$ \{x_2, y_4, y_6 \},$
$ \{x_1, y_4, y_5 \},$
$ \{y_1, y_2, y_3 \} \}$.
Notice that this implies that
\begin{alignat*}{3}
    x_3y_5  &= x_6,  &\quad   x_1y_4 &= x_5,    &\quad  y_2y_3 &= x_1,    \\ 
   x_3y_6  &= - x_5, &\quad  x_1y_5 &= -x_4,   &\quad y_4y_5 &= - y_1, \\
   x_2y_4  &= x_6,   &\quad  y_1y_2 &= x_3,    &\quad  y_4y_6 &=-y_2,    \\
   x_2y_6  &= - x_4, &\quad  y_1y_3 &= -x_2,  &\quad   y_5y_6 &= - y_3. 
\end{alignat*}
From this one sees that
\begin{eqnarray*}
L^2 &=&\mathbb{F} x_6+ \mathbb{F} x_5 +\cdots +\mathbb{F} x_1+\mathbb{F} y_1+\mathbb{F} y_2+\mathbb{F} y_3, \\
L^3 &=& \mathbb{F} x_6 + \mathbb{F} x_5 +\cdots  +\mathbb{F} x_1, \\
L^4 &=& \mathbb{F} x_6 + \mathbb{F} x_5 +\mathbb{F} x_4,\\
L^5 &=& 0,\\
L^2L^2 & =& \mathbb{F} x_3 +\mathbb{F} x_2 +\mathbb{F} x_1.
\end{eqnarray*}
In particular $L$ is nilpotent of class $4$, $L^2L^2$ is not an ideal and $L^2L^2 \not \leq L^4$.\\\\
This example indicates that SAA's do differ from Lie algebras. We are going to see in the following sections that there are some shared properties but the next lemma underlines the difference by showing that the two classes of algebras do not have many algebras in common when the characteristic is not $2$. In fact only the SAA's that are obviously Lie algebras are there, namely those of class at most $2$.
\begin{Lemma}\label{lma10}
Let $L$ be a SAA where $\mbox{char\,}{L} \neq 2$ and $L$ is either associative or a Lie algebra. Then $L^3 =\{ 0 \}$.
\end{Lemma}
\begin{proof}
Let us first assume that $L$ is associative. We then have  
\[ 0 = (xyz - x(yz), t) = (x, tzy - t(yz)) = (x, tzy - tyz) \]
for all $x, y, z, t \in L$. It follows that $tzy = tyz = -ytz $ for all $t, z, y \in L$. Using this last property repeatedly we get that
\[ xyz = -zxy = yzx = -xyz\]
and thus $2xyz =0$ for all $x,y,z \in L$. As $char\,L \neq 2$, it follows that $L^3 =0$.\\\\
Now suppose $L$ is a Lie algebra. We then have
\[ 0 = (xyz + yzx + zxy, t) = (x, tzy - t (yz) - tyz) = 2(x, tzy - tyz). \]
As $\mbox{char\,}{ L} \neq 2$, it follows again that $tzy = tyz$ for all $t, z, y \in L$ and this implies again that $L^3 = \{0 \}$.
\end{proof}
\noindent
One handicap that the SAA's have is that when $I$ is an ideal then $L/I$ is in general only an alternating algebra as there is no natural way of inducing an alternating form on this quotient.
For example simply for the reason that the quotient can have odd dimension.
There is however a weaker form of a quotient structure that we can associate to any ideal $I$ of $L$ that works. Thus for any ideal $I$ we have that $(I^\perp + I)/I$ is a well defined SAA with the natural induced multiplication and where the induced alternating form is given by 
$(u + I, v + I) = (u, v) $ for  $u, v \in I^\perp$.
The reader can easily convince himself that this is well defined and that $((I^\perp + I)/I)^\perp = 0$. 
This algebra is also isomorphic to $I^\perp/(I \cap I^\perp)$ that has a similar naturally induced structure as a SAA.
\begin{Remark}
There are some familiar facts for Lie algebras that do not reply on the Jacobi identity and remain true for SAA's. Such properties are particularly useful as we can use them when dealing with quotients $L/I$ where we only know that the resulting algebra is alternating. For example $L^2$ has co-dimension at least $2$ in any nilpotent alternating algebra $L$ of dimension greater than or equal to $2$. From this and the duality given in Lemma \ref{beautiful_relation}, it follows immediately that the dimension of $Z(L)$ is at least $2$ for any non-trivial nilpotent SAA which is something that we will also see later as a corollary of Lemma \ref{lma2.1gth}.
\end{Remark}
\section{Symplectic Alternating Algebras}
\begin{Remark} \label{rem_stressed} 
Let $U, V$ be subspaces of $L$. Notice that 
\[ UV = 0 \Leftrightarrow (UV, L) = 0 \Leftrightarrow (UL, V) = 0 \Leftrightarrow UL \leq V^\perp .\]
In other words we have that $U$ annihilates $V$ if and only if it annihilates $L/V^\perp$. This is useful property that we will be making use of later.
\end{Remark}
\begin{Lemma}\label{lma2.4gth}
Let $I$ and $J$ be ideals in a SAA $L$. We have that $I \cdot L\leq J$ if and only if $I \cdot  J^\perp=\{0\}$. In particular $L^m \cdot Z_m(L)=\{0\}$.
\end{Lemma}
\begin{proof}
From the property given in last remark, we know that $I$ annihilates $L/J$ if and only if $I$ annihilates $J^\perp$. The second part follows this, the fact that $L^m$ annihilates $L/L^{m+1}$, and the fact that $(L^{m+1})^\perp = Z_m(L)$.
\end{proof}
\begin{Remark}
It follows in particular that $I \cdot I^\perp =\{0\}$ for any ideal $I$. In particular any isotropic ideal is abelian. Notice also that the property $L^m Z_m(L)= \{ 0\}$ is equivalent to the fact that $Z_m(L)$ annihilates $L/Z_{m-1}(L)$.
\end{Remark}
\begin{Remark} 
We have seen in the introduction that it is not true in general that $L^iL^j \leq L^{i+j}$. As $(L^m)^\perp = Z_{m-1}(L)$, we however have that
\begin{align*}
L^iL^j \leq L^{i+j} &\Leftrightarrow (L^iL^j, Z_{i+j-1}(L)) = 0 
\Leftrightarrow (L^i Z_{i+j-1}(L), L^j) = 0 \\
&\Leftrightarrow L^i Z_{i+j-1}(L) \leq Z_{j-1}(L).
\end{align*}
The obvious fact that $L^m L \leq L^{m+1}$ thus gives us the interesting fact from last lemma that $L^m Z_m(L)= \{0\}$.
\end{Remark}
\begin{Lemma}\label{lma2.5gth}
Let $I$ be an ideal of $L$. Then $IL \leq I^\perp$ if and only if $I$ is abelian.
\end{Lemma}
\begin{proof}
We have that $I$ annihilates $I$ if and only if $I$ annihilates $L/I^\perp$.
\end{proof}
\begin{Remark}
As $I$ is an ideal we have in fact that $IL \leq I^\perp $ if and only if $IL \leq I \cap I^\perp$. Here $I \cap I^\perp$ is the `isotropic part' of $I$.
\end{Remark}
\begin{Lemma}\label{lma2.8gth}
Let $I,J$ be ideals of a SAA $L$ and let $x \in L$. We have $Jx \leq I $ if and only if $I^\perp x \leq J^\perp $.
\end{Lemma}
\begin{proof}
We have that $Jx \leq I$ is equivalent to $(ux, v)=0$ for all $u \in J$ and $v \in I^\perp$. But this is equivalent to saying that $(vx, u)=0$ for all $v \in I^\perp$ and $u \in J={(J^\perp)}^\perp$ and this is the same as saying that $I^\perp x \leq J^\perp$.
\end{proof}
\noindent
In particular we have that \[\{0\} =  I_0 \leq I_1 \leq \cdots \leq I_m = L\] 
is an ascending central chain if and only if
\[L = I_0^\perp \geq I_1^\perp \geq \cdots \geq I_m^\perp = \{0\}\]
is a descending central chain.
\begin{Proposition}\label{pro2.10.1gth}
Let $L$ be a SAA. No term of the upper central series has co-dimension $1$. Equivalently, no term of the lower central series has dimension $1$.
\end{Proposition}
\begin{proof}
The first fact is a well-known fact about alternating algebras and follows from the fact that if $A$ is an alternating algebra then $A/Z(A)$ cannot be one-dimensional. Now the interesting second statement is a consequence of this and the duality $(L^r)^\perp = Z_{r-1}(L)$.
\end{proof}
\noindent
From the last proposition we know that no term of the lower central series of a SAA can be $1$-dimensional. Next proposition shows that some of terms of the lower central series cannot be $2$-dimensional.
\begin{Proposition}\label{pro2.10gth}
Let $L$ be a SAA we have that $\mbox{dim\,} L^m \neq 2$ for $2 \leq m \leq 4$. Equivalently $Z_m(L)$ is not of co-dimension $2$ if $1 \leq m \leq 3$.
\end{Proposition}
\begin{proof}
We first prove that $\mbox{dim\,} L^2 \neq 2$. 
We argue by contradiction and suppose $\mbox{dim\,} L^2 = 2$. Then
$$2 = \mbox{dim\,} L^2 = \mbox{dim\,} Z(L)^\perp = \mbox{dim\,} L  -\mbox{dim\,} Z(L).$$
Suppose $L = Z(L) +  \mathbb{F} u +  \mathbb{F} v$. Then $L^2 = \mathbb{F} uv$, which contradicts $\mbox{dim\,} L^2 = 2$.\\ \\
Next we turn to showing that $\mbox{dim\,} L^3 \neq 2$. 
We argue by contradiction and let $L$ be a counter example of smallest dimension. We first notice that $Z(L)$ must be isotropic as otherwise $L = I  \obot  I^\perp$ for some $2$-dimensional ideal $I = \mathbb{F} u+ \mathbb{F} v \leq Z(L)$ where  $(u, v) = 1$. But then $M = I^\perp$ is a SAA of smaller dimension where $M^3 = L^3$ is of dimension $2$. This however contradicts the minimality of $L$. 
We can thus assume that $Z(L)$ is isotropic. Notice that 
\[ 2 = \mbox{dim\,} L^3=\mbox{dim\,} Z_2(L)^\perp = \mbox{dim\,} L - \mbox{dim\,} Z_2(L). \]
Say, $L = Z_2(L) + {\mathbb F}x + \mathbb{F}y$. Then $L^2 = Z(L) + \mathbb{F} xy$ and, as $Z(L)$ is isotropic and $xy \in L^2 = Z(L)^\perp$, $L^2$ is isotropic. Thus $L^2 \leq (L^2)^\perp = Z(L)$ and we get the contradiction that $L^3 = \{ 0 \}$.\\ \\
It now only remains to deal with $L^4$. For a contradiction, suppose that $\mbox{dim\,} L^4=2$. Then
\[ 2 = \mbox{dim\,} L^4 = \mbox{dim\,} Z_3(L)^\perp = \mbox{dim\,} L - \mbox{dim\,} Z_3(L). \]
Say $L = Z_3(L) + \mathbb{F}u + \mathbb{F}v$. Then $L^2 \leq Z_2(L) + \mathbb{F} uv$ and using the fact that $Z_2(L) \cdot  L^2 = \{ 0 \}$ we get
\begin{align*}
 L^2 \cdot L^2 \leq (Z_2(L) + \mathbb{F}uv) \cdot L^2  = \mathbb{F} uv \cdot L^2 \leq \mathbb{F} uv \cdot (Z_2(L)  +\mathbb{F}uv ) = \mathbb{F}(uv) (uv) = {0}.
\end{align*}
Thus $0 = (L, L^2 \cdot L^2) \Rightarrow (L^3, L^2) = 0 \Rightarrow (L^4, L) = 0$, that gives us the contradiction that $L^4 = \{ 0 \}$.
\end{proof}
\noindent
\textbf{Example.} Let $L$ be the nilpotent alternating algebra with presentation \\
(We only list the triples that have non-zero value)
\[ {\mathcal P}: \quad (x_2 y_3, y_4)=r,\ (x_1 y_2, y_3)=1,\ (y_1y_2,y_4)=1. \]
Then inspection shows that $\mbox{dim\,} L^5 = 2$. The bound 4 in the last proposition is therefore the best one.
\section{Nilpotent Symplectic Alternating Algebras}
We next see that, like for Lie algebras, all minimal sets of generators have the same number of elements and we can thus introduce the notion of a rank.
\begin{defn}
Let $L$ be a nilpotent SAA. We say that $\{ x_1, \ldots, x_r\}$ is a \textit{minimal set of generators} if these generate $L$ (as an algebra ) and no proper subset generates $L$.
\end{defn}
\begin{Lemma} \label{lma2.1gth}
Let $L$ be a nilpotent SAA. Any minimal set of generators has the same size which is $\mbox{dim\,} L - \mbox{dim\,} L^2$.
\end{Lemma}
\begin{proof}
Let $x_1, \ldots, x_r \in L$ and let $M$ be the subalgebra of $L$ generated by these elements. 
It suffices to show that $L = M$ if and only if $x_1 + L^2, \ldots, x_r + L^2$ generate $L/L^2$ as a vector space. 
Suppose first that $L = M$. Notice that $ M = \mathbb{F} x_1 + \cdots + \mathbb{F} x_r +M \cap L^2$ and thus it is clear that 
$L/L^2$ is generated by $ x_1 + L^2, \ldots, x_r + L^2$ as a vector space. 
Conversely suppose now that the images of $x_1, \ldots, x_r$ in $L/L^2$ generate $L/L^2$ as a vector space. 
An easy induction shows that
\[ L = M + L^{s+1}, \quad L^s = M^s + L^{s+1} \]
for all integers $s \geq 1$. 
If the class of $L$ is $n$, we get in particular that
$L=M + L^{n+1} = M$.% for all $n\geq c$.
\end{proof}
\begin{defn} 
Let $L$ be a nilpotent SAA. The unique smallest number of generators for $L$, as an algebra, is called the \emph{rank} of $L$ and is denoted $r(L)$. 
\end{defn}
\noindent
By last lemma we know that $r(L) = \mbox{dim\,} L - \mbox{dim\,} L^2.$ This has the following curious consequence.
\begin{corrollary}\label{corr16} 
Let $L$ be a nilpotent SAA. We have $r(L) = \mbox{dim\,} Z(L)$. In particular if $L \neq \{0\}$ then $\mbox{dim\,} Z(L) \geq 2$.
\end{corrollary}
\begin{proof}
From Lemma \ref{beautiful_relation} we know that $Z(L) = (L^2)^\perp$. Therefore
\[ r(L)  = \mbox{dim\,} L - \mbox{dim\,} L^2 = \mbox{dim\,}(L^2)^\perp = \mbox{dim\,} Z(L) .\]
Finally, we cannot have $r(L) = 1$ as then we would have that $L$ is one-dimensional. Hence $\mbox{dim\,} Z(L) \geq 2$.
\end{proof}
\begin{Lemma}\label{lma2.3gth} 
Let $I, J$ be ideals of a nilpotent SAA where $I \leq J$. 
If $\mbox{dim\,} J = \mbox{dim\,} I + 1$ then $I \leq J$ is central.
If $I$ is an ideal such that $\mbox{dim\,} I < 2n = \mbox{dim\,} L$ 
then there exists an ideal $J$ such that $\mbox{dim\,} J = \mbox{dim\,} I + 1$.
If furthermore $I$ is an isotropic ideal and $\mbox{dim\,} I < n$ 
then $J$ can be chosen to be isotropic. 
\end{Lemma}
\begin{proof}
Suppose $J=I + {\mathbb F}x$ for some $x \in L$. Let $y \in L$. To show that $I \leq J$ is central, it suffices to show that $x \cdot y \in I $. Suppose that $xy = u_1 + a x$ for $u_1 \in I$ and  $a \in \mathbb{F}$. As  
$I \unlhd L$ it follows by induction that $xy^r = u_r + a^r x$ for some $u_r \in I$. If $L$ is nilpotent of class at most $m$ it follows that $ 0 = u_m + a^m x$ and hence $a=0$.\\ \\
For the latter part suppose first that $I$ is any ideal such that $\mbox{dim\,} I < 2n$. Let $m$ be the largest positive integer such that $L^{m} \not \leq I$. Pick $u \in L^{m}\, \setminus\, I$. Then $J = I + \mathbb{F} u$ 
is the required ideal such that $I \leq J$ is central.
Now suppose furthermore that $I$ is isotropic and that $\mbox{dim\,} I < n$. Then $I^\perp$ is also an ideal and $I < I^\perp$. Let $m$ be the largest non-negative integer such that $I^\perp \underbrace{L \cdots L}_{m} \not \leq I$. Let $u \in I^\perp \underbrace{L \cdots L}_{m} \setminus I$ and again the ideal $J = I + \mathbb{F} u$ is the one required.
\end{proof}
\noindent
\begin{Lemma}\label{lma2.6gth}
Let $L$ be a nilpotent SAA with ideals $I, J$ where $J = I + \mathbb{F}x + \mathbb{F}y$, $(x,y)=1$ and $\mathbb{F}x + \mathbb{F}y \leq I^\perp$. Then $JL \leq I$.
Furthermore if $I$ is isotropic then $J$ is abelian.
\end{Lemma}
\begin{proof}
As $J$ is an ideal of $L$ and as $(xt, x) = 0$ for all $t \in L$ we have that $I + \mathbb{F}x$ is an ideal of $L$. By Lemma \ref{lma2.3gth} we have that $I \leq I + \mathbb{F}x$ is central. 
Similarly $I \leq  I + \mathbb{F}y$ is central and thus $JL \leq I$. 
For the second part notice that if $I$ is isotropic then $I=J \cap J^\perp$ thus $JL \leq I= J \cap J^\perp$ and 
by Lemma \ref{lma2.5gth} it follows that $J$ is abelian.
\end{proof}
\begin{Lemma}\label{lma2.7gth}
Let $L$ be a nilpotent SAA. Every ideal $I$ of dimension $2$ is contained in $Z(L)$. Equivalently, every ideal of codimension $2$ must contain $L^2$.\end{Lemma}
\begin{proof}
The second statement is a trivial fact that holds in all nilpotent alternating algebras.
The first statement is a consequence of this and that duality given by $I \leq Z(L) \Leftrightarrow L^2 = Z(L)^\perp \leq I^\perp$.
\end{proof}
\begin{Remark}\label{rem-nsaa4}
Suppose that $L$ is any nilpotent alternating algebra such that $L/L^2$ is 
$2$-dimensional. Then it follows immediately that the dimension of $L^2/L^3$ is at most $1$ and that the 
dimension of $L^3/L^4$ is at most $2$.
Using this general fact and Proposition \ref{pro2.10.1gth} one can quickly show that all nilpotent SAA's of dimension up to $4$ must be abelian. 
This is clear when the dimension is $2$. Now suppose that $L$ is a nilpotent SAA of dimension $4$. 
We know that $\mbox{dim\,} L/L^2 \geq 2$. 
If $\mbox{dim\,} L^2 = 2$ then by the reasoning above, we would have that $\mbox{dim\,} L^3 = 1$ that contradicts Proposition \ref{pro2.10.1gth}. 
By that proposition we neither can have that $\mbox{dim\,} L^2 = 1$. 
Thus we must have $L^2 = 0$ and $L$ is abelian.
\end{Remark}
\noindent
\textit{\bf A Useful Inequality.  }
Let $L$ be a nilpotent SAA. Then
\begin{align*}
&\mbox{dim\,} Z_i(L)  - \mbox{dim\,} Z_{i-1}(L)  \leq  \frac{1}{2}(\mbox{dim\,} Z_{i-1}(L)-\mbox{dim\,}Z_{i-2}(L))(\mbox{dim\,} Z_{i-1}(L)+\\
&\mbox{dim\,} Z_{i-2}(L)-1).
\end{align*}
\begin{proof}
Let $V_i$ be the complement of $L^{i+1}$ in $L^i$. Then 
\begin{eqnarray*}
L^i &=& V_i \obot L^{i+1}\\
L &=& V_1 \obot V_2 \obot \cdots \obot V_i \obot L^{i+1}\\
U_i &=&V_1 \obot V_2 \obot \cdots \obot V_i.
\end{eqnarray*}
Thus it follows that we have 
\begin{eqnarray*}
\mbox{dim\,} V_i&=& dim\ Z_i(L)-dim\ Z_{i-1}(L)\\
\mbox{dim\,} U_i&=& dim\ (L^{i+1})^\perp = dim\ Z_i(L).
\end{eqnarray*}
Now calculation gives that
\begin{eqnarray*}
V_i \obot L^{i+1}=L^i = L^{i-1}L &=& (V_{i-1} \obot L^i)(U_{i-1} \obot L^i)\\
&=& V_{i-1}U_{i-1} + L^{i+1}
\end{eqnarray*}
and we see that we have 
\begin{align*}
\mbox{dim\,} V_i  \, \leq &   \, \mbox{dim\,} V_{i-1}(U_{i-2} + V_{i-1})\\
\leq & \, (\mbox{dim\,} V_{i-1}) ( \mbox{dim\,} U_{i-2})+{\mbox{dim\,} V_{i-1} \choose 2}.
\end{align*}
It follows that
\begin{align*}
&\mbox{dim\,} Z_i(L)  - \mbox{dim\,} Z_{i-1}(L) \, \leq
\frac{1}{2}(\mbox{dim\,} Z_{i-1}(L)-\mbox{dim\,}Z_{i-2}(L))(\mbox{dim\,} Z_{i-1}(L)+ \\
&\mbox{dim\,} Z_{i-2}(L)-1).
\end{align*}
%\noindent
\end{proof}
\noindent
In particular as $\mbox{dim\,} L^{i} \, -\, \mbox{dim\,} L^{i+1} = \mbox{dim\,} {L^{i+1}}^\perp \, - \,  \mbox{dim\,} {L^i}^\perp = \mbox{dim\,} Z_i(L) \, -\,  \mbox{dim\,} Z_{i-1}(L) $, thus equivalently we have that\\
$$\mbox{dim\,} L^{i+1} \ \geq \ \mbox{dim\,} L^i \, - \, \frac{1}{2}(\mbox{dim\,} Z_{i-1}(L) \, - \, \mbox{dim\,} Z_{i-2}(L))(\mbox{dim\,} Z_{i-1}(L) \,+ \,\mbox{dim\,} Z_{i-2}(L)-1).$$
\noindent
\begin{Theorem}\label{thm2.9gth}
Let $L$ be a nilpotent SAA of dimension $2n \geq 2$. There exists an ascending chain of isotropic ideals
\[\{0\}=I_0 < I_1 < \cdots < I_{n-1} < I_n \]
such that $\mbox{dim\,} I_r = r$ for $r = 0, \ldots, n$. 
Furthermore, for $2n \geq 6$, $I_{n-1}^\perp$ is abelian and the ascending chain 
\[\{0\} < I_2 < I_3 < \cdots < I_{n-1} < I_{n-1}^\perp < I_{n-2}^\perp  < \cdots< I_2^\perp < L \]
is a central chain. 
In particular $L$ is nilpotent of class at most $2n-3$.
\end{Theorem}
\begin{proof}
Starting with the ideal $I_0=\{0\}$, we can apply Lemma \ref{lma2.3gth} iteratively to get the required chain 
\[\{0\}=I_0 < I_1 < \cdots < I_n. \]
By Lemma \ref{lma2.7gth} we have that $I_2 \leq Z(L)$. By this and Lemma \ref{lma2.3gth} we thus have that the chain 
\[I_0 < I_2 < I_3 < \cdots < I_{n-1}\]
is central in $L$. 
By Lemma \ref{lma2.8gth} it follows that the chain 
\[ I_{n-1}^\perp < I_{n-2}^\perp < \cdots < I_2^\perp < I_0^\perp \]
is also central. 
It is only remains to see that $I_{n-1} < I_{n-1}^\perp$ is central and that $I_{n-1}^\perp$ is abelian.
As $I_{n-1}^\perp= I_{n-1} + \mathbb{F}x + \mathbb{F}y$ for some $x, y \in L$ where $(x, y) = 1$ and as $I_{n-1}$ is isotropic, this follows from Lemma \ref{lma2.6gth}.
\end{proof}
\begin{Remark}\label{remlocalclass} 
When $\mbox{dim\,} Z(L) = r < n $, we can choose our chain such that $I_r = Z(L)$. 
We then get a central chain
\[I_0 < I_r < I_{r+1}< \cdots < I_{n-1} < I_{n-1}^\perp < I_{n-2}^\perp < \cdots < I_r^\perp < L .\]
In particular the class is then at most $2n - 3 - 2 (r-2) = 2n - 2 r +1$.
\end{Remark}
\noindent
{\bf Presentations of nilpotent symplectic alternating algebras. }  Last theorem tells us a great deal about the structure of nilpotent SAA's. A moments reflection should convince the reader that 
we can pick a standard basis $x_1, y_1, x_2, y_2,$
$\ldots, x_n, y_n$ such that
\[\mbox{}\mbox{}\mbox{}\mbox{}I_1={\mathbb F}x_n,\  I_2={\mathbb F}x_n + {\mathbb F} x_{n-1},\  \ldots,\  I_n={\mathbb F}x_n+\cdots+{\mathbb F}x_1,\]
\[{I^\perp_{n-1}}=I_n+{\mathbb F}y_1,\  {I^\perp_{n-2}}=I_n+{\mathbb F}y_1+{\mathbb F}y_2,\  \ldots,\  {I^\perp_0}=L=I_n+{\mathbb F}y_1+\cdots+{\mathbb F}y_n.\]
Now let $u, v, w$ be three of the basis elements. Since $I_n$ is abelian we have that $(uv, w) = 0$ whenever two of these three elements are from $\{ x_1, \ldots, x_n\}$. 
The fact that
\[ \{0\} < I_1 < \cdots < I_n \]
is central also implies that $(x_iy_j, y_k) =0 $ if $i \geq k$.
So we only need to consider the possible non-zero triples $(x_iy_j,y_k),\  (y_iy_j,y_k)$ for $1\leq i < j < k \leq n $. 
For each triple $(i, j, k)$ with $1 \leq i < j < k \leq n$, let $\alpha(i, j, k)$ and $\beta(i, j , k)$ be some elements in the field $\mathbb{F}$. 
We refer to the data  
\begin{equation*}\label{eq:nilp_per} 
{\mathcal P}: \quad (x_iy_j,y_k)=\alpha(i, j, k),\quad (y_iy_j,y_k)=\beta(i, j, k), \quad 1 \leq i < j < k \leq n 
\end{equation*}
as a \textit{nilpotent presentation}. We have just seen that every nilpotent SAA has a presentation of this type. 
Conversely, given any nilpotent presentation, let
 \[I_r={\mathbb F} x_n+{\mathbb F} x_{n-1}+\cdots+{\mathbb F} x_{n+1-r}\]
and we get an ascending central chain of isotropic ideals $\{0\}=I_0 < I_1 < \cdots < I_n$ such that $\mbox{dim\,} I_j = j$ for $j = 1, \ldots, n$. 
By Lemma \ref{lma2.8gth} we then get a central chain 
\[\{0\}=I_0 < I_1 < \cdots  < I_n  < I^\perp_{n-1} < I^\perp_{n-2} < \cdots < I^\perp_0=L \]
and thus $L$ is nilpotent. 
Thus every nilpotent presentation describes a nilpotent SAA.
\begin{Remark}
Notice that there are ${2 {n \choose 3}}$ parameters for these presentations. 
If $\mathbb{F}$ is a finite field this thus gives the value $|\mathbb{F}|^{2 {n \choose 3}}$ as the upper bound for the number of $2n$-dimensional nilpotent SAA's  over the field $\mathbb{F}$. 
Armed with this information it is not difficult to get some good information about the growth of nilpotent SAA's over a finite field $\mathbb{F}$. 
Let $V$ be a $2n$-dimensional vector space over $\mathbb{F}$ and consider $(\wedge^3 V)^*$. 
After fixing a standard basis for $V$, each presentation of a SAA corresponds to an element in $(\wedge^3 V)^*$. 
Now let ${\mathcal N}$ be the subset of $(\wedge^3 V)^*$ corresponding to all nilpotent presentations. 
The number of nilpotent SAA's of dimension $2n$ is the same as the number of ${\mathcal \mbox{Sp}}(V)$-orbits of $(\wedge^3 V)^*$ consisting of presentations that give nilpotent algebras. 
Suppose these are ${u}^{{\mathcal \mbox{Sp}}(V)}_i$, $i = 1, \ldots, m$. Then
\[ {\mathcal N} \subseteq  \bigcup_{i=1}^{m} u^{{\mathcal \mbox{Sp}}(V)}_i  \]
and thus $|\mathbb{F}|^{2 {n \choose 3}} = |{\mathcal N}| \leq m \cdot  |{\mathcal \mbox{Sp}}(V)|  \leq m \cdot |\mathbb{F}|^{ {2n+1 \choose 2}}$.
These calculations show that the number of nilpotent SAA's is
\[ m = |\mathbb{F}|^{ n^3/3 + O(n^2)} .\]
\end{Remark}
\section{The structure of nilpotent symplectic alternating algebras of maximal class}
We have seen previously that nilpotent SAA's of dimension $2n$ have class at most $2n-3$. For every algebra of dimension $2n \geq 8$ this bound is attained. 
As well as demonstrating this we will see that the structure of these algebras of maximal class is very restricted. \\ \\
\noindent
Let $L$ be a nilpotent SAA of dimension $2n \geq 8$ with an ascending chain of isotropic ideals
\[ \{0\}=I_0 < I_1 < \cdots < I_n, \]
where $\mbox{dim\,} I_j=j$ for  $j=1, \ldots, n$.
\begin{Theorem}\label{thm3.1gth} 
Suppose $L$ is of maximal class. Then 
\[  I_2=Z_1(L),\  I_3=Z_2(L),\   \ldots,\   I_{n-1}=Z_{n-2}(L), \]
\[ \mbox{} I^\perp_{n-1}=Z_{n-1}(L),\   I^\perp_{n-2}=Z_{n}(L),\   \ldots,\   I^\perp_{2}=Z_{2n-4}(L). \]
Furthermore $Z_0(L), Z_1(L), \ldots, Z_{2n-3}(L)$ are the unique ideals of $L$ of dimensions $0, 2, 3,$
$\ldots, n-1, n+1, n+2, \ldots, 2n-2, 2n$.
\end{Theorem}
\begin{proof}
Let $J_0=\{0\}, J_1=I_2, \ldots, J_{n-2}=I_{n-1}, J_{n-1}=I_{n-1}^\perp, J_n=I_{n-2}^\perp, \ldots, J_{2n-4}=I_2^\perp, J_{2n-3}=L$. 
By Theorem \ref{thm2.9gth}, the chain $J_0 < J_1 < \cdots < J_{2n-3}$ is central.
We argue by contradiction and let $i$ be the smallest integer between $1$ and $2n-4$ where $J_i < Z_i(L)$. 
Let $u \in Z_i(L)\, \setminus\, J_i$ and 
let $k$ be the smallest integer between $i$ and $2n-4$ such that $u \in J_{k+1}$. Then
\[ J_k < J_k + \mathbb{F} u \leq J_{k+1}.\]
If $J_{k+1}/J_k$ has dimension $1$ it follows that $J_{k+1} \leq Z_k(L)$ and we get the contradiction that the class is at most $2n-4$. 
We can thus suppose that $J_{k+1}/J_k$ has dimension $2$ and there are two cases to consider, either $k = n-2$ or $k = 2n-4$. 
In the former case we have
\[ I_{n-1} < I_{n-1} + \mathbb{F}u  \leq  I^\perp _{n-1}\]
which implies that $I = I_{n-1} + \mathbb{F}u$ is an isotropic ideal of maximal dimension $n$. 
As $u \in Z_{n-2}(L)$, we have that $I_{n-2} < I$ is centralised by $L$. 
By Lemma \ref{lma2.8gth} it follows that $I < I^\perp_{n-2}$ is also centralised by $L$ and we get a central series
\[  \{0\}=I_0 < I_2 < I_3< \cdots < I_{n-2} < I < I_{n-2}^\perp < I_{n-3}^\perp < \cdots < I_2^\perp < I_0^\perp=L \]
of length $2n-4$ and we get again the contraction that the class is less than $2n-3$. 
Finally suppose that $k = 2n-4$. So we have 
\[ I^\perp_2 < I^\perp_2 + \mathbb{F}u < L \]
and $u \in Z_{2n-4}(L)$. 
Now let $v \in L \,\setminus\, (I^\perp_2 + \mathbb{F}u)$. Then $L = I_2^\perp + \mathbb{F}u + \mathbb{F}v$ and $L^2 = (I^\perp_2 + \mathbb{F}u)L \leq Z_{2n-5}(L)$. 
Hence $ L \leq Z_{2n-4}(L)$ that again contradicts the assumption that $L$ is of class $2n-3$.\\ \\
We now want to show that these terms of the upper central series are the unique ideals of dimensions $ 0, 2, 3, \ldots, n-1, n+1, n+2, \ldots, 2n-2, 2n$.
First let $I$ be an ideal of dimension $2$. By Lemma \ref{lma2.7gth} we have that $I \leq Z(L)$ and as we have seen that $Z(L)$ has dimension $2$, it follows that $I=Z(L)$. 
Now suppose that for some $ 2 \leq k \leq n-2$ we know that $Z_{k-1}(L)$ is the only ideal of dimension $k$. 
Let $I$ be an ideal of dimension $k+1$. As $L$ is nilpotent 
we have that $I$ contains an ideal $J$ of dimension $k$. By the induction hypothesis we have that $J = Z_{k-1}(L)$ and as $I/J$ is of dimension $1$ we have that $I \leq Z_k(L)$. 
We have that $Z_k(L)$ has dimension $k+1$ and thus $I=Z_k(L)$.
We have thus seen that there are unique ideals of dimensions $0, 2, 3, \ldots, n-1$. 
Now let $I$ be an ideal of dimension $i \in \{ n+1, n+2, \ldots, 2n-2, 2n \}$. 
Then $I^\perp$ is an ideal whose dimension is in $\{0, 2, 3, \ldots, n-1\}$. 
By what we have just seen $I^\perp$ is unique and thus $I$ as well.
\end{proof}
\begin{Remark}
$(1)$ In particular it follows that $Z_k(L)^\perp=Z_{2n-3-k}(L)$ for $0 \leq k \leq 2n-3$.\\ \\
$(2)$ As $L^{k}=Z_{k-1}(L)^\perp$, it follows that $L, L^2, \ldots, L^{2n-2}$ are the unique ideals of dimensions $2n, 2n-2, 2n-3, \ldots, n+1, n-1, n-2, \ldots, 2, 0$. Also
\[L^k=Z_{k-1}(L)^\perp=Z_{2n-k-2}(L).\]
\end{Remark}
\begin{Remark}
Let $L$ be any nilpotent SAA of dimension $2n \geq 6$ with the property that $\mbox{dim\,} Z(L)=2$. Notice that $Z(L)$ must be isotropic since otherwise we would have a $2$-dimensional symplectic subalgebra $I$ within $Z(L)$ and we would get a direct sum $I \operp I^\perp$ of two SAA's. As $I^\perp$ has non-trivial center this would contradict the assumption that $Z(L)$ is $2$-dimensional. Now $L$ has rank $2$. Suppose it is generated by $x, y$. Then $L$ is generated by $x, y, xy$ modulo $L^3$ and thus $ \mbox{dim\,} Z_2(L) = \mbox{dim\,} (L^3)^\perp = \mbox{dim\,} L - \mbox{dim\,} L^3 = 3.$
\end{Remark}
\noindent
{\bf The complete list of ideals of $L$.}  
We have seen that there is a unique ideal of dimension $k$ for any $0 \leq k \leq 2n$ apart from $k=1, k=n$ and $k=2n-1$. 
Let us now turn to the remaining dimensions. 
Now every ideal of dimension $1$ is contained in $Z(L)$ and conversely every subspace of dimension $1$ in $Z(L)$ is an ideal.\\ \\
Next consider an ideal $I$ of dimension $2n-1$. Then $I^\perp$ is an ideal of dimension $1$ and is thus any subspace of dimension 1 such that
 \[ \{0\} < I^\perp < Z(L) \]
Equivalently, $I$ is any subspace of dimension $2n-1$ such that
\[ L^2=Z(L)^\perp < I < \{0\}^\perp=L .\]
Finally consider an ideal $I$ of dimension $n$. Since $L$ is nilpotent there exists an ideal $J$ of dimension $n+1$ containing $I$. 
By last theorem we have that $J=L^{n-1}=Z_{n-2}(L)^\perp$. Also %
$I$ contains an ideal of dimension $n-1$  that we know is $Z_{n-2}(L)$. Thus
\[ Z_{n-2}(L) < I < Z_{n-2}(L)^\perp . \]
We also know from our previous work that $Z_{n-2}(L)$ is an isotropic ideal of dimension $n-1$. $I$ is thus an isotropic ideal of the form 
\[ Z_{n-2}(L) + {\mathbb F} u \]
For some $u \in Z_{n-2}(L)^\perp \ \setminus \ Z_{n-2}(L)$. 
Conversely, as $Z_{n-2}(L)^\perp L \leq Z_{n-2}(L)$ we have that for any intermediate subspace $I$ of dimension $n$ between $Z_{n-2}(L)$ and $Z_{n-2}(L)^\perp$, $I$ is an ideal.\\ \\
We thus have a complete picture of the ideals of $L$.\\ \\
We now focus on the characteristic ideals. It turns out that there are as well always characteristic ideals of dimension $1, n$ and $2n-1$ when $2n \geq 10$.
\begin{Remark} 
Notice that if $I$ is a characteristic ideal then the ideal $I^\perp$ is also characteristic. 
To see this let $\phi$ be any automorphism of the SAA $L$ and let $a \in I^\perp$. 
As $\phi$ is an automorphism we have that $\phi(a) \in \phi(I)^\perp = I^\perp$.
\end{Remark}
\begin{Theorem}\label{thm3.2gth}
Let $L$ be a nilpotent SAA of dimension $2n \geq 10$ that is of maximal class. $L$ has a chain of characteristic ideals
\[\{0\}=I_0 < I_1 < \cdots < I_{n} < I^\perp_{n-1} <  \cdots < I^\perp_1 < I^\perp_0=L \]
where for $0 \leq k \leq n$, $I_k$ is isotropic of dimension $k$.
\end{Theorem}
\begin{proof}
By Theorems \ref{thm2.9gth} and \ref{thm3.1gth}, we know that we can get such a chain of ideals where all the ideals apart from $I_1, I_n$ and $I_{2n-1}$ are characteristic. 
We want to show that we can choose our chain such that $I_1, I_n$ and $I_{2n-1}$ are also characteristic. 
Let $x_1, y_1, \ldots, x_n, y_n$ be a standard basis such that
\[ I_k = \mathbb{F} x_n + \mathbb{F} x_{n-1} + \cdots + \mathbb{F} x_{n+1-k}\]
for $1 \leq k \leq n$. Then $I_4 I_2^\perp= \mathbb{F} x_{n-3}y_{n-2}$ is a characteristic ideal. 
We claim that this is non-trivial. Otherwise $x_{n-3}y_{n-2}=0$ and then $(x_{n-3} u, y_{n-2}) = 0$ for all $u \in L$ that implies that $x_{n-3}L \leq {\mathbb F} x_n+{\mathbb F} x_{n-1}$ 
and we get the contradiction that $x_{n-3} \in Z_2(L)=I_3$. 
Thus we have got a characteristic ideal of dimension $1$, namely $I_4I_2^\perp=Z_3(L) \cdot L^2$. 
Notice that we are assuming here that $n \geq 5$. From this we get that $(I_4I_2^\perp)^\perp$ is a characteristic ideal of dimension $2n-1$.\\ \\
It remains to find a characteristic ideal of dimension $n$.
We know that 
$L^n = \mathbb{F} x_n + {\mathbb F} x_{n-1}+\cdots + {\mathbb F} x_2,\ L^{n-1} = I^\perp_{n-1} = {\mathbb F} x_n + \cdots + {\mathbb F} x_1 + {\mathbb F} y_1$ 
and 
$L^{n-2}=I^\perp_{n-2}=I^\perp_{n-1}+{\mathbb F} y_2$. 
As $L^{n-1}=L^{n-2} \cdot L$ it follows that $L^n + {\mathbb F} x_1+{\mathbb F} y_1 = (L^{n-1} + {\mathbb F} y_2)L$ and thus 
\[ L^n + {\mathbb F} x_1 + {\mathbb F} y_1 = L^n + y_2 L.\]
Thus there exist $u, v \in L$ such that $y_2u + L^n = x_1 + L^n$ and $y_2v + L^n = y_1 + L^n$. Then
\[ (y_2 u, x_1) =0,\ (y_2 u, y_1) \neq 0,\ (y_2 v, x_1) \neq 0,\ (y_2 v, y_1) =0.\] 
Equivalently
\[ (x_1y_2, u) =0,\ (x_1y_2, v) \neq 0,\ (y_1y_2, u) \neq 0,\ (y_1y_2, v) =0\]
and this implies that $x_1y_2, y_1y_2$ are linearly independent (something that will also be useful later). 
Consider next the $2$-dimensional characteristic subspace 
\[ L^{n-1}L^{n-2} = {\mathbb F} x_1y_2 +{\mathbb F} y_1y_2.\]
Notice that $L^{n-1}L^{n-2} \leq I_{n-2}$. Let $k$ be the smallest positive integer between $1$ and $n-3$ such that $L^{n-1}L^{n-2} \leq I_{k+1}$. 
Let $J=L^{n-1} L^{n-2} \cap I_k$. Then $\mbox{dim\,} J=1$ and there is a unique one-dimensional subspace 
$\mathbb{F}u$ of ${\mathbb F} x_1+ {\mathbb F} y_1$ such that ${\mathbb F} u {L^{n-2}}=J$. 
Now $I= I_{n-1} + {\mathbb F} u$ is the characteristic ideal of dimension $n$ that we wanted. Notice that $I= \{ x \in I_{n+1}: x I_{n+2} \leq  J \}$.
\end{proof}
\begin{Remark}\label{rem-char-in-dim8}
If $L$ is a nilpotent SAA of dimension $8$ that is of maximal class then there is no characteristic ideal of 
dimension $1$. The reader can convince himself of this by looking at the classification of these algebras 
given in the next chapter.
\end{Remark}
\noindent
\begin{corrollary}\label{corr3.3gth}
 Let $L$ be a nilpotent SAA of maximal class and dimension $2n \geq 10$. The automorphism group of $L$ is nilpotent-by-abelian.
 \end{corrollary}
\begin{proof}
Consider a chain of characteristic ideals as given in the last theorem %\ref{thm3.2gth}
\[\{0\}=I_0 < I_1 < \cdots < I_{n} < I_{n-1}^\perp <   I_{n-2}^\perp < \cdots < I_0^\perp=L .\]
Consider the ordered basis $(x_n, x_{n-1}, \ldots, x_1, y_1, \ldots, y_n)$ associated with this chain, that is 
$I_k = \mathbb{F}x_n + \mathbb{F} x_{n-1} + \cdots + \mathbb{F} x_{n+1-k}$. 
As the ideals in the chain are all characteristic we see that the matrix of any automorphism with respect to that ordered basis will be upper triangular. The result follows.
\end{proof}
\noindent
We next move on to presentations of nilpotent SAA's of maximal class. Suppose $L$ is any nilpotent SAA with a presentation
\[{\mathcal P}: \quad (x_iy_j,y_k)=\alpha_{i j k},\quad (y_iy_j,y_k)=\beta_{i j k}, \quad 1 \leq i < j < k \leq n .\]
We would like to read from the presentation whether the algebra is of maximal class. This turns out to be possible.
\begin{Theorem}\label{thm3.4gth}
Let $L$ be a nilpotent SAA of dimension $2n \geq 8$ given by some nilpotent presentation ${\mathcal P}$. The algebra is of maximal class 
if and only if 
$x_{n-2}y_{n-1},\ $
$x_{n-3}y_{n-2},\  \ldots$
$,\ x_2y_3$ are non-zero 
and 
$x_1y_2, y_1y_2$ are linearly independent.
\end{Theorem}
\begin{proof} 
Let us first see that these conditions are necessary. Suppose that $L$ is of maximal class. In the proof of Theorem \ref{thm3.2gth} we have already seen that $x_1y_2$ and $y_1y_2$ must be linearly independent. As before we let $I_k = {\mathbb F} x_n+ \cdots + {\mathbb F} x_{n+1-k}$. As $x_{n-2} \not \in Z(L)$, we have $(x_{n-2} y_{n-1}, y_n) \neq 0$ and thus $x_{n-2}y_{n-1} \neq 0$. As the terms of the central chain 
\[ I_0 < I_2 < I_3 < \cdots < I_{n-1} < I^\perp_{n-1} < I^\perp_{n-2} < \cdots < I^\perp_2 < I^\perp_0 \]
are the terms of the lower central series, we know that $I_{k+1}L=I_k$ for $2 \leq k \leq n-2$. Thus we have for $3 \leq k \leq n-2$ that
\[ I_{k-1}+{\mathbb F} x_{n-k+1} = (I_k + {\mathbb F} x_{n-k})L.\]
From this it follows $I_{k-1}+ {\mathbb F} x_{n-k+1} = I_{k-1}+ x_{n-k}L$. In particular there exist $u \in L$ such that 
		$I_{k-1}+ x_{n-k+1}= I_{k-1}+ x_{n-k}u$. 
It follows that $0 \neq ( x_{n-k} u, y_{n-k+1}) = - (x_{n-k} y_{n-k+1}, u)$. 
Hence $ x_{n-k} y_{n-k+1}$ is non-zero for $ 3 \leq k \leq n-2$.\\\\
Let us then see that these conditions are sufficient. We do this by showing that $I_2 = I_3L,\ I_3=I_4L,\ \ldots ,\ I_{n-2} = I_{n-1}L,\ I_{n-1}={I^\perp_{n-1}} L,\ {I^\perp_{n-1}} = {I^\perp_{n-2}} L,\ 
\ldots,\ {I^\perp_3} = {I^\perp_2} L,\ {I^\perp_2} = {I^\perp_0} L$. This is sufficient as this would imply that $L^{2n-3} = I_2 \neq 0$ and thus $L$ is nilpotent of class $2n-3$. Firstly as $x_{n-2}y_{n-1} \neq 0$ we have that $(x_{n-2}y_{n-1}, y_n) \neq 0$ and thus $I_3L = x_{n-2}L = {\mathbb F} x_n + {\mathbb F} x_{n-1} = I_2$. Now suppose we have already established that $I_k = I_{k+1}L$ for all $2 \leq k \leq m$ where $2 \leq m \leq n-3$. Then
\[I_{m+2}L = (I_{m+1}+ {\mathbb F} x_{n-m-1})L = I_m + x_{n-m-1} \cdot L \]
As $x_{n-m-1}y_{n-m} \neq 0$ we have $(x_{n-m-1}u, y_{n-m}) = - (x_{n-m-1}y_{n-m}, u) \neq 0$ for some $u \in L$ and thus 
$I_{m+2} L = I_m + x_{n-m-1}L = I_m +{\mathbb F} x_{n-m} = I_{m+1}$.
We have thus established by induction that
\[I_2 = I_3L,\  \ldots,\  I_{n-2} = I_{n-1}L.\]
We next show that ${I^\perp_{n-1}} L =I_{n-1}$. As $x_1y_2 \neq 0$ we have that there exist $u \in L$ such that $0 \neq (x_1y_2, u) = - (x_1u, y_2)$ and 
\[{I^\perp_{n-1}} L = (I_{n-1} + {\mathbb F} x_1+ {\mathbb F} y_1) L = I_{n-2}+{\mathbb F} x_2 = I_{n-1}.\]
Next we show that ${I^\perp_{n-2}} L =I^\perp_{n-1}$. As $x_1y_2, y_1y_2$ are linearly independent there exist $u, v \in L$ such that
\[ (x_1y_2, u) =0,\ (x_1y_2, v) \neq 0,\ (y_1y_2, u) \neq 0,\ (y_1y_2, v) =0\]
and thus 
\[ (y_2 u, x_1) =0,\ (y_2 u, y_1) \neq 0,\ (y_2 v, x_1) \neq 0,\ (y_2 v, y_1) =0.\]
Hence
\[{I^\perp_{n-2}} L = ({I^\perp_{n-1}}+ {\mathbb F}y_2)L = I_{n-1} + {\mathbb F} y_2 L= I_{n-1} + {\mathbb F}x_1 + {\mathbb F}y_1 = {I^\perp_{n-1}}.\]
Now suppose that we have established that ${I}^\perp_{k-1} L ={I}^\perp_k$ for $m+1 \leq k \leq n-1$ where $ 3 \leq m \leq n-2$. As $x_{n-m}y_{n-m+1} \neq 0$ it follows that
 there exists $ u \in L$ such that $ 0 \neq ( x_{n-m}y_{n-m+1}, u)= ( y_{n-m+1} u, x_{n-m})$. Thus
\[ {I^\perp_{m-1}} L = ( {I^\perp _{m}}+ y_{n-m+1})L  = {I^\perp_{m+1}} + y_{n-m+1}L = {I^\perp_{m+1}} + {\mathbb F} y_{n-m} = {I^\perp_{m}} .\]
It now only remains to see that ${I^\perp_0} L = {I^\perp_{2}}$. But this follows from $x_{n-2}y_{n-1} \neq 0$ that implies that $(y_{n-1}y_n, x_{n-2})= ( x_{n-2} y_{n-1}, y_n) \neq 0$. Thus
\[ {I^\perp_0} L = ( {I^\perp_{2}} + {\mathbb F} y_{n-1}+ {\mathbb F} y_n) L = {I^\perp_{3}} + ( {\mathbb F} y_{n-1}+ {\mathbb F} y_n) L = {I^\perp_{3}} + {\mathbb F} y_{n-2} = {I^\perp_{2}} .\]
This finishes the proof.
\end{proof}
\begin{Remark}
In particular it follows that for each $2n \geq 8$ there exist a nilpotent symplectic alternating algebra of maximal class. One just needs to choose the presentation such that the conditions from Theorem \ref{thm3.4gth} hold. One possibility is
\begin{align*}\label{maxeg}
 {\mathcal P}: \ (x_{n-2}y_{n-1},y_n)&=-1, (x_{n-3}y_{n-2},y_n)=-1,\ldots, (x_2y_3, y_n)=-1,
(x_1y_2, y_n)=-1,  \\ (y_1y_2, y_{n-1})&=-1. 
\end{align*}
In fact the conditions are not a strong constraint. 
In particular the values of $(x_iy_j, y_k),$
$\ (y_i y_j, y_k)$ where $j - i \geq 2$ can be chosen freely. 
The number of such triples is $2 {n-1 \choose 3}$ that is a polynomial in $n$ of degree $3$ with leading coefficient $1/3$.
Let ${\mathbb F}$ be any finite field. By a similar argument as we used for determining the growth of nilpotent SAA's we see that the number $m(n)$ of nilpotent SAA's of maximal class satisfies
\[ m(n) = |{\mathbb F}|^{n^3/3 + O(n^2)} .\] 
\end{Remark}
\begin{Remark}
$(1)$ Let $L$ be a nilpotent SAA of dimension $2n \geq 10$ that is of maximal class and consider a chain $\{0\}=I_0 < \cdots < I_n$ of characteristic ideals where 
$I_k$ is of dimension $k$. We have, for $4 \leq m \leq n-1$,
\[ I_m I^\perp_{m-2} = \mathbb{F} x_{n+1-m}y_{n+2-m} \]
and thus we get that $\mathbb{F} x_2y_3, \mathbb{F} x_3y_4, \ldots, \mathbb{F} x_{n-3}y_{n-2}$ are one-dimensional characteristic subspaces of $L$. Also,
\[ I_{n-1}^\perp I_{n-2}^\perp = \mathbb{F} x_1y_2+\mathbb{F} y_1y_2 \]
is a characteristic subspace. So is
$ I_n^\perp I_{n-2}^\perp = \mathbb{F} x_1y_2 .$\\\\
$(2)$ If $V$ is a characteristic subspace of dimension $d$ then we get a chain of characteristic subspaces
\[ V \cap I_1 \leq V \cap I_2 \leq \cdots \leq V \cap I_n \leq V \cap I^\perp_{n-1} \leq \cdots \leq V \cap I^\perp_0 =V .\]
Thus there is a chain of characteristic subspaces $V_1 < V_2 < \cdots < V_d$ where $V_i$ is of dimension $i$.
\end{Remark}
\chapter{Nilpotent Symplectic Alternating Algebra of dimension $2 \lowercase{n} \leq 8$}
The classification of the nilpotent SAA's of dimensions at most $8$ is implicit in \cite{2} although this is not done explicitly and the context there is a more general setting.
To demonstrate the machinery that we have developed 
we will offer a much shorter approach here. Throughout this chapter we will be working with an arbitrary field $\mathbb{F}$.\\\\
We have observed earlier that nilpotent SAA's of dimensions $2$ or $4$ must be abelian.
\section{Algebras of dimension $6$}
Let $L$ be a non-abelian nilpotent SAA of dimension $6$ with a nilpotent presentation ${\mathcal P}$. There are at most two non-zero triple values
\[ (x_1y_2, y_3) = a,\ (y_1y_2, y_3) = b. \]
As $L$ is non-abelian, one of these must be non-zero and, by replacing 
$x_1, y_1$ by $-y_1, x_1$ if necessary,
we can assume that $b \neq 0$.
Replacing then $x_3, y_3$ by $b x_3, \frac{1}{b} y_3$ implies that we can further assume that $(y_1y_2, y_3) = 1$. Finally replacing $x_1, y_1$ by $x_1 - a y_1, y_1$ and we can also assume that $(x_1y_2, y_3) = 0$. 
Apart from the abelian algebra, there is thus only one algebra of dimension $6$ with presentation
\[ {\mathcal P}_{6}^{(3,1)}:\ \  (y_1y_2, y_3) = 1. \]
(We will normally only write down those triples where the value is non-zero).
\section{Algebras of dimension $8$}
First suppose that $Z(L)$ is not isotropic. We can then choose our standard basis such that $I = \mathbb{F} x_4 + \mathbb{F} y_4 \leq Z(L)$ and we get a direct sum
$ I \operp I^\perp$
of SAA's of dimensions $2$ and $6$. From $4.1$ we then know that apart from the abelian algebra, there is only one such algebra 
$L =  \mathbb{F} x_4 +  \mathbb{F} x_3 +  \mathbb{F} x_2 +  \mathbb{F} x_1 +  \mathbb{F} y_1 +  \mathbb{F} y_2 +  \mathbb{F} y_3 +  \mathbb{F} y_4$ with presentation
\[ {\mathcal P}_{8}^{(5,1)}: \ \  (y_1y_2, y_3) = 1. \]
We then turn to the situation where $Z(L)$ is isotropic. Let us first see that $\mbox{dim\,} Z(L) \neq 4$. We argue by a contradiction and suppose that $\mbox{dim\,} Z(L) = 4$. Pick a standard basis such that 
$ Z(L) = \mathbb{F}x_4 + \mathbb{F}x_3 + \mathbb{F}x_2 + \mathbb{F}x_1$. Now $L$ is not abelian and thus $(y_iy_j, y_k) \neq 0$ for some $ 1 \leq i < j < k \leq 4$. Without loss of generality, we can suppose that $(y_1y_2, y_3) = 1$. Suppose now that $(y_1y_2,y_4) = a,\, (y_2y_3,y_4) = b$ and $(y_3y_1,y_4) = c.$
Let $\bar{y_4} = y_4  - b y_1 - cy_2 - a y_3 $. 
Inspection shows that $\bar{y_4}$ is orthogonal to $L^2 = \mathbb{F} y_1y_2 + \mathbb{F} y_2y_3 + \mathbb{F} y_3y_1 + \mathbb{F} \bar{y_4}y_1 + \mathbb{F} \bar{y_4}y_2 + \mathbb{F} \bar{y_4}y_3$. Thus $\bar{y_4} \in (L^2)^\perp = Z(L)$ and we get the contradiction that $\mbox{dim\,} Z(L) \geq 5$.
Thus we have shown that $\mbox{dim\,} Z(L) \neq 4$ and as $\mbox{dim\,} Z(L)$ is always at least 2, we have two cases to consider: $\mbox{dim\,} Z(L) = 3$ and $\mbox{dim\,} Z(L) = 2$.
\subsection{Algebras with an isotropic center of dimension $3$} 
We can choose the standard basis such that
\begin{eqnarray*}
Z(L) & =& \mathbb{F} x_4 + \mathbb{F} x_3 + \mathbb{F} x_2 \\
L^2 = Z(L)^\perp &=& \mathbb{F} x_4 + \mathbb{F} x_3 + \mathbb{F} x_2 + \mathbb{F} x_1 + \mathbb{F} y_1.
\end{eqnarray*}
By Theorem \ref{thm2.9gth}, we know that $L^3 = L^2L \leq Z(L)$ and by Propositions \ref{pro2.10.1gth} and \ref{pro2.10gth} we must then have $L^3 = Z(L)$. 
As $x_1 \not \in Z(L)$, we must have $(x_1 y_i, y_j) \neq 0$ for some $2 \leq i < j \leq 4$. Without loss of generality $(x_1 y_2, y_3) \neq 0$. 
By replacing $y_4, y_1$ by $y_4 + a x_1, y_1 + a x_4$ for a suitable $a$, we can assume that $(y_2y_4, y_3) = 0$. 
Let $ V = \mathbb{F} y_2 + \mathbb{F} y_3 +  \mathbb{F} y_4$. 
Now $(y_2y_3, y_4) = 0$ and $L^2 = Z(L) + \mathbb{F} x_1 + \mathbb{F} y_1$. 
As $L = L^2 + V$ it follows that $V^2 = \mathbb{F} x_1 + \mathbb{F} y_1$ and as $V^2$ is not isotropic we must have that some two of $y_2y_3, y_4y_3, y_2y_4$ are not isotropic. 
Without loss of generality we can suppose that these are $y_2y_3$ and $y_4y_3$. By replacing $y_4, x_4$ by $a y_4, \frac{1}{a} x_4$ for a suitable $a \in \mathbb{F}$, we can furthermore assume that $(y_2y_3, y_4y_3) = 1$. Thus
\[ \mathbb{F} x_1 + \mathbb{F} y_1 = V^2 = \mathbb{F} y_2y_3 + \mathbb{F} y_4y_3  \]
and $y_2y_4 = a y_2y_3 + b y_4y_3$ for some $a, b \in \mathbb{F}$. It follows that $(y_2 + b y_3) (y_4 - ay_3) = 0$. Now replace $y_2, y_4, x_3$ by $y_2 + b y_3, y_4 - a y_3, x_3 - b x_2 + a x_4$ and then replace $x_1, y_1$ by $y_2y_3, y_4y_3$. It follows that we get a new standard basis where
\[y_2y_3 = x_1,\  y_4y_3 = y_1,\ y_2y_4 = 0 .\]
This implies that the only non-zero triples are $(y_1y_2, y_3) = 1$ and $(x_1y_3, y_4) = 1$. There is thus only one possible candidate here, the algebras $L$ with presentation
\begin{align*}
{\mathcal P}_{8}^{(3,2)}: \ \  (y_1y_2, y_3) = 1,\  (x_1y_3, y_4) = 1.
\end{align*}
Conversely, one sees by inspection that 
$Z(L) = \mathbb{F} x_4 + \mathbb{F} x_3 + \mathbb{F} x_2$
and this candidate is a genuine example with $\mbox{dim\,} Z(L)=3$. 
\subsection{Algebras with an isotropic center of dimension $2$} 
We know that the class of $L$ is at most $2 \cdot 4 - 3= 5$ and thus $L^5 \leq Z(L)$. Let $k$ be the smallest positive integer $2 \leq k \leq 5$ such that $ L^k \leq Z(L)$. As $\mbox{dim\,} L^k \leq 2$, it follows from Proposition \ref{pro2.10gth} that $k=5$. Hence $L$ is of maximal class and by Theorem \ref{thm3.1gth} 
can choose our standard basis such that, we get ideals $I_k = Fx_4 +\cdots + Fx_{4+1-k}$,
$k=0, \ldots, 4$ where
\[ \{ 0 \} = I_0 < I_1 < \cdots < I_4 = I^\perp_4 < I^\perp_3 < \cdots < I^\perp_0 = L \]
is a central series with $I_2 = Z(L) = L^5,\ I_3 =Z_2(L) = L^4,\ I^\perp_3 = Z_3(L) = L^3$ and $I^\perp_2 = Z_4(L) = L^2$. Thus we have
%
%
%&&&&&&&&&&&&&&&&&&&&&&&& Pictuer &&&&&&&&&&&&&&&&&&&&&&&&&&&&&&
\begin{alignat*}{2}
\color{cyan}L^5=Z(L)\ &\boxed{
\begin{matrix}
\color{green} x_4\\\color{green} x_3
\end{matrix}
}\ \ 
\begin{matrix}
y_4\\y_3
\end{matrix}\\
\color{cyan}L^4=Z_2(L)\ &\boxed{
\begin{matrix}
\color{magenta} x_2
\end{matrix}} 
\boxed{
\begin{matrix}
\color{magenta} y_2
\end{matrix}
} \ \color{cyan}L^2=Z_4(L)\\
&\boxed{
\begin{matrix}
\color{blue} x_1
\end{matrix}}
\boxed{
\begin{matrix}
\color{blue} y_1
\end{matrix}
}\ \color{cyan}L^3=Z_3(L)
\\
\end{alignat*}
%&&&&&&&&&&&&&&&&&&&&&&&&&&&&&&&&&&&&&&&&&&&&&&&&&&&&&
%
%
By Theorem \ref{thm3.4gth} we furthermore have that $x_1y_2, y_1y_2$ are linearly independent and thus a basis for $Z(L) = \mathbb{F} x_4 + \mathbb{F} x_3$. We can now pick our standard basis such that 
$ x_1y_2 = x_4$ and $y_1y_2 = x_3$.
As $x_2 \not \in Z(L)$, we also have that $(x_2 y_3, y_4) \neq 0$. This means that we have the non-zero triples $(x_1y_2, y_4) = 1,\ (y_1y_2, y_3) = 1$ and $(x_2 y_3, y_4) = r \neq 0$. 
The only remaining triples that are possibly non-zero are $(x_1 y_3, y_4), (y_1y_3, y_4)$ and $(y_2y_3, y_4)$. 
Replacing $x_1, y_1$ and $y_2$ by $x_1 - a x_2, y_1 - b x_2$ and $y_2 + a y_1 - b x_1 - c x_2$ for suitable $a, b, c \in \mathbb{F}$, we can assume that these extra triples are zero. 
We can thus choose our basis so that our algebra $L$ has presentation
\begin{align} \label{eq:dim8max}
  {\mathcal P}_{8}^{(2,3)}(r): \ \  (x_2 y_3, y_4) = r,\  (x_1y_2, y_4) = 1,\  (y_1y_2, y_3)=1. 
\end{align}
We finally need to sort out when, for $r, s \in \mathbb{F}^* = \mathbb{F} \, \setminus \, \{0\}$, 
the two presentations ${\mathcal P}_{8}^{(2,3)}(r)$ and ${\mathcal P}_{8}^{(2,3)}(s)$
describe the same algebra $L$. 
We will see that this happens if and only if $s/r \in (\mathbb{F}^*)^3$.
To see that this is a sufficient condition, suppose we have an algebra $L$ that has a presentation ${\mathcal P}_{8}^{(2,3)}(r)$ with respect to some standard basis $x_1, y_1, x_2, y_2, x_3, y_3, x_4, y_4$. 
Suppose that $s = a^3 r$ for some $a \in \mathbb{F}^*$. Let 
$\tilde{x_1} = x_1,\ \tilde{y_1} = y_1,\ \tilde{x_2} = a x_2,\ \tilde{y_2} = \frac{1}{a} y_2,\ \tilde{x_3} = \frac{1}{a} x_3,\ \tilde{y_3} = ay_3,\ \tilde{x_4} = \frac{1}{a} x_4$ and $\tilde{y_4} = a y_4$. 
Inspection shows that $L$ has presentation ${\mathcal P}_{8}^{(2,3)}(s)$ with respect to the new basis.
It remains to see that the condition is also necessary. 
Consider the algebra $L$ with presentation ${\mathcal P}_{8}^{(2,3)}(r)$ and take an arbitrary new standard basis $\tilde{x_1}, \tilde{y_1}, \tilde{x_2}, \tilde{y_2}, \tilde{x_3}, \tilde{y_3}, \tilde{x_4}, \tilde{y_4}$ 
such that $L$ satisfies the presentation ${\mathcal P}_{8}^{(2,3)}(s)$ for some $s \in \mathbb{F}^*$. 
We want to show that $s/r \in (\mathbb{F}^*)^3$. Now
\begin{align*}
 \tilde{y_{3}} &= ay_{3}+by_{4}+u,\\
\tilde{y_{4}} &= cy_{3}+dy_{4}+v ,
\end{align*}
for some $u,v \in L^2$ and $a, b, c, d \in \mathbb{F}$ where $ad-bc \neq 0$. As $\mbox{dim\,} L^2 - \mbox{dim\,} L^3 =1$ it follows readily that $L^2L^2 \leq L^4$ and it follows that
\begin{align*}
\tilde{y_{3}}{\tilde{y_{4}}} \tilde{y_{3}} & = (ay_{3}+by_{4}) (cy_{3}+dy_{4} ) (ay_{3}+by_{4}) + w,\\
\tilde{y_{3}}{\tilde{y_{4}}}{\tilde{y_{4}}} & = (ay_{3}+by_{4}) (cy_{3}+dy_{4} ) (cy_{3}+dy_{4} ) + z,
\end{align*}
for some $w, z \in L^4$. As $L^6 = 0$ we have that $L^4$ is orthogonal to $L^3$ and thus 
in the following direct calculations we can omit $w$ and $z$. We have
\[ -s^2 = (\tilde{y_{3}}{\tilde{y_{4}}} \tilde{y_{3}}, \tilde{y_{3}}{\tilde{y_{4}}}{\tilde{y_{4}}}) = - (ad - bc)^3 r^2 .\]
Hence $s/r \in (\mathbb{F}^*)^3$.
%
%
%*******************************                                   END.                                          *************************************
%+++-------------------------------------+++---------------------------------------+++-------------------------------------+++-----------------------------------
%+++-------------------------------------+++---------------------------------------+++-------------------------------------+++-----------------------------------

% Part II

\part{The Classification of nilpotent SAA of dimension $10$}
\noindent

\thispagestyle{empty}
\chapter*{INTRODUCTION}
\vspace{-2.1cm}
\hspace{500cm}
\noindent
\noindent

This part will be mostly about extending the classification of nilpotent SAA's to algebras of dimension $10$. Throughout this part we will be working with an arbitrary field $\mathbb{F}$.
The classification depends strongly on the field we are working with. When the field is algebraically closed we will see that there are $22$ such algebras of dimension $10$. Over the field $\mbox{GF\,}(3)$, where there is a $1$-$1$ correspondence with a class of powerful $2$-Engel $3$-groups, there are $25$ algebras. \\ \\
We will first consider the case when the center of $L$ is not isotropic as then we can use our classification of algebras of lower dimension.\\ \\
By a similar argument as in the previous chapter we see that $L$ is a direct sum of an abelian algebra 
			$\mathbb{F}x_5 + \mathbb{F}y_5$ 
of dimension $2$ and an algebra of dimension $8$. Our classification of the latter yields that we get the following presentations for algebras of dimension $10$ with non-isotropic center:
\begin{align*}
 {\mathcal Q}_{10}^{(7,1)}:& \ \  (y_1y_2, y_3) = 1. \\
{\mathcal Q}_{10}^{(5,1)}:& \ \ (y_1y_2, y_3) = 1,\ (x_1y_3, y_4) = 1. \\
{\mathcal Q}_{10}^{(4,1)}(r):& \ \ (x_2 y_3, y_4) = r,\ (x_1y_2, y_4) = 1,\ (y_1y_2, y_3)=1,
\end{align*}
where ${{\mathcal Q}_{10}}^{(4,1)}(r)$ and ${{\mathcal Q}_{10}}^{(4,1)}(s)$ describe the same algebra $L$ if and only if $s/r \in (\mathbb{F}^*)^3$.\\ \\
We now turn into the case where the center $Z(L)$ is isotropic. We have that $\mbox{dim\,}Z(L)$ is between $2$ and $5$. Thus we have $4$ cases to deal with in the following $4$ chapters.
\thispagestyle{empty}

\clearpage

\chapter{Algebras with an isotropic center of dimension $5$}
Let $L$ be a nilpotent SAA of dimension $10$ with an isotropic center of dimension $5$. We can then choose a standard basis $x_1, y_1, x_2, y_2, x_3, y_3, x_4, y_4, x_5, y_5$ such that
\[ Z(L) = \mathbb{F}x_5+\mathbb{F}x_4+\mathbb{F}x_3+\mathbb{F}x_2+\mathbb{F}x_1. \]
Here $x_1, y_1, x_2, y_2, x_3, y_3, x_4, y_4, x_5, y_5$ will be determined later such that some further conditions hold. The elements $y_1, \ldots , y_5$ are not in $Z(L)$ and without loss of generality we can assume that $(y_1y_2, y_3) = 1$. 
Now suppose that $(y_iy_j, y_4) = \alpha_{ij}$ and $(y_iy_j, y_5) = \beta_{ij}$ for $1 \leq i, j \leq 3$.
 Replacing $x_1, x_2, x_3, y_4, y_5$ by
\begin{align*}
\tilde{x_1} &= x_1 + \alpha_{23} x_4 + \beta_{23} x_5, & 
\tilde{y_4} &= y_4 - \alpha_{12}y_3 - \alpha_{23}y_1 - \alpha_{31}y_2,\\
\tilde{x_2}  &= x_2 + \alpha_{31} x_4 + \beta_{31} x_5, &
\tilde{y_5} &= y_5 - \beta_{12}y_3 - \beta_{23}y_1 - \beta_{31}y_2,\\
\tilde{x_3}  &= x_3 + \alpha_{12} x_4 + \beta_{12} x_5,& \mbox{}
\end{align*}
We can assume that our standard basis has the further property that $(y_iy_j, y_4) = (y_i y_j, y_5) = 0$ for $ 1 \leq i < j \leq 3$. 
As $y_4 \not \in Z(L)$, we know that one of $(y_1y_4, y_5), (y_2y_4, y_5),$
$ (y_3y_4, y_5)$ is non-zero. 
Without loss of generality we can assume that $(y_1 y_4, y_5) = 1$. The only triples whose values are not known are then $\alpha = (y_2y_4, y_5)$ and $\beta = (y_3y_4, y_5)$. 
Replacing $x_1, y_2, y_3$ by $\tilde{x_1} = x_1 + \alpha x_2 + \beta x_3, \tilde{y_2} = y_2 - \alpha y_1, \tilde{y_3} = y_3 - \beta y_1$, we get a new standard basis where the only nonzero triple values are $(y_1y_2, y_3) = 1$ and $(y_1 y_4, y_5) = 1$. We have thus proved the following result.
\begin{Proposition}
There is a unique nilpotent SAA of dimension $10$ that has isotropic center of dimension $5$. $L$ can be described by the nilpotent presentation
\[ {\mathcal P}_{10}^{(5, 1)}: \  (y_1y_2,y_3)=1,\  (y_1y_4,y_5)=1 .\]
\end{Proposition}
\chapter{Algebras with an isotropic center of dimension $4$}
The situation here is far more complicated and we will need to consider several subcases. 
Let $L$ be a nilpotent SAA of dimension $10$ with an isotropic center of dimension $4$. 
We can pick our standard basis such that
\[ Z(L)=\mathbb{F} x_5+\mathbb{F} x_4+\mathbb{F} x_3+\mathbb{F} x_2.\]
Then
\[ L^2= Z(L)^\perp=Z(L)+\mathbb{F} x_1+\mathbb{F} y_1 \]
and by Lemma \ref{lma2.6gth} we know that $L^3=L^2L \leq Z(L)$. As $L^2 \not \leq Z(L)$ we have $L^3 \neq \{0\}$ and by Propositions \ref{pro2.10.1gth} and \ref{pro2.10gth} we then have that $3 \leq \mbox{dim\,} L^3 \leq 4$. We will consider the cases $L^3 < Z(L)$ and $L^3=Z(L)$ separately. 
\section{The algebras where $L^3 < Z(L)$}
In this case we can choose our basis such that
\newpage
\begin{multicols}{2}
%&&&&&&&&&&&&&&&&&&&&&&&& Picture &&&&&&&&&&&&&&&&&&&&&&&&&&&&&&
\begin{alignat*}{2}
\color{cyan}L^3\ &\boxed{
\begin{matrix}
\color{green} x_5\\\color{green} x_4\\\color{green}x_3
\end{matrix}
}\ \ 
\begin{matrix}
y_5\\y_4\\y_3
\end{matrix}\\
\color{cyan}Z(L)\ &\boxed{
\begin{matrix}
\color{blue} x_2
\end{matrix}}
\boxed{
\begin{matrix}
\color{blue}y_2
\end{matrix}
}\ \color{magenta}Z_2(L)\\
&\boxed{
\begin{matrix}
\color{magenta}x_1
\end{matrix}\ \  
\begin{matrix}
\color{magenta}y_1
\end{matrix}
}\ \color{cyan}L^2=Z(L)^\perp
\end{alignat*}
\break
 \begin{align*}
L^3& = \mathbb{F} x_5+\mathbb{F} x_4+\mathbb{F} x_3\\
Z(L)& = L^3+\mathbb{F} x_2\\
L^2 = Z(L)^\perp &= L^3+\mathbb{F} x_2+\mathbb{F} x_1+\mathbb{F} y_1\\
%Z_2(L) = (L^3)^\perp & =L^3+\mathbb{F} x_2+\mathbb{F} x_1+\mathbb{F} y_1+\mathbb{F} y_2
Z_2(L) = (L^3)^\perp & =Z(L)+\mathbb{F} x_1+\mathbb{F} y_1+\mathbb{F} y_2
\end{align*}
\end{multicols}
%&&&&&&&&&&&&&&&&&&&&&&&&&&&&&&&&&&&&&&&&&&&&&&&&&&&&&
\noindent
Notice that, as $Z_2(L) \cdot L^2=\{0\}$, 
we have that $Z_2(L)$ is abelian and thus in particular $x_1y_2 = y_1y_2 = 0$. 
As $Z_2(L)$ is an abelian ideal we also have that $Z_2(L)L$ is orthogonal to $Z_2(L)$ and thus  $Z_2(L) \cdot L \leq Z_2(L)^\perp = L^3$. It follows that 
		\[ \mathbb{F} x_2+\mathbb{F} x_1+\mathbb{F} y_1 + L^3 = L^2 = \mathbb{F} y_3y_4+\mathbb{F} y_3y_5+\mathbb{F} y_4y_5 + L^3. \]
Suppose
 		$$x_2+L^3=\alpha y_3y_4+\beta y_3y_5+\gamma y_4y_5+L^3.$$
Now at least one of $\alpha, \beta, \gamma$ is nonzero and by the symmetry in $y_3,y_4,y_5$ we can assume 
that $\alpha \neq 0$. Thus
\begin{align*}
			x_2+L^3 =(y_3-\frac{\gamma}{\alpha}y_5)(\alpha y_4+\beta y_5)+L^3.
\end{align*}
By replacing $x_4, x_5, y_3, y_4$ by 
	$\tilde{x_4} = \frac{1}{\alpha}x_4, 
	\tilde{x_5} = x_5 -\frac{\beta}{\alpha}x_4 +\frac{\gamma}{\alpha}x_3, 
	\tilde{y_3} = y_3-\frac{\gamma}{\alpha}y_5,
	\tilde{y_4} = \alpha y_4+\beta y_5$, 
we can then assume that $x_2 + L^3=y_3 y_4 + L^3$. In particular 
			$(y_2y_3,y_4)=1$.
Suppose that $(y_3y_4,y_5) = \tau$. 
Replacing $x_2, y_5$ by $ \tilde{x_2} = x_2+ \tau x_5$ and $\tilde{y_5} = y_5 - \tau y_2$ 
we can furthermore assume that 
			$(y_3y_4,y_5)=0$. 
If we let 
			$V = \mathbb{F}y_3 + \mathbb{F}y_4 + \mathbb{F}y_5$, 
it follows that we now have 
\begin{equation}\label{eq:4.I.1} 
			V^2 = \mathbb{F}x_2 + \mathbb{F}x_1 + \mathbb{F}y_1 \mbox{\ and \ } y_3y_4 = x_2 .
\end{equation}
By \eqref{eq:4.I.1} we know that 
\[ x_1 = \alpha y_3y_4 + \beta  y_3y_5 + \gamma y_4y_5 = \alpha x_2 + \beta y_3y_5 + \gamma y_4y_5, \]
where without loss of generality we can assume that $\gamma \neq 0$. Then
\begin{align*}
		x_1 = (y_4+\frac{\beta}{\gamma}y_3)(-\alpha y_3 +\gamma y_5).
\end{align*}
Replacing $x_3, x_5, y_4, y_5$ by 
	$\tilde{x_3} = x_3-\frac{\beta}{\gamma}x_4+\frac{\alpha}{\gamma}x_5, 
	\tilde{x_5} = \frac{1}{\gamma}x_5, 
	\tilde{y_4} = y_4+\frac{\beta}{\gamma}y_3$ and 
	$\tilde{y_5} = -\alpha y_3 + \gamma y_5$, 
we obtain 
\begin{equation}\label{eq:4.I.2} 
			y_4y_5 = x_1 . 
\end{equation}
Notice that \eqref{eq:4.I.1} is not affected by these changes. 
Finally we know from \eqref{eq:4.I.1} and \eqref{eq:4.I.2} that
		$$y_1= - \alpha x_2 - \beta x_1+\gamma y_5y_3$$
for some $0 \neq \gamma \in \mathbb{F}$. Then
\begin{align*}
		y_1=(y_5 + \frac{\alpha}{\gamma}y_4)(\gamma y_3 + \beta y_4).
\end{align*}
Now replace $x_3, x_4, y_3, y_5$ by 
	$\tilde{x_3} = \frac{1}{\gamma}x_3$, 
	$\tilde{x_4} = x_4 - \frac{\alpha}{\gamma}x_5 - \frac{\beta}{\gamma}x_3$, 
	$ \tilde{y_3} = \gamma y_3 + \beta y_4$ and
	$\tilde{y_5} = y_5 + \frac{\alpha}{\gamma}y_4$. 
This gives us 
\begin{equation}\label{eq:4.I.3} 
			y_5y_3 = y_1 . 
\end{equation}
This does not affect \eqref{eq:4.I.2} but instead of \eqref{eq:4.I.1} we get $y_3y_4=\gamma x_2$. Now we
 make the final change by replacing $x_2$ and $y_2$ by 
 		$\gamma x_2$ and $\frac{1}{\gamma} y_2$ 
and we can assume that \eqref{eq:4.I.1}, \eqref{eq:4.I.2} and \eqref{eq:4.I.3} hold.
We had seen earlier that $Z_2(L)$ is abelian and thus 
all triple values involving two elements from $\{ x_5, x_4, x_3, x_2, x_1, y_1, y_2\}$ is trivial. 
Thus all the nontrivial triple values involve two of $y_3,y_4,y_5$ but from
 \eqref{eq:4.I.1}, \eqref{eq:4.I.2} and \eqref{eq:4.I.3} we know what these are. 
 We have thus proved
\begin{Proposition}\label{pro4p1}
There is a unique nilpotent SAA of dimension $10$ that has an isotropic center of dimension $4$ and where $L^3 < Z(L)$. 
This algebra can be given by the nilpotent presentation
\begin{align*}
 {\mathcal P}_{10}^{(4,1)}: \quad (y_2y_3,y_4)=1,\ (y_1y_4,y_5)=1,\ (x_1y_3,y_5)=1.
\end{align*}
\end{Proposition}
\begin{Remark}
Inspection shows that the algebra with that presentation has a center of dimension 4 and the property that  $L^3 < Z(L)$.
\end{Remark}
\section{The algebras where $L^3=Z(L)$}
We will see that this case is quite intricate and we will need to consider several subcases.
%
%&&&&&&&&&&&&&&&&&&&&&&&& Picture &&&&&&&&&&&&&&&&&&&&&&&&&&&&&&
\begin{alignat*}{2}
\color{cyan}L^3=Z(L)\ &\boxed{
\begin{matrix}
\color{green} x_5\\\color{green} x_4\\\color{green}x_3\\\color{green}x_2
\end{matrix}
}\ \ 
\begin{matrix}
y_5\\y_4\\y_3\\y_2
\end{matrix}\\
&\boxed{
\begin{matrix}
\color{magenta}x_1
\end{matrix}\ \  
\begin{matrix}
\color{magenta}y_1
\end{matrix}
}\ \color{blue}L^2=Z(L)^\perp={L^3}^\perp=Z_2(L)
\end{alignat*}
\begin{align*}
Z(L)&=L^3=\mathbb{F} x_5+ \mathbb{F} x_4+ \mathbb{F} x_3 + \mathbb{F} x_2\\
Z_2(L)&=L^2=Z(L)^\perp=Z(L)+\mathbb{F} x_1+\mathbb{F} y_1
\end{align*}
%&&&&&&&&&&&&&&&&&&&&&&&&&&&&&&&&&&&&&&&&&&&&&&&&&&&&&
%
In order to clarify the structure further, we will associate to any such algebra 
a family of new alternating forms that are defined as follows.
For each $\bar{z}=z+Z(L) \in L^2/Z(L)$, we obtain the alternating form 
		$$\phi_{\bar{z}}:\, L/L^2 \times L/L^2 \longrightarrow \mathbb{F} $$ 
given by
		\[ \phi_{\bar{z}}(\bar{u},\bar{v}) = (zu,v) \]
where $\bar{u}=u+L^2$ and $\bar{v}=v+L^2$. Notice that this is a well defined function as $L^2$ is abelian.
\begin{Remark}\label{rem442}
$(1)$ If $0 \neq \bar{z}=z+Z(L) \in L^2/Z(L)$, then $\phi_{\bar{z}} \neq 0$. 
Otherwise we would have $(zu,v)=0$ for all $u,v \in L$ that would give the contradiction that $z \in Z(L)$ and thus $ \bar{z}=0$.\\ \\
$(2)$ There is no non-zero element in $V=L/L^2$ that is common to the isotropic part of $V=L/L^2$ with respect to all the alternating forms $\phi_{\bar{z}}$ with $\bar{z} \in L^2/Z(L)$. Otherwise there would be some $0 \neq t \in \mathbb{F} y_5+ \mathbb{F} y_4+ \mathbb{F} y_3 + \mathbb{F} y_2$ such that $(zt,u)=0$ for all $z \in L^2$ and all $u \in L$. But then $ut \in (L^2)^\perp =Z(L)$ for all $u \in L$ that gives the contradiction that $t \in Z_2(L)=L^2$.
\end{Remark}
\noindent
We divide the algebras into three categories.\\ \\
{\bf A}. The algebras where there exists a basis $\bar{z},\bar{t}$ for $L^2/Z(L)$ such that 
the alternating forms $\phi_{\bar{z}}, \phi_{\bar{t}}$ are both degenerate.\\ \\
{\bf B}. The algebras where there exists $ 0 \neq \bar{z} \in L^2/Z(L)$ such that 
$\phi_{\bar{z}}$ is degenerate but $\phi_{\bar{t}}$ is non-degenerate for all $\bar{t} \in L^2/Z(L)$ that are not in $\mathbb{F}\bar{z}$.\\\\
{\bf C}. The algebras where $\phi_{\bar{z}}$ is non-degenerate for all $0 \neq$ $\bar{z} \in L^2/Z(L).$ 
%
%-------------------+++---------
\subsection{Algebras of type $A$}
Pick $x_1, y_1 \in L^2 \, \setminus \, Z(L)$ such that $\phi_{\bar{x_1}}$ and $\phi_{\bar{y_1}}$
are degenerate and such that $(x_1, y_1) = 1$. 
By the remarks above we thus know that the isotropic part of $L/L^2$ with respect to both the alternating forms 
$\phi_{\bar{x_1}}$ and $ \phi_{\bar{y_1}}$ is of dimension $2$ and the intersection of the two is trivial. Thus we can pick a basis 
	     $\bar{y_5}=y_5+L^2, 
		\bar{y_4}=y_4+L^2,  
		\bar{y_3}=y_3+L^2, 
		\bar{y_2}=y_2+L^2$ 
for $L/L^2$ such that
\begin{align*}
		\mathbb{F}\bar{y_4}+\mathbb{F}\bar{y_5} \mbox{ is the isotropic part of }L/L^2 \mbox{ with respect to }\phi_{\bar{x_1}}.\\
		\mathbb{F}\bar{y_3}+\mathbb{F}\bar{y_2} \mbox{ is the isotropic part of }L/L^2 \mbox{ with respect to } \phi_{\bar{y_1}}.
\end{align*}
This shows that we can pick our standard basis such that
\begin{eqnarray*}
		(x_1y_2,y_3)=1 && (y_1y_2,y_3)=0\\ 
		(x_1y_2,y_4)=0 && (y_1y_2,y_4)=0\\ 
		(x_1y_2,y_5)=0 && (y_1y_2,y_5)=0\\  
		(x_1y_3,y_4)=0 && (y_1y_3,y_4)=0\\
		(x_1y_3,y_5)=0 && (y_1y_3,y_5)=0\\ 
		(x_1y_4,y_5)=0 && (y_1y_4,y_5)=1.
\end{eqnarray*}
To determine the structure fully we are only left with the triples $(y_iy_j, y_k) =r_{ijk}$ for $2\leq i < j < k \leq 5$. Let 
\begin{eqnarray*}
		\tilde{y_i} & = & y_i + \alpha_i x_1 + \alpha_i y_1,\\
		\tilde{x_1} & =&  x_1 - ( \alpha_2 x_2 + \alpha_3 x_3 + \alpha_4 x_4 + \alpha_5 x_5),\\
		\tilde{y_1} & =&  y_1 + \alpha_2 x_2 + \alpha_3 x_3 +  \alpha_4 x_4 + \alpha_5 x_5.
\end{eqnarray*}
Inspection shows that we can choose $\alpha_2, \ldots, \alpha_5$ such that 
$(\tilde{y_i} \tilde{y_j}, \tilde{y_k})=0$ for all $2 \leq i < j < k \leq 5$. In fact this works for $\alpha_2 = - r_{245}, \alpha_3=- r_{345},
 \alpha_4 = - r_{234}$ and $\alpha_5 = - r_{235}$.
We have thus proved the following result.
\begin{Proposition}\label{pro4p2} 
There is a unique nilpotent SAA of dimension $10$ with an isotropic center of dimension $4$ and where $L^3=Z(L)$ that is of type $A$. This algebras can be given by the presentation
\begin{align*}
{\mathcal P}_{10}^{(4,2)}: \quad(x_1y_2,y_3)=1,\ (y_1y_4,y_5)=1.
\end{align*}
\end{Proposition}
\noindent
Notice that inspection shows that the algebra with this presentation indeed has the properties stated in the proposition.
%
%
%-------------------+++----------
\subsection{Algebras of type $B$ }
Suppose that $\phi_{\bar{x_1}}$ is degenerate and that the isotropic part of $L/L^2$ with respect to this alternating form is 
$\mathbb{F}\bar{y_4}+\mathbb{F}\bar{y_5}$. We are now assuming that $\phi_{\bar{z}}$ is non-degenerate for all 
$\bar{z} \notin \mathbb{F} \bar{x_1}$. Pick $y_1 \in L^2$ such that $(x_1, y_1) = 1$.
\begin{Remark}\label{lma443}
We must have $\phi_{\bar{y_1}}(\bar{y_4},\bar{y_5})=0$. 
Otherwise we would get a basis $\bar{y_2}, \bar{y_3}, \bar{y_4}, \bar{y_5}$ for $L/L^2$ such that
\[ \phi_{\bar{y_1}}(\bar{y_4},\bar{y_5})=1,\ \phi_{\bar{y_1}}(\bar{y_2},\bar{y_3})=1,\ \phi_{\bar{x_1}}(\bar{y_2},\bar{y_3}) = \alpha \neq 0 \]
and where 
			$\phi_{\bar{y_1}}(\bar{y_i},\bar{y_j})=0$ 
and likewise 
			$\phi_{\bar{x_1}}(\bar{y_i},\bar{y_j})=0$ 
for any pair $\bar{y_i}, \bar{y_j}$ such that $2 \leq i < j \leq 5$ that is not included above. 
But then inspection shows that $\phi_{\alpha \bar{y_1} - \bar{x_1} }$ is degenerate where the corresponding isotropic part of 
$L/L^2$ is $\mathbb{F}\bar{y_2}+\mathbb{F}\bar{y_3}$. 
But this contradicts our assumptions.
\end{Remark}
%
%**********************
\noindent
We thus know that 
			$\phi_{\bar{y_1}}(\bar{y_4},\bar{y_5}) =0$. 
As $\phi_{\bar{y_1}}$ is non-degenerate we know that there exists some $\bar{y_2} \in L/L^2$
%
% there is some $\bar{y_2} \notin F\bar{y_4}+F\bar{y_5}$ 
%
such that 
\begin{equation}\label{eq:4.II.3}
			\phi_{\bar{y_1}}(\bar{y_2},\bar{y_4}) =1. 
\end{equation}
 Replacing $y_5$ and $y_3$ by some suitable $y_5 + \alpha y_4$ and $y_3 + \beta y_4+ \gamma y_2$ we can furthermore assume that
 \begin{equation}\label{eq:4.II.4}  \begin{aligned}
			\phi_{\bar{y_1}}(\bar{y_2},\bar{y_5}) = 
			\phi_{\bar{y_1}}(\bar{y_3},\bar{y_4}) = 
			\phi_{\bar{y_1}}(\bar{y_2},\bar{y_3}) =0.
 \end{aligned} \end{equation}
As $\phi_{\bar{x_1}}$ is non-zero we must have 
			$\phi_{\bar{x_1}}(\bar{y_2}, \bar{y_3}) \neq 0$ 
and by replacing $\bar{y_3}$ by a multiple of itself we can assume that
\begin{equation}\label{eq:4.II.5} 
			\phi_{\bar{x_1}}(\bar{y_2},\bar{y_3})=1 .
\end{equation}
Notice that this does not affect $\eqref{eq:4.II.3}$ and \eqref{eq:4.II.4}. 
As $\phi_{\bar{y_1}}$ is non-degenerate we cannot have that $\bar{y_3}$ is isotropic to all vectors in $L/L^2$ with
 respect to this alternating form. Thus by \eqref{eq:4.II.4} we must have 
 			$\phi_{\bar{y_1}}(\bar{y_3},\bar{y_5}) \neq 0$ 
and by replacing $\bar{y_5}$ by a multiple of itself we can assume that 
\begin{equation}\label{eq:4.II.6} 
			\phi_{\bar{y_1}}(\bar{y_3},\bar{y_5})=1. 
\end{equation}
Again equations $\eqref{eq:4.II.3},$ \eqref{eq:4.II.4} and $\eqref{eq:4.II.5}$ are not affected. 
We thus see that we can choose a standard basis such that 
\begin{eqnarray*}
			(x_1y_2,y_3)=1 && (y_1y_2,y_3)=0\\ 
			(x_1y_2,y_4)=0 && (y_1y_2,y_4)=1\\ 
			(x_1y_2,y_5)=0 && (y_1y_2,y_5)=0\\  
			(x_1y_3,y_4)=0 && (y_1y_3,y_4)=0\\
			(x_1y_3,y_5)=0 && (y_1y_3,y_5)=1\\ 
			(x_1y_4,y_5)=0 &&  (y_1y_4,y_5)=0.
\end{eqnarray*}
As in case $A$ we are only left with the triples $(y_i y_j,y_k) =r_{ijk}$ for all $2 \leq i< j < k \leq 5$.
As in that case we let 
\begin{eqnarray*}
			\tilde{y_i} & =& y_i + \alpha_i x_1 + \alpha_i y_1,\\
			\tilde{x_1} & =& x_1 - (\alpha_2 x_2 + \alpha_3 x_3 + \alpha_4 x_4 + \alpha_5 x_5),\\
			\tilde{y_1} & =& y_1 + \alpha_2 x_2 + \alpha_3 x_3 +  \alpha_4 x_4 + \alpha_5 x_5.
\end{eqnarray*}
Inspection shows that we can choose $\alpha_2, \alpha_3, \alpha_4, \alpha_5$ such that 
			$(\tilde{y_i}\tilde{y_j}, \tilde{y_k})=0$ 
for $2 \leq i < j < k \leq 5$. 
We thus get the following result.
\begin{Proposition}\label{pro4p3}
There is a unique nilpotent SAA of dimension $10$ with isotropic center of dimension $4$ where $L^3=Z(L)$ and $L$ is of type $B$. 
This algebra can be given by the presentation 
\begin{align*}
{\mathcal P}_{10}^{(4,3)}: \quad   (x_1y_2,y_3) =1,\ (y_1y_2,y_4) =1,\  (y_1y_3,y_5) =1. 
\end{align*}
\end{Proposition}
\begin{proof}
We have already shown that this algebra is the only candidate. Inspection 
shows that conversely this algebra has isotropic center of dimension $4$ and $L^3=Z(L)$. 
It remains to see that the algebra is of type $B$. Thus 
let $r \in \mathbb{F}$. We want to show that $\phi_{r\bar{x_1}+\bar{y_1}}$ is non-degenerate. 
Let $t=\alpha y_2+\beta y_3+\gamma  y_4+ \delta y_5$ such that 
$\phi_{r\bar{x_1}+\bar{y_1}}(\bar{t},\bar{u})=0$ for all $\bar{u} \in L/L^2$ where $\bar{t}= t + L^2$. Then
\begin{align*}
			0 &=\phi_{r\bar{x_1}+\bar{y_1}}(\bar{t},\bar{y_5})= \beta \\
			0 &=\phi_{r\bar{x_1}+\bar{y_1}}(\bar{t},\bar{y_4})= \alpha \\
			0 &=\phi_{r\bar{x_1}+\bar{y_1}}(\bar{t},\bar{y_3})= r \alpha - \delta  = - \delta \\
			0 &=\phi_{r\bar{x_1}+\bar{y_1}}(\bar{t},\bar{y_2})=- r \beta  - \gamma  = -\gamma .
\end{align*}
			Thus $ \bar{t}=0$.
\end{proof}
%-------------------+++---------------
\subsection{Algebras of type $C$ }
Here we are assuming that $\phi_{z}$ is non-degenerate for all $0 \neq z \in L^2/Z(L)$.
Let $L$ be any nilpotent SAA of type $C$. Notice that 
			$L^2/Z(L) = L^2/(L^2)^\perp$ 
naturally becomes a $2$-dimensional symplectic vector space with inherited alternating form from $L$. Thus $(u + Z(L), v + Z(L)) = (u, v)$ for $u, v \in L^2$. 
We pick a basis $x,y$ for $L^2/Z(L)$ such that $(x, y) = 1$ and then choose some fixed elements $x_1, y_1 \in L^2$ such that 
			$x=\bar{x_1}=x_1+Z(L)$ and $y=\bar{y_1}=y_1+Z(L)$.
 For any vector $u \in L/L^2 $ we will denote by ${\langle u \rangle^\perp_{1}}$ the subspace of $L/L^2$ consisting of 
 all the vectors that are isotropic to $u$ with respect to $\phi_{\bar{x_1}}$. Likewise we will denote by ${\langle u \rangle^\perp_{2}}$ the subspace of
 $L/L^2$ consisting of all the vectors that are isotropic to $u$ with respect to $ \phi_{\bar{y_1}}$. 
\begin{defn}
We say that a subspace of $L/L^2$ is {\it totally isotropic} if it is isotropic with respect to $\phi_{z}$ for all $z \in L^2/Z(L)$.
\end{defn}
\begin{Lemma}\label{lma445}
For each $0\neq u \in L/L^2$ there exists a unique totally isotropic plane through $u$. 
\end{Lemma}
\begin{proof}
We know that ${\langle u \rangle^\perp_{1}}$ and ${\langle u \rangle^\perp_{2}}$ are $3$-dimensional. Thus if they are not equal then 
			\[ 4  = \mbox{dim\,} ({\langle u \rangle_{1}^\perp }+ {\langle u \rangle_{2}^\perp}) 
			       = \mbox{dim\,} {\langle u \rangle^\perp_{1}} + \mbox{dim\,} {\langle u \rangle^\perp_{2}} -\mbox{dim\,} ({\langle u \rangle_{1}^\perp} \cap  {\langle u \rangle_{2}^\perp}). \]
Therefore $ \mbox{dim\,} ({\langle u \rangle_{1}^\perp} \cap  {\langle u \rangle_{2}^\perp}) = 6-4 = 2$. 
Thus the collection of all the elements in $L/L^2$ that are isotropic to $u$ with respect to $\phi_{z}$ for all $z \in L^2/Z(L)$, namely 
${\langle u \rangle^\perp_{1}} \cap  {\langle u \rangle^\perp_{2}}$, is a plane.\\ \\
It remains to see that ${\langle u \rangle^\perp_{1}} \neq {\langle u \rangle_{2}^\perp}$. We argue by contradiction and pick a basis 
$u, v, w$ for this common subspace and add a vector $t$ to get a basis for $L/L^2$. By replacing $x_1$ by some suitable 
$x_1 + \alpha y_1$, we can assume that $\phi_{\bar{x_1}}(u,t)=0$. But then $u$ is isotropic to all elements of $L/L^2$ with respect to 
$\phi_{\bar{x_1}}$ that contradicts the assumption that $\phi_{\bar{x_1}}$ is non-degenerate.
\end{proof}
\noindent
The alternating forms $\phi_{z}$ with $z \in L^2/Z(L)$ will help us understanding the structure of algebras of type $C$. 
We will next come up with a special type of presentations for algebras of type $C$ based on the geometry arising from the family of the auxiliary alternating forms.\\ \\
Let $L$ be an algebra of type $C$. As a starting point we pick two distinct totally isotropic planes 
			$P_1,P_2 \leq L/L^2$ 
and we pick some non-zero vector 
			$\bar{y_2}$ on $P_1$. 
By Lemma \ref{lma445}, we have that $P_1 \cap P_2 =\{0\}$ and thus $L/L^2=P_1 \oplus P_2$. 
Now ${\langle \bar{y_2} \rangle^\perp_1}$ is $3$-dimensional and contains $P_1$. Thus ${\langle \bar{y_2} \rangle^\perp_1} \cap P_2$ is $1$-dimensional 
and not contained in ${\langle \bar{y_2} \rangle_2^\perp}$ by Lemma \ref{lma445}. Thus there is unique vector 
$\bar{y_5} \in P_2$ where $\phi_{\bar{x_1}}(\bar{y_2},\bar{y_5}) = 0 $ 
and 
$\phi_{\bar{y_1}}(\bar{y_2},\bar{y_5})=1$. 
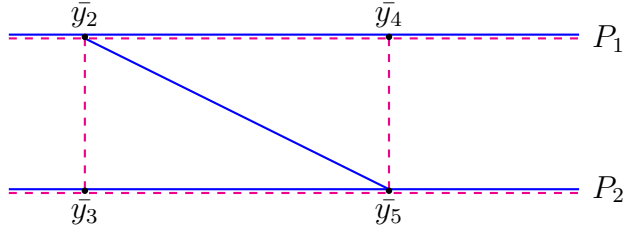
\begin{center}
\begin{tikzpicture}
\centering
       \draw [thick, magenta , dashed] (0,0) -- (0,2)      % draw  line
        node [above, black] {$\bar{y_2}$};              % add label for line
	 \node [below] at (0,0) {$\bar{y_3}$};            % like point
	 
   	\draw [thick, magenta, dashed ] (4,0) -- (4,2)      % draw  line
        node [above, black] {$\bar{y_4}$};   
	 \node [below, black] at (4,0)  {$\bar{y_5}$}; 
      \draw [thick, magenta, dashed ] (-1,-0.05) -- (6.5 ,-0.05) ;   % draw  line
 	\draw [thick, magenta , dashed ] (-1,2) -- (6.5,2) ;     % draw  line thin
              %
       %Here the line continuce not breaked
       %
      \draw [draw=blue, thick] (4,0) -- (0,2);     % can do ultra thick 
	\draw [draw=blue, thick] (-1,2.05) -- (6.5,2.05) ;
	\draw [draw=blue, thick] (-1,0) -- (6.5,0) ;
      \node [left] at (7.25,2) {$P_1$};                % label P_1
      \node [left] at (7.25,0) {$P_2$};               % label  P_2
    %
    %Here we do those 4 dots
    %
     \filldraw[black] (0,2.02) circle (1pt) node[anchor=west]{};
     \filldraw[black] (4,2.02) circle (1pt) node[anchor=west]{};
      \filldraw[black] (0,-0.02) circle (1pt) node[anchor=west]{};
     \filldraw[black] (4,-0.02) circle (1pt) node[anchor=west]{};
\end{tikzpicture}
\end{center}
\noindent
Similarly we have a unique element 
$\bar{y_3} \in P_2$ such that $\phi_{\bar{y_1}}(\bar{y_2},\bar{y_3}) = 0$ 
and 
$\phi_{\bar{x_1}}(\bar{y_2},\bar{y_3}) =1 $. 
By Lemma \ref{lma445} we have ${\langle \bar{y_5} \rangle^\perp_2} \neq {\langle \bar{y_3} \rangle^\perp_2 }$. Thus there exists a unique 
$\bar{y_4} \in P_1$ such that $\phi_{\bar{y_1}}(\bar{y_4},\bar{y_5}) = 0$ 
and 
$\phi_{\bar{y_1}}(\bar{y_3},\bar{y_4}) =1$. 
Notice also that 
			$\phi_{\bar{x_1}}(\bar{y_4},\bar{y_5}) \neq 0 $ 
and that $\bar{y_2}, \bar{y_3}, \bar{y_4}, \bar{y_5}$ form a basis for $L/L^2$. It follows from the discussion that, 
for some $\alpha,\beta \in \mathbb{F} $ with $\beta \neq 0 $, we have
\begin{equation} \label{eq:4.II.9}
\begin{aligned}
\phi_{\bar{x_1}}(\bar{y_2},\bar{y_3}) &=1 & \phi_{\bar{y_1}}(\bar{y_2},\bar{y_3}) &= 0\\ 
\phi_{\bar{x_1}}(\bar{y_2},\bar{y_4}) & = 0 & \phi_{\bar{y_1}}(\bar{y_2},\bar{y_4}) &= 0\\ 
\phi_{\bar{x_1}}(\bar{y_2},\bar{y_5}) &= 0 & \phi_{\bar{y_1}}(\bar{y_2},\bar{y_5}) &=1\\ 
\phi_{\bar{x_1}}(\bar{y_3},\bar{y_4}) &=\alpha & \phi_{\bar{y_1}}(\bar{y_3},\bar{y_4}) &= 1\\ 
\phi_{\bar{x_1}}(\bar{y_3},\bar{y_5}) &= 0 & \phi_{\bar{y_1}}(\bar{y_3},\bar{y_5}) &= 0\\ 
\phi_{\bar{x_1}}(\bar{y_4},\bar{y_5}) &=\beta & \phi_{\bar{y_1}}(\bar{y_4},\bar{y_5}) &= 0 .
\end{aligned} 
\end{equation}
\noindent
The matrix for the alternating form $\phi_{r\bar{x_1}+s\bar{y_1}}$ with respect to the ordered basis $(\bar{y_2}, \bar{y_4}, \bar{y_3}, \bar{y_5})$ is then
\begin{align*}
r \left[ {\begin{array}{rrrr}
0 & 0 & 1 & 0\\
0 & 0 & -\alpha & \beta \\
-1 & \alpha & 0 & 0 \\
0 & -\beta & 0 & 0 
\end{array} }\right ] +
s \left[ {\begin{array}{rrrr}
0 & 0 & 0 & 1\\
0 & 0 & -1 &0 \\
0 & 1 & 0 & 0 \\
-1 & 0 & 0 & 0 
\end{array} }\right  ]
\end{align*}
with determinant $(\beta r^2  +\alpha r s +s^2)^2$. As we are dealing here with algebras of type $C$ 
this determinant must be non-zero for all $(r,s) \neq (0,0)$. Equivalently we must have that the polynomial 
			$$t^2 +\alpha t + \beta $$ 
is irreducible in $\mathbb{F}[t]$.
 \noindent
Using this and \eqref{eq:4.II.9} we will now obtain a full presentation for our algebra. 
As before we are only left with the triples $(y_i y_j,y_k) =r_{ijk}$ for $2 \leq i< j < k \leq 5$. We will see that we can choose a standard basis such that
			$x_1 + Z(L) =x,\ y_1 + Z(L) = y$ and $y_i + L^2 =\bar{y_i}$ for $2 \leq i \leq 5$.
 It turns out that we do not have to alter our basis elements $x_5,\ldots, x_2$ of the center. We do this with a change of basis of the form
\begin{eqnarray*}
			\tilde{x_1} & =& x_1 - ( \alpha_2 x_2 + \alpha_3 x_3 + \alpha_4 x_4 + \alpha_5 x_5),\\
			\tilde{y_1} & =& y_1 + (\alpha_2 x_2 + \alpha_3 x_3 +  \alpha_4 x_4 + \alpha_5 x_5), \\
			\tilde{y_i} & =& y_i + \alpha_i x_1 + \alpha_i y_1.
\end{eqnarray*}
\noindent
Inspection shows that the equations $(\tilde{y_i}\tilde{y_j}, \tilde{y_k})=0$, $2 \leq i < j < k \leq 5$ are equivalent to
%
%**************************************************
%
\begin{eqnarray*}
\left[\begin{array}{cc}
 - 1 & 1 \\
\beta  & \alpha + 1 
\end{array} \right]
 \left[ \begin{array}{r} 
 \alpha_3 \\
  \alpha_5
 \end{array} \right] &
=& \left[ \begin{array}{r}
 - r_{235} \\
  - r_{345}
   \end{array} \right] , \\
\left[\begin{array} {cc}
\alpha + 1 & 1 \\
\beta  & - 1 
\end{array}\right]
 \left[\begin{array}{r} 
 \alpha_2 \\
  \alpha_4 
  \end{array} \right] &
=& \left[\begin{array}{r}
 - r_{234} \\
  - r_{245}
   \end{array} \right] .
\end{eqnarray*}
\noindent
Notice that we cannot have that $\alpha + \beta + 1 = 0$ since otherwise $1$ is a root of $t^2 + \alpha t + \beta $ that is absurd as the polynomial is irreducible. 
We thus have solution $(\alpha_2, \alpha_3, \alpha_4, \alpha_5)$ to the equation system 
and we arrive at a standard basis that gives us the following presentation
\begin{equation}\label{eq:case434}
\begin{aligned} 
\quad(x_1 y_2, y_3)&= 1, & (y_1y_2,y_5)&=1, \\ 
{\mathcal P}_{10}^{(4,4)} (\alpha,\beta) :\quad(x_1 y_3, y_4)&=\alpha, & (y_1 y_3 , y_4)&=1, \\ 
\quad(x_1 y_4 , y_5)&= \beta ,& 
\end{aligned}
\end{equation}
where the polynomial $t^2+\alpha t +\beta$ is irreducible. 
Conversely, inspection shows that any algebra with such presentation, where $t^2+\alpha t +\beta$ is irreducible, 
gives us an algebra of type $C$.\\ \\
We next turn to the isomorphism problem. That is we want to understand when two pairs 
			$(\alpha,\beta)$ and $(\tilde{\alpha},\tilde{\beta})$ 
describe the same algebra. As a first step we first prove the following lemma.
\begin{Lemma} \label{lma2.5pII} 
Let $x,y$ be elements in $L^2/Z(L)$ such that $(x,y)=1$. 
We have that the values of $\alpha$ and $\beta$ remain the same for all presentations 
of the form ${\mathcal P}_{10}^{(4,4)}(\alpha,\beta)$ where, for the given standard basis, $x_1+Z(L)=x$ and $y_1+Z(L)=y$.
\end{Lemma}
\begin{proof} 
Our method for producing $\alpha$ and $\beta$ was based on choosing some distinct totally isotropic planes $P_1, P_2$ and 
some non-zero vector $\bar{y_2}$ on $P_1$. From this we came up with a procedure that provided us with 
unique vectors $\bar{y_2}, \bar{y_4} \in P_1$ and $\bar{y_3},\bar{y_5} \in P_2$
from which the values $\alpha$ and $\beta$ can be calculated as 
			\[ \alpha=\phi_{x}(\bar{y_3},\bar{y_4}),\ \beta=\phi_{x}(\bar{y_4},\bar{y_5}) .\] 
We want to show that if $x_1+Z(L),\ y_1+Z(L)$ are kept fixed the procedure will always produce the same value for $\alpha$ and $\beta$. 
As a starting point we will see that the values do not depend on which vector $\bar{y_2}$ from $P_1$ we choose. 
We do this in two steps. First notice that if we choose $\tilde{y_2} = a \bar{y_2}$ for some $ 0 \neq a \in \mathbb{F}$, 
then the procedure gives us the new vectors $\tilde{y_3} = \frac{1}{a} \bar{y_3},\  \tilde{y_5} = \frac{1}{a} \bar{y_5}$ and $\tilde{y_4} = a \bar{y_4}$ and this gives us the values
$$\begin{array}{l}
\tilde{\alpha}=\phi_{\bar{x_1}}(\tilde{y_3},\tilde{y_4}) = \phi_{\bar{x_1}}(\frac{1}{a}\bar{y_3}, a \bar{y_4}) = \alpha,\\
\tilde{\beta}=\phi_{\bar{x_1}}(\tilde{y_4},\tilde{y_5}) = \phi_{\bar{x_1}}( a \bar{y_4}, \frac{1}{a} \bar{y_5}) = \beta.
\end{array}$$
It remains to consider the change $\tilde{y_2}= \bar{y_4}+a \bar{y_2}$. Following the mechanical procedure above produces the elements 
\begin{eqnarray*}
\tilde{y_5}&=& \frac{-\beta }{a^2-a \alpha  +\beta }\, \bar{y_3}+\frac{a-\alpha }{a^2-a \alpha +\beta } \,\bar{y_5},\\
\tilde{y_3}&=& \frac{a}{a^2-a \alpha +\beta }\, \bar{y_3}+\frac{1}{a^2-a \alpha +\beta } \, \bar{y_5}, \\  
\tilde{y_4} &=& -\beta \bar{y_2}+(a-\alpha ) \bar{y_4}.
\end{eqnarray*}
Inspection shows that again we have $\phi_{\bar{x_1}}(\tilde{y_3},\tilde{y_4}) = \alpha$ and $\phi_{\bar{x_1}}(\tilde{y_4},\tilde{y_5}) = \beta$.\\\\
%++++++-------------------------
We have thus seen that for a given pair $P_1, P_2$ of distinct totally isotropic planes we get unique values $\alpha(P_1,P_2)$ and $\beta(P_1,P_2)$ 
not depending on which vector $\bar{y_2}$ from $P_1$ we choose for the procedure above. The next step is to see that 
 			$\alpha(P_2,P_1)=\alpha(P_1,P_2)$ and $\beta(P_2,P_1)=\beta(P_1,P_2)$. 
So suppose we have some standard basis with respect to the pair $P_1,P_2$ that gives us the presentation 
${\mathcal P}_{10}^{(4,4)}(\alpha,\beta)$. 
Now consider $\tilde{y_2}=\bar{y_5},  \tilde{y_4}=\beta \bar{y_3} \in P_2$ and $\tilde{y_3}=\frac{-1}{\beta}\bar{y_4}, \tilde{y_5}=-\bar{y_2}\in P_1$. Inspection shows that
this is standard basis for the pair $P_2,P_1$ and 
$$\begin{array}{l}
\alpha(P_2,P_1)= \phi_{\bar{x_1}}( \tilde{y_3}, \tilde{y_4}) =\phi_{\bar{x_1}}( \frac{-1}{\beta}\bar{y_4}, \beta \bar{y_3})  =\alpha(P_1,P_2) \\
\beta(P_2,P_1) =\phi_{\bar{x_1}}( \tilde{y_4}, \tilde{y_5}) =
\phi_{\bar{x_1}}( \beta \bar{y_3}, - \bar{y_2}) = \beta(P_1,P_2).
\end{array}$$
Now pick any totally isotropic plane $P_3$ that is distinct from $P_1,P_2$. The aim is to show that 
 			$\alpha(P_3,P_2)=\alpha(P_1,P_2)$ and $\beta(P_3,P_2)=\beta(P_1,P_2)$. 
Take any basis for $P_3$. This must be of the form $u_1+v_1, u_2+v_2$ with $u_1,u_2 \in P_1$ and $v_1,v_2 \in P_2$. 
Notice first that $u_1+P_2, u_2+P_2$ are linearly independent vectors in $P_1+P_2/P_2$. To see this, 
take $a,b \in \mathbb{F}$ such that
			\[ P_2=a u_1+b u_2+P_2 = a (u_1+v_1) + b (u_2+v_2) + P_2.\] 
Then $ a (u_1+v_1) + b (u_2+v_2)  \in P_2 \cap P_3=\{0\}$. 
As the vectors $u_1+v_1, u_2+v_2$ are linearly independent it follows that $ a=b=0 $.
In particular 
we can choose our basis for $P_3$ to be of the form $\bar{y_2}+u, \bar{y_4}+v$ with $u,v \in P_2$. Inspection shows that for 
$\tilde{y_2}=\bar{y_2}+u, \tilde{y_4}=\bar{y_4}+v \in P_3$ and $\bar{y_3}, \bar{y_5} \in P_2$  we have a standard basis with respect to the pair $P_3, P_2$. Furthermore the corresponding parameters are
$ \phi_{\bar{x_1}}( \bar{y_3}, \tilde{y_4})=\alpha$ and $ \phi_{\bar{x_1}}( \tilde{y_4}, \bar{y_5})=\beta $.\\\\
We have now all the input we need to finish the proof of the lemma. 
Take any four totally isotropic planes $P_1, P_2, P_3, P_4$ in $L/L^2$ such that $P_1 \neq P_2$ and $P_3 \neq P_4$. 
If these planes are not all distinct then we get directly from the analysis above that 
			$\alpha(P_3,P_4)=\alpha(P_1,P_2)$ and $\beta(P_3,P_4)=\beta(P_1,P_2)$. 
Now suppose the planes are distinct. Then 
			$\alpha(P_3,P_4)=\alpha(P_1,P_4)=\alpha(P_1,P_2)$
and
			$\beta(P_3,P_4)=\beta(P_1,P_4)=\beta(P_1,P_2)$. 
This finishes the proof of the lemma.
\end{proof}
\noindent 
If follows from the lemma that if we want to obtain a new presentation for some given algebra $L$, 
then we must choose different vectors $x,y$ for $L^2/Z(L)$.
We thus only need to consider a change of standard basis for $L$ of the form 
			$\tilde{x_1}= a x_1 + b y_1, \tilde{y_1} =c x_1 + d y_1$ where $1=(\tilde{x_1},\tilde{y_1})=ad-bc$. 
Suppose that we have a presentation ${\mathcal P}_{10}^{(4,4)}(\alpha,\beta)$ with respect to some standard basis 
$x_1, y_1, \ldots, x_5, y_5$ and let $\tilde{x_1},\tilde{y_1}$ be given as above.
Going again through the standard procedure with respect to 
			$P_1=\mathbb{F} \bar{y_2}+\mathbb{F} \bar{y_4},\  
			P_2=\mathbb{F} \bar{y_3}+\mathbb{F} \bar{y_5}$ 
			and
			$\bar{y_2} \in P_1$ 
gives us the new basis $\bar{y_2}, \tilde{y_3}, \tilde{y_5}, \tilde{y_4}$ 
where
\begin{eqnarray*}
\tilde{y_5}&=& -b \bar{y_3}+a\bar{y_5} \\ 
\tilde{y_3}&=& d\bar{y_3}-c\bar{y_5}\\
 \tilde{y_4}&=&\frac{-\alpha bc-\beta ac-bd}{\beta c^2+\alpha cd+d^2} \, \bar{y_2}+\frac{1}{\beta c^2+\alpha cd +d^2} \, \bar{y_4}.
 \end{eqnarray*}
 From this we can calculate the new parameters $\tilde{\alpha}$ and $\tilde{\beta}$ and we obtain the following proposition.
\begin{Proposition}\label{prop2.6pII}
Let $L$ be a nilpotent SAA of dimension $10$ with an isotropic center of dimension $4$ that is of type $C$. Then $L$ has a presentation of the form
\begin{equation*}\label{pro4p4} 
\begin{aligned} 
\quad(x_1 y_2, y_3)&= 1, & \\ 
{\mathcal P}_{10}^{(4,4)}(\alpha,\beta) :\quad(x_1 y_3, y_4)&=\alpha, & (x_1 y_4 , y_5)&= \beta, \\ 
\quad   (y_1y_2,y_5)&=1, &  (y_1 y_3 , y_4)&=1 , 
\end{aligned}
\end{equation*}
where the polynomial $t^2+\alpha t +\beta$ is irreducible in $\mathbb{F}[t]$. 
Conversely any such presentation gives us an algebra of type $C$. Furthermore 
two presentations ${\mathcal P}_{10}^{(4,4)}(\alpha, \beta)$ and ${\mathcal P}_{10}^{(4,4)}(\tilde{\alpha}, \tilde{\beta})$ 
describe the same algebra if and only if 
\begin{eqnarray*}
\tilde{\alpha} &=&\frac{ (ad+bc) \alpha + 2ac \beta + 2bd }{ d^2 +  c d \alpha+ c^2 \beta}\\
\tilde{\beta} &=& \frac{ b^2 + ab \alpha + a^2 \beta}{ d^2 +cd \alpha + c^2 \beta}
\end{eqnarray*}
for some $a,b,c,d \in \mathbb{F}$ where $ad-bc=1$.
\end{Proposition}
 \noindent
\subsection{Further analysis of algebras of type $C$ and some examples }
In order to get a more transparent picture of the algebras of type $C$, it turns out to be useful to consider the case 
when the characteristic is $2$ and the case when the characteristic is not $2$ separately.
\begin{Lemma}\label{lma2.7pII}
Let $L$ be an algebra of type $C$ over a field $\mathbb{F}$ of characteristic that is not $2$. 
Then $L$ has a presentation of the form ${\mathcal P}_{10}^{(4,4)}(0,\beta)$ with respect to some standard basis.
 \end{Lemma}
\begin{proof} 
By Proposition \ref{pro4p4} we know that we can choose a standard basis for $L$ so that $L$ has a presentation of the form 
${\mathcal P}_{10}^{(4,4)}(\alpha,\beta)$ with respect to this basis.
Now let 
			$a=0$, $b=1$, $c=-1$ and $d = \alpha/2$. 
Then $ad - bc = 1$ and by Proposition \ref{pro4p4} again we know that there is presentation for $L$ of the form 
${\mathcal P}_{10}^{(4,4)}(\tilde{\alpha}, \tilde{\beta})$ where $\tilde{\alpha} = 0$.
\end{proof}
\begin{Proposition}\label{prop2.8pII}
Let $L$ be a nilpotent SAA of type $C$ over a field $\mathbb{F}$ of characteristic that is not $2$. 
Then $L$ has a presentation of the form
			\[ {\mathcal P}(\beta) :\quad (x_1y_2,y_3)=1,\ (x_1y_4,y_5)=\beta,\ (y_1y_2,y_5)=1,\ (y_1y_3,y_4)=1,\]
where $\beta \notin - \mathbb{F}^2$. 
Conversely any such presentation gives us an algebra of type $C$. 
Furthermore two such presentations ${\mathcal P}(\beta)$ and ${\mathcal P}(\tilde{\beta})$ describe the same algebra 
					if and only if 
		\[ \tilde{\beta}/\beta = (a^2+b^2 \beta)^2 \]
for some $(a,b) \in \mathbb{F} \times  \mathbb{F} \, \setminus \, \{(0,0)\}$.
\end{Proposition}
\begin{proof}
From Lemma \ref{lma2.7pII} we know that such a presentation exists and the polynomial $t^2 + \beta$ is irreducible if and only if 
$\beta \notin -\mathbb{F}^2$.
%
%{\color{cyan} (wonder what about the infinite field of $char\ p \neq 0$?)}
%
By Proposition \ref{prop2.6pII} we then know that 
${\mathcal P}(0,\beta)$ and ${\mathcal P}(0,\tilde{\beta})$ describe 
the same algebra if and only if there are $a,b,c,d \in \mathbb{F}$ such that
\[ 0 = ac \beta +bd ,\ ad-bc = 1 \]
and
\[\tilde{\beta} = \frac{b^2 + a^2 \beta}{d^2 + c^2 \beta}. \]
Solving these together shows that 
%for $c = \frac{-b}{b^2 + a^2 \beta}$, $d=\frac{ a \beta}{b^2 + a^2 \beta}$ 
these conditions are equivalent to saying that
			$$\tilde{\beta}= ((\frac{b}{\beta})^2 \beta + a^2)^2 \beta$$
for some $a,b \in {\mathbb F}$. As $\frac{b}{\beta}$ is arbitrary, the second part of the proposition follows.
\end{proof}
\noindent
\textbf{Examples.} (1) If $\mathbb{F} =\mathbb{C}$ then there are no algebras of type $C$.
This holds more generally for any field $\mathbb{F}$ whose characteristic is not $2$ and where all elements in $\mathbb{F}$ have a square root in $\mathbb{F}$. \\ \\
(2) Suppose $\mathbb{F} =\mathbb{R}$. Here $\beta \notin -\mathbb{R}^2$ if and only if $\beta >0$ 
in which case there exist
$a \in \mathbb{R}$ such that $\beta = a^4$.
Hence, by Proposition \ref{prop2.8pII}, ${\mathcal P}(\beta)$ describes the same algebra as ${\mathcal P}(1)$. 
This shows that there is only one algebra of type $C$ over $\mathbb{R}$ that can be given by the presentation
\[ {\mathcal P}(1) : \quad (x_1y_2,y_3) = 1,\ (x_1y_4,y_5) = 1,\ (y_1y_2,y_5) = 1,\ (y_1y_3,y_4) = 1.\]
(3) Let $\mathbb{F}$ be a finite field of some odd characteristic $p$. Suppose that $|\mathbb{F}| = p^n$. 
The non-zero elements form a cyclic group $\mathbb{F}^*$ of order $p^n -1$ that is divisible by $2$.
Thus there are two cosets of $(\mathbb{F}^*)^2$ in $\mathbb{F}^*$ and 
			\[ \mathbb{F}^* = -(\mathbb{F}^*)^2 \cup \beta (\mathbb{F}^*)^2 \]
for some $\beta \in \mathbb{F}^*$.
Suppose $\tilde{\beta} = \beta c^2$ is an arbitrary field element that is not in $- \mathbb{F}^2$. 
As there are $(|\mathbb{F}| + 1) / 2$ squares in $\mathbb{F}$ we have that the set 
			$c - \mathbb{F}^2 $ and $\beta \mathbb{F}^2$ intersect. 
Hence there exist $a , b \in \mathbb{F}$ such that 
			$c = a^2 + b^2 \beta $ and thus 
			$\tilde{\beta} = (a^2 + b^2 \beta )^2 \beta $. 
Hence the situation is like in (2) and we get only one algebra with presentation
		\[{\mathcal P}(\beta) : \  (x_1y_2,y_3) = 1,\ (x_1y_4,y_5) = \beta,\ (y_1y_2,y_5) = 1,\ (y_1y_3,y_4) = 1,\]
where $\beta$ is any element not in $- \mathbb{F}^2$.
\begin{Remark}\label{lma456}
Let $\beta \in \mathbb{F}$ that is not in $-\mathbb{F}^2$ and consider a splitting field $\mathbb{F}[\gamma]$ of the polynomial 
$t^2 +\beta $ in $\mathbb{F}[t]$ where $\gamma ^2=- \beta $. So 
			$a^2 + b^2 \beta $
 is the norm 
 			$N(a+\gamma b)$ of $a+\gamma b$ 
that is a multiplicative function and thus 
\[G(\beta)=\{ (a^2 + b^2 \beta)^2 :\, (a,b) \in \mathbb{F} \times \mathbb{F} \, \setminus \, \{(0,0)\} \} \]
is a subgroup of $(\mathbb{F}^*)^2$.
Let $S=\{\beta \in \mathbb{F} :\, \beta ^2 \not \in - \mathbb{F}^2 \}$, we now have a relation on $S$ given by 
\[ \tilde{\beta} \sim \beta  \mbox{\  if and only if\  } \tilde{\beta} / \beta \in G(\beta) \]
From Proposition \ref{prop2.8pII}, we know that this is an equivalence relation. We can also 
see this directly. First notice that 
\[ (a^2 + b^2 \beta)^2 = (a^2 + (b/c)^2 \beta c^2)^2 \]
for all $c \in \mathbb{F}^*$. Hence 
			$G(\beta ) = G(\beta c^2)$ 
for all $c \in \mathbb{F}^*$. In particular we have that 
			 $G(\tilde{\beta})=G(\beta)$ if $ \tilde{\beta} \sim \beta$. 
%
%{\color{cyan} can't say "iff" cos conversely $a^2+b^2 \beta=c^2+d^2 \tilde{\beta}$ for some $c,d \in F$ wont lead to the other part isn't? }Yup\\
%
Let us now see that $\sim$ is an equivalence relation. Firstly it is reflexive as $\beta/\beta=1 \in G(\beta)$, the latter being a group. To see that $\sim$ is symmetric, 
notice that $G(\tilde{\beta})=G(\beta)$ is a group and thus $ \tilde{\beta}/ \beta \in G(\beta)$ if and only if $\beta / \tilde{\beta} \in G(\tilde{\beta})$. 
Finally to see that $\sim$ is transitive, let $\alpha, \beta, \delta \in S$ such that $\alpha \sim \beta$ and $\beta \sim \delta$. Then 
$G(\alpha)=G(\beta)=G(\delta)$ and $\beta/\alpha, \delta/\beta \in G(\alpha)$ implies that their product $\delta/\alpha \in G(\alpha)$. 
\end{Remark}
\noindent
Let us now move to the case when the characteristic of $\mathbb{F}$ is $2$. We first see that the algebras here split naturally into two classes. 
\begin{Lemma} \label{lma2.9pII}
Let $L$ be an algebra of type $C$ over a field $\mathbb{F}$ of characteristic $2$. Then $L$ cannot have both a presentation of the form ${\mathcal P}_{10}^{(4,4)}(0,\beta)$ and 
${\mathcal P}_{10}^{(4,4)}(\alpha,\gamma)$ where $\alpha \neq 0$.
\end{Lemma}
\begin{proof}
We argue by contradiction and suppose we have an algebra satisfying both types of presentations. 
By Proposition \ref{prop2.6pII} we then have $0=(ad+bc)\alpha =(ad-bc) \alpha = \alpha $ that contradicts 
our assumption that $\alpha \neq 0$. 
\end{proof}
\noindent
For the algebras where $\alpha =0$, the same analysis works as for algebras with $\mbox{char\,} {\mathbb F} \neq 2$ and we get the same result as in Proposition \ref{prop2.8pII}. 
This leaves us with algebras where $\mbox{char\,} \mathbb{F} = 2$ and where $\alpha \neq 0$. 
Notice that Proposition \ref{prop2.6pII} tells us here that ${\mathcal P}_{10}^{(4,4)}(\alpha,\beta)$ and 
${\mathcal P}_{10}^{(4,4)}(\tilde{\alpha},\tilde{\beta})$ describe the same algebra if and only if 
\begin{eqnarray*}\label{eq:charFis2}
\tilde{\alpha} &=& \frac{\alpha} {d^2 + cd \alpha+ c^2 \beta }\\
\tilde{\beta} &=& \frac{ b^2 + ab \alpha + a^2 \beta }{d^2 + cd \alpha+ c^2 \beta}
\end{eqnarray*}
for some $a,b,c,d \in \mathbb{F}$ where $ad+bc=1$.\\ \\
We don't take the analysis further but end by considering an example, the finite fields of characteristic $2$.
%
% Let $\mathbb{F}$ be the finite field of order $2^n$. As a first step we show that we can always in that case 
%
%\textbf{Example. }
Let $\mathbb{F}$ be the finite field of order $2^n$. 
%
%We show there is a unique nilpotent SAA of type $C$ over $\mathbb{F}$. 
%
As a first step we show that we can always in that case, choose our presentation such that $\beta=1$.
To see this take first some arbitrary $\alpha$ and $\beta$ such that $L$ satisfies the presentation ${\mathcal P}_{10}^{(4,4)}(\alpha, \beta)$. 
The groups of units $\mathbb{F}^*$ is here a cyclic group of odd order $2^n - 1$ 
and thus $(\mathbb{F}^*)^2 = \mathbb{F}^*$. Now pick $b \in \mathbb{F}^*$ such that 
			$b^2=\beta/\alpha$
and let
			 $a=0$, $c=1/b$ and $d=b$.
Then $ad+bc=1$ and 
			\[ \tilde{\beta} = \frac{b^2}{b^2 + \alpha + \beta/b^2}=1.\]
Thus we can assume from now on that $\beta = 1$.
Now let $b \in \mathbb{F}$ be arbitrary and let 
			$a=b+1$, $c=b$ and $d=b+1$. 
Then  
		$ad+bc=1$,
		$\tilde{\beta}=1$ 
and
		$\tilde{\alpha}=\frac{\alpha}{b(b+1) \alpha+1}$. 
The number of such values of $\tilde{\alpha}$ is $2^{n-1}$ that gives us all possible values such that $t^2+\tilde{\alpha} t +1$ is irreducible
(easy to count the number of reducible polynomials of the form $t^2 + ut + 1$ is $2^{n-1}$). 
We thus conclude that there is only one algebra of type $C$ in this case.

\chapter{Algebras with an isotropic center of dimension $3$}
In this chapter we will be assuming that $Z(L)$ is isotropic of dimension $3$. First we derive some properties that hold for these algebras. Here throughout
\begin{eqnarray*}
Z(L)&=&\mathbb{F} x_5+ \mathbb{F} x_4 + \mathbb{F} x_3 \\
L^2&=& Z(L)+ \mathbb{F} x_2 + \mathbb{F} x_1 + \mathbb{F} y_1 + \mathbb{F} y_2.
\end{eqnarray*}
\begin{Lemma}\label{lma3.1pII} $Z(L) \leq L^3$.
\end{Lemma}
\begin{proof} Otherwise $Z_2(L) = (L^3)^\perp \not \leq Z(L)^\perp = L^2$. Without loss of generality we can suppose that $y_3 \in Z_2(L) \, \setminus \, L^2$.
 As $Z_2(L) \cdot L^2  = \{0 \}$, we then have $y_3 \cdot L^2 = \{0 \}$. Now also $x_2 \cdot L^2 = \{ 0 \}$. Let $\alpha = (x_2 y_4, y_5)$ and $\beta = (y_3 y_4, y_5)$. 
Notice that $\alpha, \beta \neq 0$ as $x_2, y_3 \not \in Z(L)$. But then
\[ (( \beta x_2 - \alpha y_3) y_4, y_5) = 0 \]
that implies that $\beta x_2 - \alpha y_3 \in Z(L)$. This is absurd.
\end{proof}
\noindent
\begin{Lemma} $\mbox{dim\,} L^3 \geq 5$. \end{Lemma}
\begin{proof} Otherwise $\mbox{dim\,} L^3 \leq 4$ and as $Z(L) \leq L^3 \leq L^2 = Z(L)^\perp$ we can choose our standard basis such that $Z(L) = \mathbb{F} x_5+ \mathbb{F} x_4 + \mathbb{F} x_3 $ and
\[ L^3 \leq \mathbb{F} x_5+ \mathbb{F} x_4 + \mathbb{F} x_3 + \mathbb{F} x_2. \]
This implies that $\mathbb{F} x_5+ \mathbb{F} x_4 + \mathbb{F} x_3 + \mathbb{F} x_2 + \mathbb{F} x_1 + \mathbb{F} y_1 \leq (L^3)^\perp = Z_2(L)$ and 
(notice that $Z_2(L) \leq L^2 $ as $Z(L) \leq L^3$) $ L^2 = Z_2(L) + \mathbb{F} y_2$ that implies that $L^2$ is abelian. Then for any $x \in L^2$ and $a, b, c \in L$, we have
\[ (x, abc) = - (x(ab), c) = - (0, c) = 0\]
and $L^3 \leq (L^2)^\perp = Z(L)$. Hence $L^3 = Z(L)$ and $Z_2(L) = L^2$. Suppose $L = Z_2(L) +  \mathbb{F} u_1  +  \mathbb{F} u_2 +  \mathbb{F} u_3 $. Then $L^2 = Z(L) +  \mathbb{F} u_1u_2 +  \mathbb{F} u_1u_3 + \mathbb{F} u_2u_3$ and we get the contradiction that $ 4 = \mbox{dim\,} L^2 - \mbox{dim\,} Z(L) \leq 3$.
\end{proof}
\begin{Lemma} If $\mbox{dim\,} L^3 = 5$, then $L^3$ is isotropic. \end{Lemma}
\begin{proof} Otherwise we can choose our basis such that $L^3 = \mathbb{F} x_5+ \mathbb{F} x_4 + \mathbb{F} x_3 + \mathbb{F} x_1 + \mathbb{F} y_1 $ and then
$Z_2(L) = (L^3)^\perp = \mathbb{F} x_5+ \mathbb{F} x_4 + \mathbb{F} x_3 + \mathbb{F} x_2 + \mathbb{F} y_2 $ and as $L^2 \cdot Z_2(L) = \{0 \}$, it follows that $x_1 y_2 = y_1y_2 = 0$. Then $L^2$ is abelian and thus we get the contradiction that $L^3 \leq Z(L)$.
\end{proof}
\begin{Lemma}\label{zlinl4} $Z(L) \leq L^4$ .\end{Lemma}
\begin{proof}
We have seen that $\mbox{dim\,} L^3 \geq 5$. So we can choose our standard nilpotent basis such that either
\[L^3 =   \mathbb{F}x_5 + \mathbb{F}x_4 + \mathbb{F}x_3+ \mathbb{F}x_2 + \mathbb{F}x_1 \]
or
\[L^3 =   \mathbb{F}x_5 + \mathbb{F}x_4 + \mathbb{F}x_3+ \mathbb{F}x_2 + \mathbb{F}x_1 + \mathbb{F}y_1. \]
We consider the two cases in turn beginning with the first case. If $Z(L) \not \leq L^4$, then $\mbox{dim\,} Z(L) \cap L^4 \leq 2$ and thus $\mbox{dim\,} L^2 + Z_3(L) = \mbox{dim\,} (Z(L) \cap L^4)^\perp \geq 8$. Suppose $L = L^2 + Z_3(L) + \mathbb{F} u + \mathbb{F} v$. Then $L^2 = L^3 + Z_2(L) + \mathbb{F} uv = L^3 + \mathbb{F} uv$ and we get the contradiction that $\mbox{dim\,} L^2 \leq 5 + 1 = 6$. We now turn to the second case where $L^3 = \mathbb{F}x_5 + \mathbb{F}x_4 + \mathbb{F}x_3+ \mathbb{F}x_2 + \mathbb{F}x_1 + \mathbb{F}y_1$. We argue by contradiction and suppose that $Z(L) \cap L^4 < Z(L)$. Then we can choose our basis such that
\[ L^4 \leq  \mathbb{F}x_5 + \mathbb{F}x_4 + \mathbb{F}x_2  \]
and $Z(L) \cap L^4 \leq \mathbb{F}x_5 + \mathbb{F}x_4$. Now $y_3 \in (L^4)^\perp = Z_3(L)$ and as $Z_3(L) \cdot L^3 = \{0\}$, it follows that
\[ x_1y_3 = x_2y_3 = y_1y_3 = 0.\]
It follows from this that $x_1y_2, y_1y_2, y_3y_2 \in \mathbb{F}x_5 + \mathbb{F}x_4$. Thus in particular these three elements are linearly dependent and we have $(\alpha x_1 + \beta y_1 + \gamma y_3) y_2 = 0$ where not all of $\alpha, \beta , \gamma $ are zero. Then $x_2 , \alpha x_1 + \beta y_1 + \gamma y_3$ commute with all the basis elements except possibly $y_4$ and $y_5$. Suppose
\begin{eqnarray*} (x_2y_4, y_5) & =& r\\
((\alpha x_1 + \beta y_1 + \gamma y_3)y_4, y_5) &=&  s.
\end{eqnarray*}
If $ r =0$ then we get the contradiction that $x_2  \in Z(L)$ and if $r \neq 0$, we get the contradiction that $ - s x_2 + r \alpha x_1 + r \beta y_1 + r \gamma y_3 \in Z(L)$.
\end{proof}
\noindent
After these more general results we classify all the algebras where $Z(L)$ is isotropic of dimension $3$. We consider the two subcases $\mbox{dim\,} L^3 = 5$ and $\mbox{dim\,} L^3 = 6$ separately. 
%
%++-------------------++-----------------
\section{ The algebras where $\mbox{dim\,} L^3 = 5$ }
We have seen that $L^3$ must be isotropic and thus in particular we have that $L^3 = (L^3)^\perp = Z_2(L)$ that implies that $L^4 \leq Z(L)$. By Lemma \ref{zlinl4} we thus have $L^4 = Z(L)$. 
We have thus determined the terms of the lower and the upper central series
%
%
%&&&&&&&&&&&&&&&&&&&&&&&& Pictuer &&&&&&&&&&&&&&&&&&&&&&&&&&&&&&
%
\begin{multicols}{2}
%
%\scriptsize
\begin{alignat*}{2}
\color{magenta}L^2 \cdot L^2\ &\boxed{
\begin{matrix}
\color{green} x_5
\end{matrix}
}\ \ 
\begin{matrix}
y_5
\end{matrix} \\
\color{cyan}Z(L)=L^4\ &\boxed{
\begin{matrix}
\color{green} x_4\\\color{green}x_3
\end{matrix}
}
\boxed{
\begin{matrix}
\color{green}y_4\\ \color{green}y_3
\end{matrix}} \ \color{magenta}(L^2 \cdot L^2)^\perp\\
\color{cyan}Z_2(L)=L^3\ &\boxed{
\begin{matrix}
\color{blue} x_2\\ \color{blue}x_1
\end{matrix}}
\boxed{
\begin{matrix}
\color{blue}y_2\\\color{blue}y_1
\end{matrix}
}\ \color{cyan}Z_3(L)=L^2
\end{alignat*} 
\break
%\centering
%&&&&&&&&&&&&&&&&&&&&&&&&&&&&&&&&&&&&&&&&&&&&&&&&&&&&&
%$$\begin{array}{l}
\begin{align*}
Z(L) &= L^4 = \mathbb{F} x_5+ \mathbb{F} x_4 + \mathbb{F} x_3 \\
Z_2(L) &= L^3=Z(L)+\mathbb{F} x_2 + \mathbb{F} x_1 \\
Z_3(L) &= L^2=Z_2(L)+\mathbb{F} y_1 + \mathbb{F} y_2 
%
%\end{array}$$
\end{align*}
\end{multicols}
\begin{Remark} \label{rem33I.1}
As $L^2\cdot Z_2(L)=\{0\}$ we see that $x_1y_2=0$. Now
 $L^2$ is not abelian as this would imply that $L^3 \leq  Z(L)$. It follows
 that $y_1y_2 \neq 0$ and we get a one dimensional characteristic subspace
\[ L^2 \cdot L^2=\mathbb{F} y_1y_2 .\]
\end{Remark}
\noindent
Notice that $y_1y_2 \in Z(L)$. We choose our basis such that $y_1y_2=x_5$. 
We will also work with the $9$ dimensional characteristic subspace
\[V= (L^2 \cdot L^2)^\perp=  \mathbb{F} x_5+ \mathbb{F} x_4 + \mathbb{F} x_3 + \mathbb{F} x_2 +
 \mathbb{F} x_1+ \mathbb{F} y_1 + \mathbb{F} y_2+ \mathbb{F} y_3 + \mathbb{F}
 y_4. \]
As $x_1y_2=0$ we have that $x_1y_3, x_1y_4 \perp y_2$. 
As $y_1y_2=x_5$ we also have that $y_1y_3, y_1y_4 \perp y_1,y_2$ and $y_2y_3,
 y_2y_4 \perp y_1,y_2$. It follows that 
\begin{align*}
V^2+L^4= (L^2+\mathbb{F} y_3 + \mathbb{F} y_4)(L^2+\mathbb{F} y_3 +
 \mathbb{F} y_4)+L^4
=  \mathbb{F} y_3y_4+L^4.
\end{align*}
We consider few subcases.
%
%+++-------------+++------------------+++----
\subsection{ Algebras where $ V^2 \leq L^4$ }
Notice then that $y_3y_4 \in \mathbb{F} x_5$ and thus $x_1y_3, x_2y_3, x_1y_4, x_2y_4 \in
 \mathbb{F} x_5$. As $L^3 = L^4+ \mathbb{F} x_2 + \mathbb{F} x_1$, we have 
\begin{equation*}
\begin{aligned}
L^4 = (\mathbb{F} x_2 + \mathbb{F} x_1)( \mathbb{F} y_3 + \mathbb{F} y_4 +
  \mathbb{F} y_5) = \mathbb{F} x_5 + \mathbb{F} x_2y_5+ \mathbb{F} x_1y_5.
\end{aligned}
\end{equation*}
\noindent 
Pick $x_5, y_5, x_2, x_1, y_1$ satisfying the conditions above and let 
$$ x_4 = - x_2 y_5,\ x_3= - x_1 y_5 $$
We can then extend $x_5, x_4, x_3, y_1, y_2, y_5$ to a standard basis 
$x_5, x_4, x_3, x_2, x_1, y_1, y_2,$
$y_3, y_4, y_5$ satisfying the conditions above.
All triples involving both $x_1$ and $y_2$ are $0$. The remaining ones are
\begin{alignat*}{3}
 (x_{1}y_{3},y_{5})&=1 &\quad  (x_{1}y_{3},y_{4})&=0  &\quad  (x_{1}y_{4},y_{5})&=0\\
 (x_{2}y_{3},y_{4})&=0 &\quad  (x_{2}y_{3},y_{5})&=0  &\quad  (x_{2}y_{4},y_{5})&=1\\
 (y_{1}y_{2},y_{3})&=0 &\quad  (y_{1}y_{2},y_{4})&=0  &\quad  (y_{1}y_{2},y_{5})&=1\\
 (y_{2}y_{3},y_{4})&=0 &\quad  (y_{2}y_{3},y_{5})&=\alpha  &\quad  (y_{2}y_{4},y_{5})&=\beta\\
 (y_{1}y_{3},y_{4})&=0 &\quad   (y_{1}y_{3},y_{5})&=\gamma  &\quad  (y_{1}y_{4},y_{5})&=\delta\\
 (y_{3}y_{4},y_{5})&=r  &\quad   &\quad  \mbox{}  \\
\end{alignat*}
Now let
\begin{eqnarray*}
\tilde{y_3}&=& y_3+ \alpha y_1 - \gamma y_2 - s x_2 - s \gamma x_3 - s \delta x_4 \\
\tilde{y_4}&=& y_4+ \beta y_1 - \delta y_2\\
\tilde{x_1}&= & x_1 - \alpha x_3 - \beta x_4\\
\tilde{x_2}&=  & x_2+ \gamma x_3 + \delta x_4 \\
\tilde{y_2}&= & y_2 - s x_3
\end{eqnarray*}
where $s=r + \alpha \delta - \beta \gamma$. One checks readily that we get a new standard basis with a presentation like the one above where
$\tilde{\alpha}=\tilde{\beta}=\tilde{\gamma}=\tilde{\delta}=\tilde{r}=0$.\\ \\
So we arrive at a unique algebra with presentation 
\begin{align*}
\quad(x_1y_3,y_5)&=1\\
{\mathcal P}_{10}^{(3,1)}: \quad(x_2y_4,y_5)&=1\\
\quad(y_1y_2,y_5)&=1.
\end{align*}
One can check that the center has dimension $3$ and that $L^3$ has dimension $5$. Also $((L^2 \cdot L^2)^\perp)^2=\mathbb{F}x_5 \leq L^4$.
%
%+++-------------+++------------------+++-------------
\subsection{ Algebras where $ V^2 \nleq L^4$ but $V^2 \leq L^3$ }
Here we can pick our basis such that
\[ V^2+L^4 =  \mathbb{F} x_5 + \mathbb{F} x_4 +  \mathbb{F} x_3 + \mathbb{F} x_2 .\]
Notice that $V^3 = \mathbb{F} x_2 y_3+ \mathbb{F} x_2 y_4$ and as $(y_3 y_4, x_2)=0$ we have that $V^3 \leq \mathbb{F} x_5$. 
As $x_2 \not \in Z(L)$, we furthermore must have that $\mbox{dim\,} V^3=1$. This means that there is a characteristic ideal $W$ of codimension $1$ in $V$ such that $x_2W = V^2W = \{0\}$. We choose our basis such that
\[ W = \mathbb{F} x_5 + \mathbb{F} x_4 +  \mathbb{F} x_3 + \mathbb{F} x_2+  \mathbb{F} x_1 + \mathbb{F} y_1 + \mathbb{F} y_2 +  \mathbb{F} y_3 .\]
It follows that we have a chain of characteristic ideals : 
\newpage
\begin{multicols}{2}
%
%&&&&&&&&&&&&&&&&&&&&&&&& Pictuer &&&&&&&&&&&&&&&&&&&&&&&&&&&&&&
\begin{alignat*}{2}
\color{magenta}L^2 \cdot L^2\ &\boxed{
\begin{matrix}
\color{green} x_5
\end{matrix}
}\ \ 
\begin{matrix}
y_5
\end{matrix} \\
\color{magenta}W^\perp \ &\boxed{
\begin{matrix}
\color{green} x_4
\end{matrix}
}
\boxed{\begin{matrix}
\color{green}y_4
\end{matrix}} \ \color{magenta}V\\
\color{cyan}Z(L)= L^4\ &\boxed{
\begin{matrix}
\color{green}x_3
\end{matrix}
}
\boxed{
\begin{matrix}
 \color{green}y_3
\end{matrix}} \ \color{magenta}W\\
\color{magenta}V^2+L^4\ &\boxed{
\begin{matrix}
\color{blue} x_2
\end{matrix}}
\boxed{
\begin{matrix}
\color{blue}y_2
\end{matrix}
}\ \color{cyan}Z_3(L)=L^2\\
\color{cyan}Z_2(L)=L^3\ &\boxed{
\begin{matrix}
\color{blue} x_1
\end{matrix}}
\boxed{
\begin{matrix}
\color{blue}y_1
\end{matrix}
}\ \color{magenta}(V^2)^\perp \cap L^2
\end{alignat*}
\break
%&&&&&&&&&&&&&&&&&&&&&&&&&&&&&&&&&&&&&&&&&&&&&&&&&&&&&
\begin{equation*}\label{eq:33IB.1}
\begin{aligned}
L^2 \cdot L^2 & =\mathbb{F} x_5\\
W^\perp &= \mathbb{F} x_5 + \mathbb{F} x_4\\
Z(L)=L^4&=\mathbb{F} x_5+ \mathbb{F} x_4 + \mathbb{F} x_3 \\
V^2+L^4 &=\mathbb{F} x_5+ \mathbb{F} x_4 + \mathbb{F} x_3 + \mathbb{F} x_2 \\
Z_2(L)=L^3 &=Z(L) +\mathbb{F} x_2 + \mathbb{F} x_1 \\
(V^2)^\perp \cap L^2&=L^3 + \mathbb{F} y_1 \\
Z_3(L)=L^2 &=Z_2(L) +\mathbb{F} y_1 + \mathbb{F} y_2 \\
W&=L^3+\mathbb{F} y_1 +\mathbb{F} y_2 +\mathbb{F} y_3\\
V&=L^3+\mathbb{F} y_1 +\mathbb{F} y_2 +\mathbb{F} y_3+\mathbb{F} y_4
\end{aligned}
\end{equation*}
\end{multicols}
%
%***************************************
\noindent
We want to show that there is again a unique algebra satisfying these conditions. 
We modify the basis and reach a unique presentation. Notice that $V^2W=\{0\}$ and $L^2 \cdot Z_2(L)=\{0\}$ imply that
		\[ x_1y_2 = x_2 y_3 = 0 .\]
We have also chosen our basis such that
\begin{equation}\label{eq:33IB.2} \begin{aligned}
			 y_1y_2 = x_5 .
\end{aligned} \end{equation}
Notice next that $x_2y_3=0$ implies that $x_2y_4$ is orthogonal to $y_3$ and $y_4$ and thus 
			$x_2y_4=r x_5$ 
where $r$ must be nonzero as $x_2 \notin Z(L)$. By replacing $y_4$ and $x_4$ by $r y_4$ and 
$\frac{1}{r} x_4$, we can assume that
\begin{equation}\label{eq:33IB.3} \begin{aligned} 
			x_2y_4 = x_5 .
\end{aligned} \end{equation}
As $y_3y_4 \in V^2 \leq L^3$ and as $x_1y_2=0$ we have that $x_1y_4$ is orthogonal to $y_2, y_3, y_4$. Thus 
		$x_1y_4 = \alpha x_5$ 
for some $\alpha \in \mathbb{F}$. Replacing $x_1, y_2$ by $x_1 - \alpha x_2$ and $y_2+\alpha y_1$ 
we get a new standard basis where
\begin{equation}\label{eq:33IB.4} \begin{aligned} 
			x_1y_4=0 .
\end{aligned} \end{equation}
Notice that the change in $y_2$ does not affect \eqref{eq:33IB.2}. We next turn our 
attention to $x_1y_3$. As $x_1y_2=0$ and $x_1y_4 = 0$, we have that $x_1y_3$ is 
orthogonal to $y_1, y_2, y_3, y_4$ and thus 
			$x_1y_3= r x_5$ 
where $r$ is nonzero since $x_1 \notin Z(L)$. Be replacing $y_3$ and $x_3$ by $r y_3$ 
and $\frac{1}{r} x_3$ we get 
\begin{equation}\label{eq:33IB.5} \begin{aligned} 
			x_1y_3=x_5 . 
\end{aligned} \end{equation}
Now we see, as $y_1y_2=x_5$ and $y_3y_4 \in L^4+ V^2$, that $y_1y_3$ is orthogonal 
to $y_1, y_2, y_3$ and $y_4$. Thus 
			$y_1y_3=a x_5$ 
for some $a \in \mathbb{F}$. Replacing $y_1$ by $y_1-a x_1$ we can assume that
\begin{equation}\label{eq:33IB.6} \begin{aligned} 
			y_1y_3=0. 
\end{aligned} \end{equation}
As $x_1y_2=0$ the change in $y_1$ does not affect \eqref{eq:33IB.2}. From the discussion above we know 
that $y_1y_4$ is orthogonal to $y_1, y_2, y_3$ and $y_4$ and thus 
			$y_1y_4= a x_5$ 
for some $a \in \mathbb{F}$. Replacing $y_4, x_2$ by $ y_4-a y_2,$ $x_2+a x_4$, we get 
a new standard basis where
\begin{equation}\label{eq:33IB.7} \begin{aligned} 
			y_1y_4=0. 
\end{aligned} \end{equation}
These changes do not affect \eqref{eq:33IB.3} and \eqref{eq:33IB.4}. As $y_3y_4 \in V^2+L^4$ but 
not in $L^4$ we know that $(y_2y_3, y_4)=r$ for some nonzero $r \in \mathbb{F}$. Suppose also that $(y_3y_4,y_5)=\alpha$. Then $y_3y_4= r x_2 + \alpha x_5$. Replace $x_2$ and $y_5$ by $x_2 +\frac{\alpha}{r} x_5$ and $y_5-\frac{\alpha}{r} y_2$. Then
\begin{equation}\label{eq:33IB.9} \begin{aligned} 
			y_3y_4=r x_2 .
\end{aligned} \end{equation}
The changes do not affect \eqref{eq:33IB.3}. Then consider the triples
\[ (y_2y_3, y_5)=a,\ \  (y_2y_4, y_5)=b .\]
Replacing $y_5, x_4, x_3$ by $y_5-\frac{a}{r} y_4 + \frac{b}{r} y_3,\ x_4 + \frac{a}{r} x_5$ and $x_3 -\frac{b}{r} x_5$ we can assume that
\begin{equation}\label{eq:33IB.8} \begin{aligned} 
			(y_2y_3, y_5)= (y_2y_4, y_5)=0. 
\end{aligned} \end{equation}
 We have then arrived at a presentation where the only nonzero triples are
\[  (x_2y_4, y_5)=1,\ (x_1y_3, y_5)=1,\ (y_1y_2, y_5) =1,\ (y_2y_3, y_4 )=r  .\]
Replacing $x_1, y_1, x_2, y_2, x_3, y_3, x_4, y_4$ by $\frac{1}{r} x_1, r y_1, r x_2, \frac{1}{r} y_2, \frac{1}{r} x_3, r y_3, r x_4, \frac{1}{r} y_4$, we get a unique algebra with presentation:
\begin{align*}
								   (x_2y_4 ,y_5)&=1\\
{\mathcal P}_{10}^{(3,2)}: \quad (x_1y_3 ,y_5)&=1 \\
								   (y_1y_2 ,y_5)&=1\\
								   (y_2y_3 ,y_4)&=1.
\end{align*}
One can easily check that conversely this algebra belongs to the category that we have been studying.
%
%+++-------------+++------------------+++-------------
\subsection{ Algebras where $V^2 \leq L^2$ but $V^2 \nleq L^3$}
Pick our basis such that
\[ V^2 + L^4 = \mathbb{F} x_5+ \mathbb{F} x_4 + \mathbb{F} x_3 +\mathbb{F} y_1 \]
Notice then that
\begin{align*}
 \mathbb{F} x_5+ \mathbb{F} x_4 + \mathbb{F} x_3 +\mathbb{F} x_2 = (V^2+L^4)^\perp \cap L^3 &= (V^2)^\perp \cap L^3 
\end{align*}
is a characteristic ideal of $L$. As $y_3y_4 \in V^2+L^4$ we have that $x_2y_3 \perp y_4$ and $x_2y_4 \perp y_3$. 
Thus $x_2V \leq  \mathbb{F} x_5$. As $x_2 \notin Z(L)$, we must furthermore have that
$ x_2 V= ((V^2)^\perp \cap L^3)V= \mathbb{F} x_5 .$
This implies that the centraliser of $(V^2)^\perp \cap L^3$ in $V$ is a characteristic ideal $W$ of codimension $1$. 
We can choose our basis such that 
\[W=\mathbb{F} x_5+ \mathbb{F} x_4 + \mathbb{F} x_3 +\mathbb{F} x_2 + \mathbb{F} x_1+\mathbb{F} y_1+\mathbb{F} y_2+\mathbb{F} y_3.\]
We now get a chain of characteristic ideals as before
\newpage
\begin{multicols}{2}
%
%&&&&&&&&&&&&&&&&&&&&&&&& Pictuer &&&&&&&&&&&&&&&&&&&&&&&&&&&&&&
\begin{alignat*}{2}
\color{magenta}L^2 \cdot L^2\ &\boxed{
\begin{matrix}
\color{green} x_5
\end{matrix}
}\ \ 
\begin{matrix}
y_5
\end{matrix} \\
\color{magenta}W^\perp \ &\boxed{
\begin{matrix}
\color{green} x_4
\end{matrix}
}
\boxed{\begin{matrix}
\color{green}y_4
\end{matrix}} \ \color{magenta}V\\
\color{cyan}Z(L)=L^4\ &\boxed{
\begin{matrix}
\color{green}x_3
\end{matrix}
}
\boxed{
\begin{matrix}
 \color{green}y_3
\end{matrix}} \ \color{magenta}W\\
\color{magenta}(V^2)^\perp \cap L^3\ &\boxed{
\begin{matrix}
\color{blue} x_2
\end{matrix}}
\boxed{
\begin{matrix}
\color{blue}y_2
\end{matrix}
}\ \color{cyan}Z_3(L) = L^2\\
\color{cyan}Z_2(L)= L^3\ &\boxed{
\begin{matrix}
\color{blue} x_1
\end{matrix}}
\boxed{
\begin{matrix}
\color{blue}y_1
\end{matrix}
}\ \color{magenta} {V^2} +  L^3
\end{alignat*}
\break
%
%&&&&&&&&&&&&&&&&&&&&&&&&&&&&&&&&&&&&&&&&&&&&&&&&&&&&&
\begin{equation*}\label{eq:33IC.1}
\begin{aligned}
L^2 \cdot L^2 & =\mathbb{F} x_5\\
W^\perp &= \mathbb{F} x_5 + \mathbb{F} x_4\\
Z(L) = L^4&=\mathbb{F} x_5+ \mathbb{F} x_4 + \mathbb{F} x_3 \\
(V^2)^\perp \cap L^3 &=\mathbb{F} x_5+ \mathbb{F} x_4 + \mathbb{F} x_3 + \mathbb{F} x_2 \\
Z_2(L) = L^3&=Z(L) +\mathbb{F} x_2 + \mathbb{F} x_1 \\
V^2 + L^3 &=L^3 +\mathbb{F} y_1 \\
Z_3(L) = L^2&=Z_2(L) +\mathbb{F} y_1 + \mathbb{F} y_2 \\
W&=L^3 +\mathbb{F} y_1 +\mathbb{F} y_2 +\mathbb{F} y_3\\
V&=L^3+\mathbb{F} y_1 +\mathbb{F} y_2 +\mathbb{F} y_3+\mathbb{F} y_4
\end{aligned}
\end{equation*}
\end{multicols}
\noindent
As $((V^2)^\perp \cap L^3)W=\{0\}$ and $L^2 \cdot Z_2(L)=\{0\}$, we see that
			\[x_1y_2= x_2y_3 =0. \]
We have also chosen our basis such that
\begin{equation}\label{eq:33IC.2} \begin{aligned} 
			y_1y_2=x_5. 
\end{aligned} \end{equation}
Notice next that $x_2y_3=0$ implies that $x_2y_4$ is orthogonal to $y_3$ and $y_4$ and thus 
			$x_2y_4=r x_5$ 
where $r$ must be nonzero as $x_2 \notin Z(L)$. By replacing $y_4$ and $x_4$ by $r y_4$ and 
$\frac{1}{r} x_4$, we can assume that
\begin{equation}\label{eq:33IC.3} \begin{aligned} 
			x_2y_4=x_5.
\end{aligned} \end{equation}
As $y_3y_4 \in V^2+L^4$ and $y_1y_2=x_5$, we have that $y_1y_4$ is orthogonal to $y_2, y_3, y_4$. Thus 
			$y_1y_4= a x_5$ 
for some $a \in \mathbb{F} $. Replacing  $y_4$, $x_2$ by $y_4 - a y_2$ and $x_2 + a x_4$ we get
\begin{equation}\label{eq:33IC.4} \begin{aligned} 
			y_1y_4=0 .
\end{aligned} \end{equation}
Notice that the change does not affect \eqref{eq:33IC.3}. Next notice similarly that $y_1y_3$ is orthogonal 
to $y_2, y_3, y_4$ and thus 
			$y_1y_3=a x_5$ 
for some $a \in \mathbb{F}$. Replacing $y_3$ and $x_2$ by $y_3 - a y_2$ and $x_2+a x_3$ we get
\begin{equation}\label{eq:33IC.5} \begin{aligned} 
			y_1y_3=0.
 \end{aligned} \end{equation}
Notice that \eqref{eq:33IC.4} is not affected by this change. We know that $x_1y_2=0$. 
The possible nonzero triples involving $x_1$ are then
		\[(x_1y_3, y_4)=r,\  (x_1y_3, y_5)=a,\  (x_1y_4, y_5)=b .\]
Notice that as $y_3y_4 \in (Z(L)+\mathbb{F}y_1) \, \setminus \, Z(L)$ we must have that $r \neq 0$. 
Replace $y_5, x_4, x_3$ by $y_5 - \frac{a}{r}y_4+\frac{b}{r}y_3$, $x_4+\frac{a}{r}x_5$ and $x_3 - \frac{b}{r}x_5$ 
and we get a new standard basis where
			\[x_1y_3 = r x_4,\ \  x_1y_4=-r x_3.\]
Replacing $y_3, x_3$ by $ry_3, \frac{1}{r} x_3$ gives
\begin{equation}\label{eq:33IC.6} \begin{aligned} 
 x_1y_3 = x_4,\ \ x_1y_4 = -x_3. 
 \end{aligned} \end{equation}
It follows that $(y_2y_3, y_4) = (y_3y_4, y_5) = 0$. Suppose $(y_2y_3, y_5) = a,\  (y_2y_4, y_5) = b$. 
Replace $y_3, y_4, x_1$ by $y_3+ a y_1$, $y_4+ b y_1$ and $x_1-a x_3- b x_4$. 
Notice that these changes do not affect the equations above and we now arrive at a unique algebra with presentation:
\begin{align*}
 								   (x_2 y_4,y_5)&=1\\
{\mathcal P}_{10}^{(3,3)}: \quad  (x_1 y_3,y_4)&=1\\
  								   (y_1 y_2,y_5)&=1.
\end{align*}
Calculations show that conversely this algebra belongs to the relevant category. 
There are thus exactly three algebras where $Z(L)$ is isotropic of dimension $3$ and where the dimension of $L^3$ is $5$.
\begin{Proposition}\label{pro33I}
There are exactly three nilpotent SAA's of dimension $10$ that have an isotropic center of dimension $3$ and where $\mbox{dim\,} L^3 = 5$. 
These are given by the presentations:
\begin{align*}
{\mathcal P}_{10}^{(3,1)}: &\quad (x_2y_4,y_5)=1,\ (x_1y_3,y_5)=1,\ (y_1y_2,y_5)=1 .\\
{\mathcal P}_{10}^{(3,2)}: &\quad (x_2y_4  ,y_5)=1,\ (x_1y_3 ,y_5)=1,\ (y_1y_2,y_5)  =1,\ (y_2y_4  ,y_3 )=1  .\\
{\mathcal P}_{10}^{(3,3)}: &\quad (x_2 y_4,y_5)=1,\ (x_1  y_3,y_4)=1,\ (y_1 y_2  ,y_5)=1. 
\end{align*} 
\end{Proposition}
%
%---------------+++++++--------------------
\section{The algebras where $\mbox{dim\,} L^3 = 6 $}
Here we are thus assuming that
\[ L^3 =  \mathbb{F}x_5 + \mathbb{F}x_4 + \mathbb{F}x_3+ \mathbb{F}x_2 + \mathbb{F}x_1 + \mathbb{F}y_1. \]
\begin{Lemma} We have $\mbox{dim\,} L^4 = 4$. \end{Lemma}
\begin{proof} By Lemma \ref{zlinl4}, we know that $Z(L) \leq L^4$ and we also have that $L^4 \leq \mathbb{F}x_5 + \mathbb{F}x_4 + \mathbb{F}x_3+ \mathbb{F}x_2 $. 
Thus if the dimension of $L^4$ is not $4$, then $L^4 = Z(L) = \mathbb{F}x_5 + \mathbb{F}x_4 + \mathbb{F}x_3$ and $Z_3(L) = (L^4)^\perp = L^3 +  \mathbb{F}y_2$. 
As $Z_3(L) \cdot L^3 = \{0\}$, it follows that $x_1y_2 = y_1y_2 = 0$ and $L^2$ is abelian. Hence we get the contradiction that $L^3 \leq Z(L)$.
\end{proof}
\noindent
It follows that we have $L^4 =  \mathbb{F}x_5 + \mathbb{F}x_4 + \mathbb{F}x_3+ \mathbb{F}x_2 $.
\begin{Lemma}\label{lmaodd} We have $\mbox{dim\,} L^5 = 2$.
\end{Lemma}
\begin{proof}
We have an alternating form
\[ \phi : L/L^2 \times L/L^2 \longrightarrow \mathbb{F}  \]
given by $\phi (\bar{y},\bar{z}) = (x_2 y, z)$.
As $L/L^2$ has odd dimension we know that the isotropic part must be non-trivial. Thus we can then choose our standard basis such that 
			$(x_2 y_3, y_4)= (x_2 y_3, y_5)=0$ and thus $x_2y_3 = 0$. 
It follows that $L^5 = \mathbb{F}x_2 (y_3 + y_4 + y_5) = \mathbb{F}x_2y_4 + \mathbb{F}x_2y_5$ and thus of dimension at most $2$. 
As $L^4 \not \leq Z(L)$ we have $\mbox{dim\,} L^5 >0$ and as we know by Proposition \ref{pro2.10.1gth} that  $\mbox{dim\,} L^5 \neq 1$ we must have that $\mbox{dim\,} L^5=2$.
\end{proof}
\noindent
We thus have determined the lower and upper central series of $L$. We have 
\newpage
\begin{multicols}{2}
%&&&&&&&&&&&&&&&&&&&&&&&& Pictuer &&&&&&&&&&&&&&&&&&&&&&&&&&&&&&
\begin{alignat*}{2}
\color{cyan}L^5\ &\boxed{
\begin{matrix}
\color{green} x_5\\\color{green} x_4
\end{matrix}
}\ \ 
\begin{matrix}
y_5\\y_4
\end{matrix}\\
\color{cyan}Z(L)\ &\boxed{
\begin{matrix}
\color{green}x_3
\end{matrix}
}
\boxed{
\begin{matrix}
\color{red}y_3
\end{matrix}
}\ \color{cyan}Z_4(L)\\
\color{cyan}L^4=Z_2(L)\ &\boxed{
\begin{matrix}
\color{cyan} x_2
\end{matrix}}
\boxed{
\begin{matrix}
\color{magenta}y_2
\end{matrix}
}\ \color{cyan}L^2\\
&\boxed{
\begin{matrix}
\color{blue} x_1
\end{matrix}}
\boxed{
\begin{matrix}
\color{blue}y_1
\end{matrix}
}\ \color{cyan}L^3=Z_3(L)
\end{alignat*}
\break
\begin{align*}
L^5 &= \mathbb{F} x_5 + \mathbb{F} x_4\\
Z(L)&=\mathbb{F} x_5+ \mathbb{F} x_4 + \mathbb{F} x_3 \\
Z_2(L) = L^4&=\mathbb{F} x_5+ \mathbb{F} x_4 + \mathbb{F} x_3 + \mathbb{F} x_2 \\
Z_3(L)= L^3&=Z(L)+\mathbb{F} x_2 + \mathbb{F} x_1+\mathbb{F} y_1 \\
L^2&=L^3 + \mathbb{F} y_2 \\
Z_4(L)&=L^3 +\mathbb{F} y_2 +\mathbb{F} y_3
\end{align*} 
\end{multicols}
\noindent 
Notice that $x_2 \in L^4$ and $y_3 \in Z_4(L)$ and thus $x_2y_3=0$. Also
\[ L^3L^2 = \mathbb{F}  x_1y_2 + \mathbb{F}  y_1y_2 \leq Z(L). \]
Furthermore $x_1y_2$ and $y_1y_2$ are linearly independent. To see this we argue by contradiction and suppose that $0 = a x_1 y_2 + b y_1 y_2$ for some $a, b \in \mathbb{F} $
 where not both $a, b$ are zero. Then 
			$(a x_1 + b y_1) L \leq Z(L)$ 
that would give us the contradiction that $a x_1 + b y_1 \in Z_2(L)$. \\ \\
We thus have that $L^3L^2$ is a $2$-dimensional subspace of $Z(L)$ and we consider two possible cases namely $L^3L^2 = L^5$ and $L^3L^2 \neq L^5$. We consider the latter first.
%
%+++++++-----------------++++++++---------------
\subsection{ Algebras where $L^3 \cdot L^2 \neq L^5$}
Here $L^3L^2 \cap L^5$ is one dimensional and we can choose our standard basis such that $L^3L^2 \cap L^5 = \mathbb{F}x_5$. 
In order to clarify the structure further we introduce the following isotropic characteristic ideal of dimension $5$:
\[ U = \{ x \in L^3:\, x L^2 \leq L^3L^2 \cap L^5 \} .\]
Now $L^3L^2$ is of dimension $2$ and $L^4L^2=0$ and thus $U$ is of codimension $1$ in $L^3$ and contains $L^4$. We can
 thus choose our standard basis such that $U=\mathbb{F} x_5+ \mathbb{F} x_4 + \mathbb{F} x_3 +\mathbb{F} x_2 +
 \mathbb{F} x_1$. We thus have the following picture\\
%
%
%&&&&&&&&&&&&&&&&&&&&&&&& Pictuer &&&&&&&&&&&&&&&&&&&&&&&&&&&&&&
\begin{alignat*}{2}
\color{magenta}L^5 \cap L^3L^2\ &\boxed{
\begin{matrix}
\color{green} x_5
\end{matrix}
}\ \ 
\begin{matrix}
y_5
\end{matrix} \\
\color{cyan}L^5 \ &\boxed{
\begin{matrix}
\color{green} x_4
\end{matrix}
}
\boxed{\begin{matrix}
\color{green}y_4
\end{matrix}} \ \color{magenta}(L^5 \cap L^3L^2)^\perp\\
\color{cyan}Z(L)\ &\boxed{
\begin{matrix}
\color{green}x_3
\end{matrix}
}
\boxed{
\begin{matrix}
 \color{green}y_3
\end{matrix}} \ \color{cyan}Z_4(L)\\
\color{cyan}L^4=Z_2(L)\ &\boxed{
\begin{matrix}
\color{blue} x_2
\end{matrix}}
\boxed{
\begin{matrix}
\color{blue}y_2
\end{matrix}
}\ \color{cyan}L^2\\
\color{magenta}U\ &\boxed{
\begin{matrix}
\color{blue} x_1
\end{matrix}}
\boxed{
\begin{matrix}
\color{blue}y_1
\end{matrix}
}\ \color{cyan}L^3=Z_2(L)
\end{alignat*}
%&&&&&&&&&&&&&&&&&&&&&&&&&&&&&&&&&&&&&&&&&&&&&&&&&&&&&
%
Notice that $UZ_4(L)= \mathbb{F} x_1y_2+  \mathbb{F} x_1 y_3 =  \mathbb{F} x_5+  \mathbb{F} x_1 y_3$, where $x_1y_3 \in L^5$. 
Again we consider two possible cases.
%
%+++++++-----------------++++++++-------------------
\subsubsection{I. Algebras where $ UZ_4(L)$ is $1$-dimensional }
Here $UZ_4(L)= x_1 Z_4(L)=\mathbb{F} x_5$ and there is a characteristic subspace $V$ of codimension $1$ in $Z_4(L)$ 
that contains $L^3$ given by the formula
			 $$V=\{ x \in Z_4(L):\, Ux =0\}.$$
%
%As $x_1y_2 \neq 0$, 
We can then choose our standard basis such that
\[ V = L^3 + \mathbb{F} y_3 = U+ \mathbb{F} y_1 + \mathbb{F} y_3. \]
Notice that in particular $x_1y_3=0$. From this we also get a $1$-dimensional characteristic subspace 
			$V^2= \mathbb{F} y_1 y_3$. 
Notice that $(y_1y_3, y_2) \neq 0$ as otherwise $y_1y_2 \in L^5 \cap L^3L^2$ that contradicts our assumption that 
$L^3L^2 \neq L^5$. Thus $y_1y_3 \in L^4 \, \setminus \, Z(L)$ and we can choose our standard basis such that 
			$\mathbb{F}y_1 y_3 = \mathbb{F} x_2$. 
In fact it is not difficult to see that with the data we have acquired so far we can choose our standard basis such that
\begin{equation}\label{eq:33II-I.1} \begin{aligned} 
			x_1y_2 = x_5,\ y_1y_2 = x_3,\ x_1y_3 = 0,\ y_1y_3 = -x_2. 
\end{aligned} \end{equation}
This deals with all triple values apart from
\begin{alignat*}{4}
(x_1y_4, y_5)&=a,&\quad  (y_2y_3, y_4)&=c,&\quad  (y_2y_4, y_5)&=e,&\quad  (x_2y_4, y_5)&=r,\\
(y_1y_4, y_5)&=b,&\quad  (y_2y_3, y_5)&=d,&\quad  (y_3y_4, y_5)&=f, &\quad  \mbox{}
\end{alignat*}
Notice that $r \neq 0$ as $x_2y_3 = 0$ but $x_2 \not \in Z(L)$. 
We will show that we can choose a new standard basis so that the values of $a=b=c=d=e=f=0$ and $r=1$. 
First by replacing $x_2, x_1, y_4$ by $x_2 + a x_4, x_1+c x_4, y_4- a y_2 - c y_1$ we can assume that $a = c = 0$, that is 
\begin{equation}\label{eq:33II-I.2} \begin{aligned} 
			(x_1 y_4, y_5) = 0,\ (y_2y_3, y_4)=0. 
\end{aligned} \end{equation}
Inspection shows that \eqref{eq:33II-I.1} still holds under these changes. 
Next replacing $y_1, y_2$ by $ y_1-(b/r)x_2, y_2- (b/r) x_1 $ we can assume that $b=0$ and thus 
\begin{equation}\label{eq:33II-I.3} \begin{aligned} 
			(y_1y_4, y_5)=0. 
\end{aligned} \end{equation}
Again \eqref{eq:33II-I.1} and \eqref{eq:33II-I.2} are not affected by the changes. 
We continue in this manner always making sure that the previously established identities still hold. 
We first replace 
			$x_1, y_2, y_5$ by $x_1+ d x_5, y_2 - (e/r) x_2, y_5 - d y_1$
that gives
\begin{equation}\label{eq:33II-I.4} \begin{aligned} 
			(y_2y_3, y_5)=0,\ (y_2y_4, y_5)=0, 
\end{aligned} \end{equation}
then replace 
			$y_2, y_3$ by $y_2-(f/r) x_3, y_3 - (f/r) x_2$ 
that gives furthermore that 
\begin{equation}\label{eq:33II-I.5} \begin{aligned} 
			(y_3y_4, y_5)=0. 
\end{aligned} \end{equation}
Finally replacing $x_4, y_4$ by $rx_4, (1/r) y_4$ gives us a presentation with $r=1$. 
Thus we have that we get a unique algebra.
\begin{Proposition}\label{pro33I}
There is a unique nilpotent SAA $L$ of dimension $10$ with an isotropic center of dimension $3$ and where $\mbox{dim\,} L^3 =  6$ 
that has the further properties that $L^3L^2 \neq L^5$ and $\mbox{dim\,} UZ_4(L)=1$. This algebra is given by the presentation
\begin{align*}
{\mathcal P}_{10}^{(3,4)}: \  (x_1y_2,y_5)=1,\ (y_1y_2,y_3)=1,\ (x_2y_4,y_5)=1. 
\end{align*} 
\end{Proposition}
\begin{Remark} 
As before, inspection shows that the algebra with the presentation above satisfies all the properties listed.
\end{Remark}
%
%+++++++-----------------++++++++---------------
%
\subsubsection{II. Algebras where $UZ_4(L)$ is $2$-dimensional }
Here we can pick our standard basis such that $UZ_4(L) = \mathbb{F} x_5 + \mathbb{F} x_4$. 
As $L^3L^2 \neq L^5$ we know that $(y_1y_2, y_3) \neq 0$ and from this one sees that $L^3Z_4(L)=L^4$. 
Furthermore it is not difficult to see that we can choose our standard basis such that
\begin{equation}\label{eq:33II-I.6} \begin{aligned} 
			x_1y_2=x_5,\ x_1y_3=x_4,\ y_1y_2=x_3,\ y_1y_3=-x_2 .
 \end{aligned} \end{equation}
In order to clarify the structure further we are only left with the triple values
\begin{alignat*}{4}
(x_1y_4, y_5)&=a,&\quad  (y_2y_3, y_4)&=c,&\quad  (y_2y_4, y_5)&=e,&\quad   (x_2y_4, y_5)&=r,\\
(y_1y_4, y_5)&=b,& \quad (y_2y_3, y_5)&=d,&\quad  (y_3y_4, y_5)&=f, & \quad  \mbox{}
\end{alignat*}
We show as before that one can change the basis such that $a=b=c=d=e=f=0$ and $r=1$. First we replace 
			$x_1, y_2$ by $x_1-(a/r)x_2, y_2=y_2 + (a/r) y_1$ 
that gives furthermore
\begin{equation}\label{eq:33II-I.7} \begin{aligned} 
			(x_1y_4, y_5)=0, 
\end{aligned} \end{equation}
and then we replace 
			$ y_1, y_2$ by $y_1-(b/r) x_2, y_2 - (b/r) x_1$  
that gives
\begin{equation}\label{eq:33II-I.71} \begin{aligned}
			(y_1y_4, y_5)=0 . 
\end{aligned} \end{equation}
Next we replace 
			$x_1, y_4, y_5$ by $x_1+c x_4 + d x_5, y_4 - c y_1, y_5 - d y_1$ 
that allows us to further assume that
\begin{equation}\label{eq:33II-I.8} \begin{aligned} 
			(y_2y_3, y_4)=0,\ (y_2y_3, y_5)=0 
 \end{aligned} \end{equation}
and then 
			$y_2, y_3$ by $y_2 - (e/r) x_2 - (f/r) x_3, y_3 - (f/r) x_2$ 
that gives
\begin{equation}\label{eq:33II-I.9} \begin{aligned} 
			(y_2y_4, y_5)=0,\ (y_3y_4, y_5)=0 . 
\end{aligned} \end{equation}
Finally changing $x_2, x_3, x_4, x_5, y_2, y_3, y_4, y_5$ by $(1/r)x_2, r x_3, (1/r) x_4, r x_5, r y_2,(1/r) y_3,$
$ r y_4,$
$ (1/r) y_5$ gives us $(x_2 y_4, y_5) = 1$ and thus
we see again that we have a unique algebra.
\begin{Proposition}\label{pro33I}
There is a unique nilpotent SAA $L$ of dimension $10$ with an isotropic center of dimension $3$ and where $\mbox{dim\,} L^3 =  6$ 
that has the further properties that $L^3L^2 \neq L^5$ and $\mbox{dim\,} UZ_4(L) = 2$. This algebra can be given by the presentation
\begin{align*}
{\mathcal P}_{10}^{(3,5)}: \ (x_1y_2,y_5) = 1,\ (y_1y_2,y_3) = 1,\ (x_1y_3, y_4) = 1,\ (x_2y_4,y_5) = 1. 
\end{align*} 
\end{Proposition}
%
%
%++----------------------------------------------++---
%
\newpage
\section{Algebras where $L^3L^2=L^5$}
\begin{multicols}{2}
%
%&&&&&&&&&&&&&&&&&&&&&&&& Pictuer &&&&&&&&&&&&&&&&&&&&&&&&&&&&&&
\begin{alignat*}{2}
\color{cyan}L^5\ &\boxed{
\begin{matrix}
\color{green} x_5\\\color{green} x_4
\end{matrix}
}\ \ 
\begin{matrix}
y_5\\y_4
\end{matrix}\\
\color{cyan}Z(L)\ &\boxed{
\begin{matrix}
\color{green}x_3
\end{matrix}
}
\boxed{
\begin{matrix}
\color{red}y_3
\end{matrix}
}\ \color{cyan}Z_4(L)\\
\color{cyan}L^4=Z_2(L)\ &\boxed{
\begin{matrix}
\color{cyan} x_2
\end{matrix}}
\boxed{
\begin{matrix}
\color{magenta}y_2
\end{matrix}
}\ \color{cyan}L^2\\
&\boxed{
\begin{matrix}
\color{blue} x_1
\end{matrix}}
\boxed{
\begin{matrix}
\color{blue}y_1
\end{matrix}
}\ \color{cyan}L^3=Z_3(L)
\end{alignat*}
\break
%&&&&&&&&&&&&&&&&&&&&&&&&&&&&&&&&&&&&&&&&&&&&&&&&&&&&&
%
\begin{align*}
L^5 &= \mathbb{F} x_5 + \mathbb{F} x_4\\
Z(L)&=\mathbb{F} x_5+ \mathbb{F} x_4 + \mathbb{F} x_3 \\
Z_2(L) = L^4&=\mathbb{F} x_5+ \mathbb{F} x_4 + \mathbb{F} x_3 + \mathbb{F} x_2 \\
Z_3(L) = L^3&=Z(L) +\mathbb{F} x_2 + \mathbb{F} x_1+\mathbb{F} y_1 \\
L^2&=L^3 + \mathbb{F} y_2 \\
Z_4(L)&=L^3 +\mathbb{F} y_2 +\mathbb{F} y_3
\end{align*}
\end{multicols}
\noindent
Here we are assuming that $L^5=L^3L^2= \mathbb{F} x_1y_2+  \mathbb{F}  y_1y_2$ and thus in particular we know that 
$x_1y_2, y_1y_2$ is a basis for $L^5$. We will now introduce some linear maps that will help us in understanding the structure. 
Consider first the linear maps
%
%$$\begin{array}{ll}
%
\begin{align*}
\phi &:\, L^3/L^4  \longrightarrow L^5,\  \bar{u}=u+L^4 \longmapsto u \cdot y_2 \\
\psi &:\, L^3/L^4  \longrightarrow L^5,\  \bar{u}=u+L^4  \longmapsto u \cdot y_3 .
\end{align*}
%\end{array}$$
%
As $L^4 Z_4(L) = \{0\}$, these maps are well defined. 
As $L^3L^2=L^5$ we also know that $\phi$ is bijective. We thus have the linear map 
$$\tau= \psi \phi^{-1} :\, L^5 \longrightarrow L^5.$$
It is the map $\tau$ that will be our key towards understanding the structure of the algebra.
\begin{Lemma}\label{lma130}
The minimal polynomial of $\tau = \psi \phi^{-1}$ must be of degree $2$.
\end{Lemma}
\begin{proof} 
We argue by contradiction and suppose that $\tau=\lambda \mbox{id} $. Replacing $y_3, x_2$ by $y_3- \lambda y_2, x_2+\lambda x_3$ 
gives us a new standard basis where $\tau =0 $. Pick our standard basis such that $\bar{x_1} = x_1 + L^4 = \phi^{-1}(x_4)$ and 
$\bar{y_1} = y_1 + L^4 = \phi^{-1}(x_5)$. We then have
\begin{align*}
			 x_1y_2=x_4,\ y_1y_2=x_5,\ x_1 y_3=0,\ y_1 y_3=0. 
\end{align*}
Now $y_2y_3 \perp x_1, y_1, y_2, y_3$ and thus 
			$$y_2y_3=a x_4+b x_5 $$
for some $a, b \in \mathbb{F}$. Replacing $y_3$ by $y_3+a x_1+b y_1$, $x_1$ by $x_1-b x_3$ and 
$y_1$ by $y_1+a x_3$, we can assume that $y_2 y_3=0$. \\ \\
Now suppose that $(y_3y_4,y_5)=a$ and $(x_2 y_4, y_5)=b$. Notice that $b \neq 0$ as $x_2 \not \in Z(L)$ and $x_2y_3 =0$. 
Replace $y_3, y_2$ by $y_3-(a/b)x_2,\ y_2-(a/b)x_3$
and we get a new standard basis where all the previous identities hold but also $(y_3 y_4, y_5)=0$. 
We thus get the contradiction that $y_3 \in Z(L)$. 
\end{proof}
\noindent
Notice next that if we have an alternative standard basis $\tilde{x_1}, \tilde{x_2}, \ldots, \tilde{y_5}$, then $\tilde{y_2} = c y_2+u$ and $\tilde{y_3} = a y_3 + b y_2+v$ 
where $a, c \neq 0$ and where $u,v \in L^3$. 
If the minimal polynomial of $\tau$ with respect to the old basis is $f(t)$ then the minimal polynomial with respect to the new basis is a 
		multiple of $f((c/a)(t-(b/c))$.
In particular we have the following possible distinct scenarios that do not depend on what standard basis we choose.\\ \\
{\bf A}. The minimal polynomial of $\tau$  has two distinct roots in $\mathbb{F}$.\\ \\
{\bf B}. The minimal polynomial of $\tau$ has a double root in $\mathbb{F}$. \\ \\
{\bf C}. The minimal polynomial of $\tau$ is irreducible in $\mathbb{F}[t]$.
%
%++++++------------------------+++++------------------
%
\subsection{Algebras of type $A$.}
Suppose the two distinct roots of the minimal polynomial of $\tau = \psi \phi^{-1}$ are $\lambda$ and $\mu$. 
Pick some eigenvectors $x_4$ and $x_5$ with respect to the eigenvalues $\lambda$ and $\mu$ respectively. Thus
 \begin{eqnarray*} 
\psi \phi^{-1}(x_4) & = &  \lambda x_4,\\ 
\psi \phi^{-1}(x_5) & = &  \mu x_5 .
 \end{eqnarray*}
Replacing $y_3$, $x_2$ by $y_3 - \lambda y_2$, $x_2 + \lambda x_3$ we see that
$\psi \phi^{-1}(x_4) = 0$
and we can assume that $\lambda =0$. Then replace $y_3, x_3$ by $(1/\mu)y_3, \mu x_3$ and we get that 
$\psi \phi^{-1}(x_5) = x_5$ and we can now assume that $\mu = 1$.\\\\
We would like to pick our standard basis such that 
	$\bar{x_1} = x_1 + L^{4} = \phi^{-1}(x_4)$ and $\bar{y_1} = y_1 + L^{4} = \phi^{-1}(x_5)$ .
The only problem here is that we need $(x_1, y_1) = 1$ but this can be easily arranged. If $(x_1,y_1) = \sigma $ then we just need to replace
$y_1, x_5, y_5$ by $(1/\sigma) y_1, (1/\sigma) x_5, \sigma y_5$. 
We have thus seen that we can choose our standard basis such that
\begin{equation}\label{eq:33II-II.A.1} \begin{aligned} 
	x_1y_2 = x_4,\ y_1y_2 = x_5,\ x_1y_3 = 0,\ y_1y_3 = x_5. 
\end{aligned} \end{equation}
Recall also that $x_2y_3=0$ since $L^4Z_4(L)=0$. In order to fully determine the structure of the algebra we 
are only left with the following triple values
\begin{alignat*}{4}
(x_1y_4, y_5)&=a,&\quad  (y_2y_3, y_4)&=c,&\quad  (y_2y_4, y_5)&=e,& \quad
 (x_2y_4, y_5)&=r,\\
(y_1y_4, y_5)&=b,&\quad  (y_2y_3, y_5)&=d,& \quad (y_3y_4, y_5)&=f. &\quad 
\mbox{} 
\end{alignat*}
Notice that $r \neq 0$ as $x_2y_3 = 0$ but $x_2 \not \in Z(L)$. We will show that we can choose our basis 
such that $a=b=c=d=e=f=0$. 
First replace $x_1, y_1, y_2$ by $x_1-(a/r) x_2 , y_1-(b/r) x_2, y_2+(a/r) y_1 - (b/r) x_1$ and we see that we can assume that $a=b=0$, 
that is 
\begin{eqnarray}\label{eq:33II-II.A.2} 
		(x_1 y_4,y_5) = (y_1 y_4,y_5)=0. 
 \end{eqnarray}
Inspection shows that \eqref{eq:33II-II.A.1} remains same under these changes. Then replacing 
$x_1, y_1, y_3$ by $x_1-d x_3, y_1+c x_3, y_3+c x_1+d y_1$ we can furthermore assume that 
\begin{eqnarray}\label{eq:33II-II.A.3} 
			(y_2y_3,y_4) = (y_2y_3,y_5)=0. 
\end{eqnarray}
Finally replacing $x_1, y_2, y_3, y_4$ by $x_1-e x_4, y_2- ((f-e)/r) x_3, y_3-((f-e)/r) x_2, y_4+ e y_1$ 
and we can also assume that
\begin{eqnarray}\label{eq:33II-II.A.4} 
			(y_2y_4,y_5)=(y_3y_4,y_5)=0. 
\end{eqnarray}
We have thus seen that $L$ has a presentation of the form 
${\mathcal P}_{10}^{(3,6)}(r)$ as described in the next proposition. 
\begin{Proposition}\label{pro33IAA}
Let $L$ be a nilpotent SAA of dimension $10$ with an isotropic center of dimension $3$ that has the further properties that $\mbox{dim\, }L^3 =6$, $L^3L^2 = L^5$
and $L$ is of type $A$. Then $L$ has a presentation of the form
\begin{align*}
{\mathcal P}_{10}^{(3,6)}(r): \ \ (x_2y_4,y_5) = r,\ (x_1y_2,y_4) = 1,\ (y_1y_2,y_5) = 1,\ (y_1y_3, y_5) = 1
\end{align*} 
where $r \neq 0$. Furthermore the presentations ${\mathcal P}_{10}^{(3,6)}(r)$ and 
${\mathcal P}_{10}^{(3,6)}(s)$ describe the same algebra 
if and only if $s/r \in (\mathbb{F}^*)^{3}$.
\end{Proposition}
\begin{proof}
We have already seen that all such algebras have a presentation
of the form ${\mathcal P}_{10}^{(3,6)}(r)$ for some $0 \neq  r \in {\mathbb F}$.
Straightforward calculations show that conversely any algebra with such a presentation has the properties stated in the Proposition. 
It remains to prove the isomorphism property. To see that the property is sufficient, suppose that
we have an algebra $L$ with presentation ${\mathcal P}_{10}^{(3,6)}(r)$ with
respect to some given standard basis. Let $s$ be any element in ${\mathbb F}^{*}$ such that $s/r=b^3 \in ({\mathbb F}^*)^{3}$. 
Replace the basis with a new standard basis $\tilde{x_{1}}, \ldots, \tilde{y_{5}}$ where $\tilde{x_{1}}=x_1,\ \tilde{y_{1}}=y_1,\  
\tilde{x_{2}}=bx_2,\ \tilde{y_{2}}=(1/b)x_2,\ \tilde{x_{3}}=bx_3,\ \tilde{y_{3}}=(1/b)y_3,\ \tilde{x_{4}}=(1/b)x_4,\ \tilde{y_{4}}= by_4,\  
\tilde{x_5}=(1/b)x_5,\ \tilde{y_5}=by_5$. 
Direct calculations show that $L$ has presentation ${\mathcal P}_{10}^{(3,6)}(s)$ with respect to the new basis.\\ \\
It remains to see that the property is necessary. Consider again an algebra $L$ with presentation ${\mathcal P}_{10}^{(3,6)}(r)$
and suppose that $L$ has also a presentation ${\mathcal P}_{10}^{(3,6)}(s)$ with respect to some other standard basis $\tilde{x_{1}}, \cdots, \tilde{y_{5}}$. 
We want to show that $s/r \in  ({\mathbb F}^{*})^3$. 
We know that $L = {\mathbb F} \tilde{y_{5}} + {\mathbb F} \tilde{y_{4}} + Z_4(L) = {\mathbb F} y_{5} + {\mathbb F} y_{4} +Z_4(L) $. Thus
\begin{eqnarray*}
\tilde{y_4}&=& ay_4+by_5+u \\
\tilde{y_5}&=& cy_4+dy_5+v
\end{eqnarray*}
for some $u, v \in Z_4(L)$ and $a, b, c, d \in {\mathbb F}$ where $ad-bc \neq 0$. 
As $L^3L^2 = L^5 \perp Z_4(L)$ and as $Z_4(L)L^4 =0 $ we have $(Z_4(L) L^2, L^3) = (Z_4(L)L, L^4) =0$ and 
thus $Z_4(L)L^2 \leq (L^3)^\perp = L^4$ and $Z_4(L)L \leq (L^4)^\perp = L^3$. It follows that
\begin{eqnarray*}
\tilde{y_4}\tilde{y_5}\tilde{y_5}  &=& (ay_4 + by_5)(cy_4 + dy_5)(cy_4 + dy_5)+ w\\
\tilde{y_5} \tilde{y_4} \tilde{y_4} &=& (cy_4 + dy_5)(ay_4 + by_5)(ay_4 + by_5)+z
\end{eqnarray*}
for some $w, z \in L^4$. Using the fact that $(L^4, L^3) = 0$, as $L^6 = 0$, we then see that
\[s^2=(\tilde{y_4} \tilde{y_5}\tilde{y_5}, \tilde{y_5}\tilde{y_4}\tilde{y_4}) = r^2(ad-bc)^3.\]
Hence $s/r \in ({\mathbb F}^*)^3$.
\end{proof}
\begin{Remark}
Notice that it follows that we have only one algebra if $({\mathbb F}^*)^{3} = {\mathbb F}^*$. 
This includes all fields that are algebraically closed as well as ${\mathbb R}$.
For a finite field of order $p^n$ there are $3$ algebras if $3 | p^n-1$ but otherwise one.
For ${\mathbb Q}$ there are infinitely many algebras.
\end{Remark}
%++++++------------------------+++++----------------
\subsection{Algebras of type $B$}
Suppose that the double root of the minimal polynomial of $\tau = \psi \phi^{-1}$ is $\lambda $. We can then have a basis $x_4$, $x_5$ for $L^5$ such that
\begin{eqnarray*}
\psi \phi^{-1}(x_4) &=& \lambda x_4 \\
\psi \phi^{-1}(x_5) &=& \lambda x_5+x_4 .
\end{eqnarray*}
If we replace $y_3, x_2$ by $y_3-\lambda y_2$, $x_2+\lambda  x_3$ then we can furthermore assume that $\lambda = 0$.
We want to pick our standard basis such that $\bar{x_1} = x_1 + L^{4} = \phi^{-1}(x_4)$ and $\bar{y_1} = y_1 + L^{4} = \phi^{-1}(x_5)$.
Again the only problem is to arrange for $(x_1, y_1)=1$.
But if $(x_1,y_1) = \sigma $ then we replace $x_5, x_3, y_1, y_3, y_5$ by $(1/\sigma) x_5, (1/\sigma) x_3, (1/\sigma) y_1 , \sigma y_3, \sigma y_5 $ and that gives  
		$(x_1, y_1)=1$.
We have thus seen that we can choose our standard basis such that
\begin{equation}\label{eq:33II-II.B.1} \begin{aligned} 
	x_1y_2=x_4,\ y_1y_2= x_5,\ x_1y_3=0,\ y_1y_3=x_4. 
\end{aligned} \end{equation}
As before we have furthermore $x_2y_3 = 0$ and we are only left with the following triple values 
\begin{alignat*}{4}
(x_1y_4, y_5)&=a,&\quad  (y_2y_3, y_4)&=c,&\quad  (y_2y_4, y_5)&=e,&\quad
 (x_2y_4, y_5)&=r,\\
(y_1y_4, y_5)&=b,&\quad  (y_2y_3, y_5)&=d,&\quad  (y_3y_4, y_5)&=f .&\quad
\mbox{} 
\end{alignat*}
Notice that $r \neq 0$ as $x_2y_3 = 0$ but $x_2 \not \in Z(L)$. We will show that we can choose a new standard basis so
that all the other values are zero. 
First by replacing $x_1, y_1, y_2$ by $x_1-(a/r) x_2, y_1-(b/r) x_2, y_2+(a/r) y_1 - (b/r) x_1$, we can assume that $a=b=0$. That is 
\begin{eqnarray}\label{eq:33II-II.B.2} 
			(x_1 y_4,y_5)=(y_1 y_4,y_5)=0. 
 \end{eqnarray}
These changes do not affect \eqref{eq:33II-II.B.1}. Then replace $x_1, y_1, y_3$ by $x_1-d x_3, y_1+c x_3, y_3+c x_1+d y_1$ we can
furthermore assume that 
\begin{eqnarray}\label{eq:33II-II.B.3}
			(y_2y_3,y_4)= (y_2y_3,y_5)=0. 
\end{eqnarray}
Finally replacing $x_1, y_4, y_5$ by 
$x_1-e x_4 + f x_5, y_4+ e y_1, y_5 - f y_1$ further allows us to assume that
\begin{eqnarray}\label{eq:33II-II.B.4}
			(y_2y_4,y_5)= (y_3 y_4, y_5)=0. 
 \end{eqnarray}
We thus arrive at a presentation of the form ${\mathcal P}_{10}^{(3,7)}(r)$ as given in the next proposition
\begin{Proposition}\label{pro33IB}
Let $L$ be a nilpotent SAA of dimension $10$ with an isotropic center of dimension $3$ that has the further properties that $\mbox{dim\,} L^3 = 6$, $L^3L^2 = L^5$
 and $L$ is of type $B$. Then $L$ has a presentation of the form
\begin{align*}
{\mathcal P}_{10}^{(3,7)}(r): \ \  (x_2y_4,y_5) = r,\ (x_1y_2,y_4) = 1,\ (y_1y_2,y_5) = 1,\ (y_1y_3, y_4) = 1
\end{align*} 
where $r \neq 0$. Furthermore the presentations ${\mathcal P}_{10}^{(3,7)}(r)$ and ${\mathcal P}_{10}^{(3,7)}(s)$ describe the same algebra 
if and only if $s/r \in (\mathbb{F}^*)^{3}$.
\end{Proposition}
\begin{proof} 
Similar to the proof of Proposition \ref{pro33IAA}.
\end{proof}
\subsection{Algebras of type $C$}
It turns out to be useful to consider the cases
$\mbox{char\,} \mathbb{F} \neq 2$ and 
$\mbox{char\,} \mathbb{F}=2$ separately.
\subsubsection{\underline{a. The algebras where $\mbox{char\,} \mathbb{F} \neq 2$}}
Suppose the minimal polynomial of $\tau = \psi \phi^{-1}$ is $t^2 + at + b$ with respect to some $y_2, y_3$. 
Replacing $y_3$ by $y_3 + (a/2) y_2$, one gets a minimal polynomial of the form $t^2 - s$ with $s \not \in  \mathbb{F}^{2}$.
\noindent
\begin{Remark}\label{lma133}
Let $\tilde{y_3} = \alpha y_3 + u$ where $\alpha \neq 0$ and $u \in L^2$. For the minimal polynomial of $\tau$ to have trivial linear term we must have 
$u \in L^3$. Thus $\mathbb{F} y_3 + L^3$ is a characteristic subspace of $L$.
\end{Remark}
\noindent
Pick any $0 \neq x_5 \in L^5$ and let 
			$x_4=\psi\phi^{-1}(x_5)$. Then $\psi\phi^{-1}(x_4) = s x_5$.
We want to pick our standard basis such that $\phi^{-1}(x_4)=x_1+L^4$, $\phi^{-1}(x_5)=y_1+L^4$. For this to work out we need $(x_1,y_1)=1$.
Again this can be easily arranged. If $(x_1,y_1)=\sigma$, then we replace $x_5, y_1, y_3$ by $(1/\sigma) x_5,$ $(1/\sigma)y_1,$ $\sigma y_3$
and we get $(x_1,y_1)=1$ and $\psi\phi^{-1}(x_4)=(\sigma^2 s) x_5$. 
We have thus seen that we can choose our standard basis such that
\begin{eqnarray}\label{eq:33II-II.C.1} 
		 x_1y_2=x_4,\ y_1y_2= x_5,\ x_1y_3 = s x_5,\ y_1y_3=x_4
\end{eqnarray}
for some $s \not \in \mathbb{F}^{2}$. In order to clarify the structure further we are only left with the following triple values
\begin{alignat*}{4}
(x_1y_4, y_5)&=a,&\quad  (y_2y_3, y_4)&=c,&\quad  (y_2y_4, y_5)&=e,&\quad 
 (x_2y_4, y_5)&=r,\\
(y_1y_4, y_5)&=b,&\quad  (y_2y_3, y_5)&=d,&\quad  (y_3y_4, y_5)&=f. & \quad
\mbox{} 
\end{alignat*}
Notice that $r \neq 0$ as $x_2y_3=0$ but $x_2 \not \in Z(L)$. We will show that the remaining values are zero.
First by replacing $x_1, y_1, y_2$ by $x_1-(a/r) x_2 , y_1-(b/r) x_2, y_2+(a/r) y_1 - (b/r) x_1$ we can assume that $a=b=0$, that is 
\begin{eqnarray}\label{eq:33II-II.C.2}
			(x_1 y_4,y_5)=(y_1 y_4,y_5)=0. 
 \end{eqnarray}
These changes do not affect \eqref{eq:33II-II.C.1}. Next replace $x_1, y_1, y_3$ by $x_1-d x_3, y_1+c x_3, y_3+c x_1+d y_1$ we see that we can further assume that
\begin{eqnarray}\label{eq:33II-II.C.3}
			(y_2y_3,y_4)=(y_2y_3,y_5)=0. 
\end{eqnarray}
Finally replacing $x_1, y_4, y_5$ by $x_1-e x_4 + f x_5, y_4+ e y_1, y_5 - f y_1$ we can also assume that
\begin{eqnarray}\label{eq:33II-II.C.4} 
			(y_2y_4,y_5)= (y_3 y_4, y_5)=0. 
 \end{eqnarray}
Thus $L$ has a presentation of the form ${\mathcal P}_{10}^{(3,8)}(r,s)$ as described in the next proposition.
\begin{Proposition}\label{pro33IIC}
Let $L$ be a nilpotent SAA of dimension $10$ over a field of characteristic that is not $2$ that has an isotropic center of dimension $3$.
Suppose also that $L$ has the further properties that $\mbox{dim\,} L^3 =6$, $L^3L^2 = L^5$ and $L$  is of type $C$.
Then $L$ has a presentation of the form
\begin{align*}
{\mathcal P}_{10}^{(3,8)}(r,s):\  (x_2y_4,y_5)&=r,\ (x_1y_3, y_5)= s,\ (x_1y_2,y_4) =1,\ (y_1y_2,y_5)=1,\\
 (y_1y_3, y_4)&=1
\end{align*} 
where $r \neq 0$ and $s \notin \mathbb{F}^2$. Furthermore the presentations ${\mathcal P}_{10}^{(3,8)}(\tilde{r}, \tilde{s})$ and 
${\mathcal P}_{10}^{(3,8)}(r,s)$ describe the same algebra if and only if $\frac{\tilde{r}}{r} \in (\mathbb{F}^*)^{3}$ 
and $\frac{s}{\tilde{s}} \in G(s) $ where $G(s)=\{(x^2 - y^2s)^2 :\ (x,y) \in \mathbb{F} \times \mathbb{F} \, \setminus \, \{( 0, 0 )\}\}$.
\end{Proposition}
\begin{proof}
We have already seen that any such algebra has a presentation of the given form. Direct calculations show that an algebra with a presentation 
${\mathcal P}_{10}^{(3,8)}(r,s)$ has the properties stated. 
We turn to the isomorphism property. To see that the condition is sufficient, suppose we have an algebra $L$ that has presentation 
${\mathcal P}_{10}^{(3,8)}(r,s)$ with respect to some standard basis $x_1, y_1, \ldots, x_5, y_5$. Suppose that 
	$\frac{\tilde{r}}{r}=\frac{1}{\beta^3}$ 
					and 
	$\frac{s}{\tilde{s}}= [(b/\beta)^2-s(a/\beta)^2]^2$ 
for some $(a,b) \in \mathbb{F} \times \mathbb{F} \, \setminus \, \{(0,0)\}$. 
Consider a new standard basis 
\begin{equation*}
\begin{aligned}
\tilde{x_1}&=(\alpha/\beta^2)(bx_1 + asy_1), &           \tilde{y_1}&=(1/\beta) (by_1+ax_1),\\
\tilde{x_2}&=(1/\beta)x_2  ,                           &           \tilde{y_2}&=\beta y_2, \\
\tilde{x_3}&=(1/\alpha)x_3       ,                    &            \tilde{y_3}&=\alpha y_3,\\
\tilde{x_4}&=(\alpha/\beta)(bx_4 + asx_5) ,    &            \tilde{y_4}&=(1/\beta^2)(by_4-ay_5),\\
\tilde{x_5}&=ax_4+bx_5    ,                                    &            \tilde{y_5}&=(\alpha/\beta^3)(b y_5-asy_4),
\end{aligned}
\end{equation*}
where $\alpha=\frac{\beta^3}{b^2-a^2s}$. 
Calculations show that $L$ has then presentation ${\mathcal P}_{10}^{(3,8)}(\tilde{r},\tilde{s})$ with respect to the new standard basis.\\ \\
It remains to see that the conditions are also necessary. 
Consider an algebra $L$ with presentation ${\mathcal P}_{10}^{(3,8)}(r, s)$ with respect to some standard basis $x_1, y_1, \ldots, x_5, y_5$. 
Take some arbitrary new standard basis $\tilde{x_1}, \tilde{y_1}, \ldots, \tilde{x_5}, \tilde{y_5}$ such that $L$ satisfies the presentation 
${\mathcal P}_{10}^{(3,8)}(\tilde{r},\tilde{s})$ with respect to the new basis
%
%									?yes
	for some $0 \neq \tilde{r} \in \mathbb{F}$ and $\tilde{s} \notin \mathbb{F}^{2}$.
Notice that
\begin{eqnarray*}
			\tilde{x_5} &=& a x_4 + b x_5\\
			\tilde{y_2} &=& \beta y_2 + u\\
			\tilde{y_3} &=& \gamma y_3 + v ,
\end{eqnarray*}
such that $u, v \in L^3$ and $0 \neq \alpha, \beta \in \mathbb{F}$. 
The reader can convince himself that $\tilde{r}/r \in (\mathbb{F}^{*})^3$ and $s/\tilde{s} \in G(s) $. 
\end{proof}
\noindent
\textbf{Examples. }(1) If $\mathbb{F} =\mathbb{C}$ then as any quadratic polynomial is
 reducible, there are not algebras of type $C$. This holds more generally for any field $\mathbb{F}$
 of characteristic that is not $2$ and where all elements in $\mathbb{F}$ have a square
 root in $\mathbb{F}$.\\ \\
(2) Suppose $\mathbb{F} =\mathbb{R}$. Let $ s \not \in \mathbb{R}^2$ and $ 0 \neq r \in \mathbb{R}$. Then 
			$r/1 = r \in (\mathbb{R}^{*})^3$ and $s < 0$. Also
			$s/(-1) = a^4 = (a^2 - 0^2  (-1))^2$ 
for some $a \in \mathbb{R}$. We thus have that ${\mathcal P}_{10}^{(3,8)}(r, s)$ describes the same algebra as 
${\mathcal P}_{10}^{(3,8)}(1, -1)$.  There is thus a unique algebra in this case.\\\\
(3) Let $\mathbb{F}$ be a finite field of some odd characteristic $p$. Suppose that 
$|\mathbb{F}|=p^n$. Let $s$ be any element that is not in $(\mathbb{F}^*)^2$.
Notice then that $\mathbb{F}^* = (\mathbb{F}^*)^2 \cup s (\mathbb{F}^*)^2$ and thus for any $\tilde{s}$ that is not in $\mathbb{F}^2$, we have  
$s/\tilde{s} \in (\mathbb{F}^*)^2  = G(s)$.  
We can keep $s$ fixed and each algebra has a presentation of the form $ {\mathcal Q}(r) = {\mathcal P}_{10}^{(3,8)}(r,s)$
where ${\mathcal Q}(\tilde{r}) $ and ${\mathcal Q}(r)$ describe the same algebra if and only if $\tilde{r}/ r \in (\mathbb{F}^*)^3$.
There are thus either three or one algebra 
according to whether $3$ divides $p^n -1$ or not.
\subsubsection{\underline{b. The algebras where $\mbox{char\,} \mathbb{F} = 2$}}
If the irreducible minimal polynomial of $\psi \phi^{-1}$ is 
			$t^2 + r t + s$ 
with respect to $y_2, y_3$ then the minimal polynomial with respect to $a y_2,\ b y_3 + c y_2$, where $a, b \neq 0$, is 
\[t^2 +  r (b/a)  t + [ (c/a)^2 - r (c/a) (b/a) + (b/a)^2 s].\] 
Thus we have two distinct subcases (that do not depend on the choice of the basis). Let $m = m(y_2, y_3)$ be the minimal polynomial of $\psi \phi^{-1}$
with respect to a given standard basis for $L$.\\ \\
(1) The minimal polynomial $m$ is of the form $t^2 - s$ for some $s \not \in \mathbb{F}^{2}$ .\\ \\
(2) The minimal polynomial $m$ is of the form $t^2 + r t + s$ where $r \neq 0$ and the polynomial is irreducible.\\ \\
For case (1) we get the same situation as in Proposition \ref{pro33IIC}.
\begin{Proposition}\label{pro33IICIIa}
Let $L$ be nilpotent SAA of dimension $10$ over a field of characteristic $2$ that has an isotropic center of dimension $3$. 
Suppose also that $L$ has the further properties that $\mbox{dim\,}L^3 =6$, $L^3L^2=L^5$ and $L$ is of type $C$ where
the minimal polynomial $m(y_2, y_3)$ is of the form $t^2 - s$ for some $s \not \in \mathbb{F}^2$. 
Then $L$ has a presentation of the form
\begin{align*}
{\mathcal P}_{10}^{(3,8)}(r,s): \ \  (x_2y_4,y_5)&=r,\ (x_1y_3, y_5)= s,\ (x_1y_2,y_4) =1,\  
								    (y_1y_2,y_5)=1,\\  (y_1y_3, y_4)&=1,  					      
\end{align*} 
where $r \neq 0$ and $s \notin \mathbb{F}^2$. Furthermore the presentations 
${\mathcal P}_{10}^{(3,8)}(\tilde{r}, \tilde{s})$ 
and 
${\mathcal P}_{10}^{(3,8)}(r,s)$
describe the same algebra if and only if $\frac{\tilde{r}}{r} \in (\mathbb{F}^{*})^3$ 
and $\frac{s}{\tilde{s}} \in G(s) $, where $G(s)=\{(x^2 - y^2s)^2 :\ (x,y) \in \mathbb{F} \times \mathbb{F}\, \setminus \, \{( 0, 0 )\}\}$.
\end{Proposition}
\noindent
\textbf{Example. }  Consider the field $\mathbb{Z}_2(x)$ of rational functions in one variable over $\mathbb{Z}_2$. Notice that 
		$$\mathbb{Z}_2(x)^* = \{ f(x)^2 + x g(x)^2: (f(x),g(x)) \in \mathbb{Z}_2(x) \times  \mathbb{Z}_2(x) \, \setminus \, \{(0,0)\}\}.$$ Thus 
$G(x)=(\mathbb{Z}_2(x)^{*})^2$ and last proposition tells us that 
${\mathcal P}_{10}^{(3,8)}(\tilde{r}(x), \tilde{s}(x))$ and ${\mathcal P}_{10}^{(3,8)}(r(x),s(x))$ describe the same algebra 
			if and only if 
	$\tilde{r}(x)/r(x) \in (\mathbb{Z}_2(x)^{*})^{3}$ and 
	$s(x)/\tilde{s}(x) \in (\mathbb{Z}_2(x)^{*})^{2} $. 
We thus have infinitely many algebras here.\\ \\              
We then move to the latter collection of algebras. For the rest of this subsection we will be assuming that the minimal polynomial of 
$\psi \phi^{-1}$ is an irreducible polynomial of the form $t^2 + rt + s$ where $r \neq 0$.\\\\
Pick $0 \neq x_5 \in L^5$ and let 
			$x_4=\psi\phi^{-1}(x_5)$. 
Then 
			$\psi\phi^{-1}(x_4)=r x_4 + s x_5$. 
We want to pick our standard basis such that 
			$ \bar{x_1} = x_1+L^4 = \phi^{-1}(x_4)$ and $ \bar{y_1} = y_1+L^4 = \phi^{-1}(x_5)$.
The only constraint to worry about is, as before, that $(x_1,y_1) = 1$. 
If $(x_1,y_1)=\sigma$, we just need to then replace 
$ y_3$ by $(1/\sigma)y_3$. 
Notice that this changes the minimal polynomial of $\psi \phi^{-1}$ to $t^2 + ( r /\sigma) t + ( s/\sigma^2)$.
In any case this shows that we can choose our standard basis such that 
\begin{eqnarray} \label{eq:33II-III.C.1}
		x_1y_2=x_4,\ y_1y_2= x_5,\ x_1y_3= r x_4 + s x_5,\ y_1y_3=x_4
  \end{eqnarray}
for some $r, s \in \mathbb{F}$ where $ r \neq 0$ and $t^2 + rt + s$ is irreducible.
As before we also know that $x_2y_3=0$. In order to clarify the structure further 
we are only left with the following triple values:
\begin{alignat*}{4}
(x_1y_4, y_5)&=a,&\quad  (y_2y_3, y_4)&=c,&\quad  (y_2y_4, y_5)&=e,&\quad 
 (x_2y_4, y_5)&= \alpha ,\\
(y_1y_4, y_5)&=b,&\quad  (y_2y_3, y_5)&=d,&\quad  (y_3y_4, y_5)&=f .& \quad 
\mbox{} 
\end{alignat*}
Notice that $\alpha \neq 0$ as $x_2y_3 =0 $ but $x_2 \not \in Z(L)$. We will show that we can choose a new standard
 basis so that the remaining values are zero. First replace $x_1, y_1, y_2$ by $x_1-(a/\alpha) x_2 , y_1-(b/\alpha) x_2,
 y_2+(a/\alpha) y_1 - (b/\alpha) x_1$ and we can assume that $a=b=0$, that is 
\begin{eqnarray}\label{eq:33II-III.C.2}
			(x_1 y_4,y_5) = (y_1 y_4,y_5)=0. 
\end{eqnarray}
Next replace $x_1, y_1, y_3$ by 
$x_1-d x_3, y_1+c x_3, y_3+c x_1+d y_1$ and we can furthermore assume that 
\begin{eqnarray}\label{eq:33II-II.C.3} \begin{aligned} 
			(y_2y_3,y_4)= (y_2y_3,y_5)=0. 
\end{aligned} \end{eqnarray}
Finally replacing $x_1, y_4, y_5$ by $x_1-e x_4 + f x_5, y_4+ e y_1, y_5 - f y_1$ allows us to further 
assume  that
\begin{eqnarray}\label{eq:33II-III.C.4} \begin{aligned} 
			(y_2y_4,y_5)=(y_3 y_4, y_5)=0. 
\end{aligned} \end{eqnarray}
We have thus arrived at a presentation of the form 
${\mathcal P}_{10}^{(3,9)}$ as described in the next proposition.
Before stating that proposition we introduce two groups that are going to play a role.
\begin{defn}  
For each minimal polynomial $t^2+rt+s$, we let
\begin{eqnarray*}
			H(r)    &=&  \{x^2 + rx              :\ x \in \mathbb{F}\} \\
			G(r,s)  &=&  \{x^2 + rxy + sy^2 :\  (x,y) \in \mathbb{F} \times \mathbb{F} \,\setminus \, \{(0,0)\} \}.
\end{eqnarray*}
\end{defn}
\noindent
\begin{Remark} \label{lma136} 
$(1)$ $H(r)$ is a subgroup of the additive group $\mathbb{F}$. \\ \\
$(2)$ Consider the splitting field $\mathbb{F}[\gamma]$ of the polynomial $t^2 + r t + s$ in $\mathbb{F}[t]$. Then
		$a^2+a b r+b^2s$ is the norm 
		$N(a+b \gamma) = (a+b \gamma)(a+b (\gamma+r))$ 
of $a+b \gamma$. As this is a multiplicative function we have that
$G(r,s)$ is a multiplicative subgroup of $\mathbb{F}^* $. 
\end{Remark}
\noindent
\begin{Proposition}\label{pro33IICl}
Let $L$ be a nilpotent SAA of dimension $10$ over a field of characteristic $2$ that has an isotropic center of dimension $3$.
Suppose also that $L$ has the further properties that $\mbox{dim\,}L^3 =6$, $L^3L^2 = L^5$ and $L$ is of type $C$ where
the minimal polynomial $m(y_2, y_3)$ is irreducible with a non-zero linear term. Then $L$ has a presentation of the form
\begin{align*}
 {\mathcal P}_{10}^{(3,9)}(\gamma, r, s) : \ \ & (x_2y_4,y_5) = \gamma,\ (x_1y_3,y_4) = r,\ (x_1y_3,y_5) = s,\ \\ 
											  & (x_1y_2,y_4) = 1,\ (y_1y_2,y_5) = 1,\ (y_1y_3,y_4) = 1
\end{align*} 
where $  \gamma, r \neq 0$ and $t^2 + r t + s$ is irreducible. Furthermore the presentations 
${\mathcal P}_{10}^{(3,9)}(\tilde{\gamma}, \tilde{r}, \tilde{s})$ and ${\mathcal P}_{10}^{(3,9)}(\gamma,r,s)$ describe the same algebra if and only if
 $\frac{\tilde{\gamma}}{\gamma} \in (\mathbb{F}^*)^{3}$, $\frac{\tilde{r}}{r} \in G(r, s)$ and $\tilde{s} - (\frac{\tilde{r}}{r})^2s \in H(\tilde{r}) $.
\end{Proposition}
\begin{proof} %
We have already seen that any such algebra has a presentation of the given form. Direct calculations show that conversely 
any algebra with a presentation for this type satisfies all the properties listed.
It remains to deal with the isomorphism property. To see that the condition is sufficient, suppose we have an algebra $L$ that has a presentation 
${\mathcal P}_{10}^{(3,9)}(\gamma, r, s) $ with respect to some standard basis $x_1, y_1, \ldots, x_5, y_5$. Suppose that 
			$\frac{\tilde{\gamma}}{\gamma}     =     \frac{1}{\beta^3}$, 
			$\frac{r}{\tilde{r}}                          =     (\frac{b}{\beta})^2+(\frac{b}{\beta})(\frac{a}{\beta})r+(\frac{a}{\beta})^2 s$ 
										     and
			 $ \tilde{s}-(\frac{\tilde{r}}{r})^2 s   =    (\frac{\delta}{\beta})^2+(\frac{\delta}{\beta}) \tilde{r}$ 
for some $a, b, \delta , \beta \in F$ where $\beta  \neq 0$. We let
			 $\alpha  =  \beta/((\frac{b}{\beta})^2+(\frac{b}{\beta})(\frac{a}{\beta})r+(\frac{a}{\beta})^2 s)$.
Consider the new standard basis 
\begin{equation*}
\begin{aligned}
\tilde{x_1}&=\frac{1}{\beta^2}[(\alpha a r+\alpha b+\delta a) x_1+(\delta b+\alpha a s) y_1],       &           \tilde{y_1}&=\frac{1}{\beta}(ax_1+by_1),\\
\tilde{x_2}&=\frac{1}{\alpha \beta}(\alpha x_2+\delta x_3)  ,                                                       &         \tilde{y_2}&=\beta y_2, \\
\tilde{x_3}&=\frac{1}{\alpha}x_3       ,                                                                                       &            \tilde{y_3}&=\alpha y_3+\delta y_2,\\
\tilde{x_4}&=\frac{1}{\beta}[(\alpha a r+\alpha b+\delta a) x_4+(\delta b+\alpha a s) x_5] ,          &            \tilde{y_4}&= \frac{1}{\beta^2}(by_4+ay_5),\\
\tilde{x_5}&=ax_4+bx_5     ,               &            \tilde{y_5}&=\frac{1}{\beta^3}[(\alpha a r+\alpha b+\delta a)y_5+\\
	      &					&			&(\delta b+\alpha a s) y_4].
\end{aligned}
\end{equation*}
Calculations show that $L$ has then presentation ${\mathcal P}_{10}^{(3,9)}(\tilde{\gamma}, \tilde{r}, \tilde{s})$ with respect to the new standard basis.\\ \\
It remains to see that the conditions are also necessary. 
Consider an algebra $L$ with presentation ${\mathcal P}_{10}^{(3,9)}(\gamma, r, s) $ with respect to some standard basis $x_1, y_1, \ldots, x_5,$
$ y_5$. Take some arbitrary new standard basis $\tilde{x_1}, \tilde{y_1}, \ldots, \tilde{x_5}, \tilde{y_5}$ such that $L$ has presentation 
${\mathcal P}_{10}^{(3,9)}(\tilde{\gamma}, \tilde{r}, \tilde{s})$ with respect to the new basis
	where $\tilde{\gamma}, \tilde{r} , \tilde{s} \neq 0$ and where $t^2 + \tilde{r} t + \tilde{s} $ is irreducible. Then
\begin{eqnarray*}
			\tilde{x_5} &=& a x_4 + b x_5,\\
			\tilde{y_2} &=& \beta y_2 + u,\\
			\tilde{y_3} &=&\alpha y_3+\delta y_2+v ,
\end{eqnarray*}
for some $u,v \in L^3$, $ \alpha, \beta , \delta \in \mathbb{F}$ where $\alpha, \beta \neq 0$.
The reader can convince himself that new standard basis that we get satisfies the conditions stated.
\end{proof}
\noindent
Before we give an example to an algebra of this form. We list some useful properties of the groups $G(r,s)$ and $H(r)$.
\begin{Lemma}\label{lem136} 
For any irreducible polynomials $t^2 + r t+ s $ and $t^2 + \tilde{r} t + \tilde{s}$, we have that
\begin{description}
\item$(1)$ $H(\tilde{r})=(\tilde{r}/r)^2H(r)$. 
\item$(2)$ $G(\tilde{r},\tilde{s})=G(r,s)$ if $\tilde{s}-(\tilde{r}/r)^2 s \in H(\tilde{r})$.
\end{description}
\end{Lemma}
\begin{proof} Straightforward calculations. \end{proof}
\noindent
{\bf Example. }\ \ 
Let $\mathbb{F}$ be the finite field of order $2^n$. Let $\gamma, r, s, \tilde{\gamma}, \tilde{r}, \tilde{s} $ be as 
in last lemma. Then
$ G(\tilde{r}, \tilde{s}) = G (r, s) = \mathbb{F}^* $ and thus $\tilde{r}/r \in G(r,s)$.
Also $[ \mathbb{F}:\ H(\tilde{r})  ] = 2$ and thus $\tilde{s} - (\tilde{r}/r)^2 s \in H(\tilde{r})$.
It follows from last proposition that the presentations ${\mathcal P}_{10}^{(3,9)}(\gamma, r, s)$ and 
${\mathcal P}_{10}^{(3,9)}(\tilde{\gamma}, \tilde{r}, \tilde{s})$ describe the same algebra if and only if $\tilde{\gamma} / \gamma \in (\mathbb{F}^*)^3$.
There are thus either three algebras or one algebra according to whether $3$ divides $2^n -1$ or not.\\ \\
We end this section by giving a direct explanation why the relation
									\[ ( \tilde{r}, \tilde{s} ) \sim ( r ,s) 
													\mbox{ if and only if } 
		\frac{\tilde{r}}{r} \in G(r, s), \tilde{s} - (\frac{\tilde{r}}{r})^2s \in H(\tilde{r}) \]
is an equivalence relation. \\ \\
First it is easy to show $(r, s) \sim (r,s)$ as $1 \in G(r,s) $ and $0 \in H(r)$. 
Next if $(\tilde{r},\tilde{s}) \sim (r,s)$ then, as $G(r,s) = G(\tilde{r},\tilde{s})$ is a group, we have that 
						$r/\tilde{r} \in  G(\tilde{r},\tilde{s})$ 
										and
			 $s-(r/\tilde{r})^2\tilde{s}=(r/\tilde{r})^2 (\tilde{s}-(\tilde{r}/r)^2s) \in (r/\tilde{r})^2 H(\tilde{r})=H(r)$.
This shows that $\sim$ is symmetric. 
Finally suppose $(r^*, s^*) \sim ( \tilde{r},\tilde{s})$ and $( \tilde{r},\tilde{s}) \sim ( r, s)$. 
Then we have that 
			  $r^*/r = r^*/\tilde{r} \, \cdot \, \tilde{r}/r \in G(r,s) $ 
and
 				$s^*-(r^*/r)^2s = [s^* - (r^*/ \tilde{r})^2 \tilde{s}] + [ (r^*/ \tilde{r})^2 \tilde{s} -  (r^*/r)^2s] 
			=  [s^* - (r^*/ \tilde{r})^2 \tilde{s}] +  (r^*/ \tilde{r})^2 [ \tilde{s} -  (\tilde{r}/r)^2 s ]$
 is in $H(r^*) + (r^*/ \tilde{r})^2 H(\tilde{r}) 
			= H(r^*)$.
Hence $\sim$ is also transitive and we have an equivalence relation.
\noindent

\chapter{Algebras with an isotropic center of dimension $2$}
In this chapter we will be assuming that $Z(L)$ is isotropic of dimension $2$. 
Notice that if $L = {\mathbb F} u +  {\mathbb F} v + L^2$, then $L^2 = {\mathbb F} uv + L^3$. It follows that $L^2 = Z(L)^{\perp}$ is of dimension $8$ and that $L^3$ is of dimension $7$. We can then pick our standard basis such that
\begin{eqnarray*}
Z(L)&=& {\mathbb F} x_5+ {\mathbb F} x_4, \\
%
%Z_2(L)&=& {\mathbb F} x_5+ {\mathbb F} x_4 + {\mathbb F} x_3,\\
%
L^2&=&{\mathbb F} x_5+ \cdots  + {\mathbb F} x_1 + {\mathbb F} y_1 + {\mathbb F} y_2+ {\mathbb F} y_3,\\
L^3&=&{\mathbb F} x_5+ \cdots + {\mathbb F} x_1 + {\mathbb F} y_1 + {\mathbb F} y_2 .
\end{eqnarray*}
Furthermore $L^3 = {\mathbb F} uvu +  {\mathbb F} uvv + L^4$ and thus $\mbox{dim\,} L^4 \in \{5,6\}$. Let $k$ be the nilpotence class of $L$. We know that the maximal class is $7$ and as  $\mbox{dim\,} L^{k}  \neq 1$ and $L^{k} \leq Z(L)$, we must have that $L^{k} = Z(L)$. Moreover, we know that $\mbox{dim\,} L^{s} \neq 2$ for $ 1 \leq s \leq 4$ and thus $ 5 \leq k \leq  7$. If $L^{5} = Z(L)$  then 
 $\mbox{dim\,} Z_2(L) - \mbox{dim\,} Z(L) =\mbox{dim\,} L^2 - \mbox{dim\,} L^3 = 1$
and we get the contradiction that $L^{4} = Z_2(L)$ is of dimension $3$. Thus $6 \leq k \leq 7$. We will deal with the two cases separately.
\section{The algebras of class $6$}
As the class is $6$, it follows that $(L^4, L^4) =(L^7, L) = 0$ and thus
$L^4$ is isotropic. We have seen that the dimension of $L^4$ is at least $5$ and thus $\mbox{dim\,} L^4 = 5$. We can thus now furthermore choose our standard basis such that 
\newpage
%++-------------------------++++---------PICTURE-----------++++--------++++---------++
%
%\begin{minipage}
\begin{multicols}{2}
\begin{alignat*}{2}
\color{cyan}L^6=Z(L)\ &\boxed{
\begin{matrix}
\color{green} x_5\\\color{green} x_4
\end{matrix}
}\ \ 
\begin{matrix}
y_5\\y_4
\end{matrix}\\
\color{cyan}L^5=Z_2(L)\ &\boxed{
\begin{matrix}
\color{magenta} x_3
\end{matrix}} 
\boxed{
\begin{matrix}
\color{magenta} y_3
\end{matrix}
} \ \color{cyan}L^2=Z_5(L)\\
\color{cyan}L^4=Z_3(L)\ &\boxed{
\begin{matrix}
\color{blue} x_2\\ \color{blue} x_1
\end{matrix}}
\boxed{
\begin{matrix}
\color{blue} y_2\\ \color{blue} y_1
\end{matrix}
}\ \color{cyan}L^3=Z_4(L)
\end{alignat*}
\break
%
%$$\begin{array}{l}
\begin{align*}
Z(L)=L^6 &={\mathbb F} x_5+ {\mathbb F} x_4 \\
Z_2(L)=L^5 &={\mathbb F} x_5+ {\mathbb F} x_4 + {\mathbb F} x_3 \\
Z_3(L)=L^4 &={\mathbb F} x_5+ \cdots + {\mathbb F} x_1\\
Z_4(L)=L^3 &=L^4+ {\mathbb F} y_1+ {\mathbb F} y_2\\
Z_5(L)=L^5 &=L^4+ {\mathbb F} y_1+ {\mathbb F} y_2+ {\mathbb F} y_3.
%
%\end{array}$$
\end{align*}
\end{multicols}
%\end{minipage}
%

 %
%
\noindent
As $ L^4Z_4(L)=0$ we must have
\begin{equation*}\label{eq:32I-I.1} \begin{aligned} x_1y_2=0. \end{aligned} \end{equation*}
It then follows that $L^4L^3=0$ and then
\[L^3L^3 =  {\mathbb F} y_1y_2.\]
Notice that 
$L^3L^3 \neq 0$  since this would imply that $(L^6, L) = (L^3 , L^4 ) = (L^3 L^3 , L) =0 $ and we would get 
the contradiction that the class of $L$ is at most $5$.  
Next let us see that $x_1y_3$ and $x_2y_3$ are linearly independent. To see this we argue by contradiction and suppose that 
$(a x_1 + b x_2)y_3 = 0$ for some $a, b \in {\mathbb F} $ where not both $a, b$ are zero.
But this would imply that $(a x_1 + b x_2) L \in Z(L)$ and we would thus get the contradiction that $a x_1 + b x_2 \in Z_2(L) = L^5$. 
Thus
 $$L^4 L^2= {\mathbb F} x_1 y_3 + {\mathbb F} x_2 y_3 = Z(L).$$
Notice that $L^3L^3 =  {\mathbb F} y_1y_2$ is a one-dimensional characteristic subspace of $Z_2(L)$. We consider two possibilities:  $L^3L^3 \leq  Z(L)$ and $L^3L^3 \not \leq Z(L)$.
%+++++++++++++++__________________________________---------------------------------------+++++++++++++++++________________---------------------------
\subsection{Algebras where $L^3L^3 \leq Z(L)$ }
We pick our standard basis such that
 \begin{equation} \label{eq:32I-I.2} \begin{aligned} 
	L^3L^3 ={\mathbb F} y_1y_2 = {\mathbb F} x_5. 
\end{aligned} \end{equation}
We have seen above that $Z(L)  = {\mathbb F} x_2 y_3 + {\mathbb F} x_1 y_3 = L^4L^2$.
 In order to clarify the structure of $L$ we introduce the characteristic subspace
	\[ W = \{ x \in L^4 :\, x L^2  \leq   L^3L^3 \} . \] 
Notice that $W$ is the kernel of the surjective linear map $L^4 \rightarrow Z(L)/L^3L^3,\, x \mapsto xy_3 + L^3L^3$ and thus
$W$ is of codimension $1$ in $L^4$. Also $L^5 < W$. We can thus pick our standard basis such that
	$$W =  {\mathbb F} x_5+ {\mathbb F} x_4 + {\mathbb F} x_3 + {\mathbb F} x_2.$$
From this one sees that we have a chain of characteristic ideals of $L$
%&&&&&&&&&&&&&&&&&&&&&&&& Pictuer &&&&&&&&&&&&&&&&&&&&&&&&&&&&&&
\begin{alignat*}{2}
\color{magenta}L^3L^3 \ &\boxed{
\begin{matrix}
\color{green} x_5
\end{matrix}
}\ \ 
\begin{matrix}
y_5
\end{matrix} \\
\color{cyan}L^6=Z(L) \ &\boxed{
\begin{matrix}
\color{green} x_4
\end{matrix}
}
\boxed{\begin{matrix}
\color{green}y_4
\end{matrix}} \ \color{magenta}(L^3L^3 )^\perp \\
\color{cyan}L^5=Z_2(L)\ &\boxed{
\begin{matrix}
\color{green}x_3
\end{matrix}
}
\boxed{
\begin{matrix}
 \color{green}y_3
\end{matrix}} \ \color{cyan}L^2=Z_5(L)\\
\color{magenta}W\ &\boxed{
\begin{matrix}
\color{blue} x_2
\end{matrix}}
\boxed{
\begin{matrix}
\color{blue}y_2
\end{matrix}
}\ \color{cyan}L^3=Z_4(L)\ \\
\color{cyan}L^4=Z_3(L)\ &\boxed{
\begin{matrix}
\color{blue} x_1
\end{matrix}}
\boxed{
\begin{matrix}
\color{blue}y_1
\end{matrix}
}\ \color{magenta}W^\perp
\end{alignat*}
%&&&&&&&&&&&&&&&&&&&&&&&&&&&&&&&&&&&&&&&&&&&&&&&&&&&&&
Notice that ${\mathbb F} x_2 y_3 ={\mathbb F} x_5$. We continue considering characteristic subspaces. Let
\[ S = \{ x \in L^3 :\, x \cdot L^2 \leq L^3L^3 \}.\]
Notice that $L^3L^2 = Z(L)$ and that $S$ is the kernel of the surjective linear map 
$L^3 \rightarrow Z(L)/L^3L^3,\,$
$x\mapsto x \cdot y_3 + L^3L^3$ 
and is thus of codimension $1$ in $L^3$. Notice also that $x_1 \not \in S$
whereas $ W \leq S$. It follows that we can pick our standard basis such that
$$S = {\mathbb F} x_5+ {\mathbb F} x_4 + {\mathbb F} x_3+ {\mathbb F} x_2  + {\mathbb F} y_1 + {\mathbb F} y_2.$$
In particular we have $y_1y_3, y_2y_3 \in L^3L^3$. Notice that
$$S^\perp =L^5+ {\mathbb F} y_1$$
and that $L^2S^\perp = L^2y_1 = {\mathbb F} y_1 y_2+ {\mathbb F} y_1 y_3 = L^3L^3$. Let  
\[ T = \{ x \in L^2 :\, x S^\perp = 0\} .\]
Then $T$ is the kernel of the surjective linear map $L^2 \rightarrow L^3L^3,\,x\mapsto y_1x$. Notice that $W^\perp \leq T$ but that $y_2 \not \in T$. We can then pick our standard basis such that 
$$T = W^\perp+{\mathbb F} y_3.$$
In particular $ y_1 y_3=0$. We now have a characteristic isotropic subspace $T^\perp =  {\mathbb F} x_5+ {\mathbb F} x_4 + {\mathbb F} x_2$ where $T^\perp \cdot (L^3L^3)^\perp= x_2 (L^3L^3)^\perp =  L^3L^3$. We now let  
\[ R = \{ x \in (L^3L^3)^\perp :\, x  T^\perp = 0 \} .\]
This is the kernel of the surjective linear map $(L^3L^3)^\perp \rightarrow L^3L^3,\,x\mapsto x_2x$ that contains $L^3$. We now refine our standard basis such that
$$R = L^3 + {\mathbb F} y_4$$
and we have in particular $ x_2 y_4=0$. Let us summarize. For every standard basis that respects the list of characteristic subspaces above, we have
\begin{eqnarray*}
x_2y_4 & = & 0\\
\mbox{} {\mathbb F} x_2y_3 & = & {\mathbb F} x_5 \\
x_1y_2 & = & 0 \\
{\mathbb F} x_5 + {\mathbb F} x_1y_3 & = & {\mathbb F} x_5 + {\mathbb F} x_4\\
{\mathbb F} x_5 + {\mathbb F} x_1y_4 & = & {\mathbb F} x_5 + {\mathbb F} x_3 \\
 {\mathbb F} y_1y_2 & = & {\mathbb F} x_5\\
y_1y_3 & = & 0 \\
y_2y_3 & \in & {\mathbb F} x_5
\end{eqnarray*}
It is not difficult to see that we can furthermore refine our basis such that 
\begin{equation}\label{eq:32I-I.7} \begin{aligned}
 x_2y_3 = x_5,\  x_2y_4 = 0,\ x_1y_2 = 0,\ x_1y_3 = x_4,\ x_1y_4 = - x_3,\ y_1y_2 = x_5,\ y_1y_3 =0 .
\end{aligned} \end{equation}
 This deals with all triple values apart from
\begin{alignat*}{3}
(y_1y_4, y_5)&=a, &\quad   (y_2y_3, y_5)&=c , &\quad  (x_3y_4, y_5)&=r,\\
(y_3y_4, y_5)&=b, & \quad (y_2y_4, y_5)&=d, &\quad  \mbox{}
\end{alignat*}
where $r \neq 0$ as $x_3 \not \in Z(L)$. Replace $x_2, y_4$ by $x_2 + a x_4, y_4 - a y_2 $ and we can assume that 
\begin{equation}\label{eq:32I-I.2} (y_1y_4, y_5)=0.\end{equation}
This does not affect \eqref{eq:32I-I.7}. Then replace $y_2, y_3$ by $y_2 - (d/r) x_3 - c x_2, y_3 - (b/r) x_3 -  (d/r) x_2$ 
gives furthermore that
\begin{equation}\label{eq:32I-I.5} \begin{aligned}
(y_2y_3, y_5)=0, (y_2y_4, y_5)=0, (y_3 y_4, y_5)=0. 
\end{aligned} \end{equation}
Again these changes do not affect \eqref{eq:32I-I.7} and  \eqref{eq:32I-I.2}. Finally replacing $x_1, x_2, x_4, x_5, y_1,$
$y_2, y_4, y_5$ by $r^2 x_1, (1/r) x_2,$
$ r^2 x_4, (1/r) x_5$ 
$, (1/r^2) y_1, r y_2,$
$ (1/r^2) y_4, r y_5$ implies that we can assume that $(x_3y_4, y_5) =1$ while \eqref{eq:32I-I.7}-\eqref{eq:32I-I.5} remain nonaffected. We thus arrive at a unique presentation.
\begin{Proposition}\label{pro33I}
There is a unique nilpotent SAA $L$ of dimension $10$ with an isotropic center of dimension $2$ that has the further properties that $L$
is nilpotent of class $6$ and $L^3L^3 \leq Z(L)$. This algebra can be given by the presentation
\begin{align*}
{\mathcal P}_{10}^{(2,1)}: \ \  (x_3 y_4, y_5)=1,\  (x_2 y_3, y_5)=1,\  (x_1 y_3, y_4)=1,\  (y_1y_2,y_5)=1. 
\end{align*} 
\end{Proposition} \noindent
One readily verifies that the algebra with the presentation above has the properties stated.
%
%+++++++++++++++__________________________________---------------------------------------+++++++++++++++++________________---------------------------
%
\subsection{Algebras where $L^3L^3  \nleq Z(L)$}
Recall that $x_1y_2 = 0$. As we had observed before, ${\mathbb F} y_1y_2 = L^3L^3 \leq Z_2(L) = Z(L) + {\mathbb F} x_3$. We had also seen that ${\mathbb F} x_1y_3 + {\mathbb F} x_2y_3 = Z(L)$. We can now pick our standard basis such that
\begin{equation}\label{eq:32I-II.1} \begin{aligned}
x_1y_2 = 0,\ y_1y_2 = x_3,\ x_1y_3=x_4,\ x_2y_3=x_5. 
\end{aligned} \end{equation}
This leaves us with the following list of triple values to determine.
\begin{alignat*}{4}
(x_{1}y_{4},y_{5})&=a,&\quad  (y_{2}y_{3},y_{4})&=c,&\quad  (y_{1}y_{3},y_{4})&=f, &\quad  (y_{1}y_{4},y_{5})&= h,\\
\mbox{} (x_{2}y_{4},y_{5})&=b,&\quad  (y_2y_3, y_5)&=d,&\quad  (y_{1}y_{3},y_{5})&=g, &\quad  (y_3 y_4, y_5)&=k,\\
\mbox{} (x_{3}y_{4},y_{5}) &= r, &\quad  (y_2 y_4, y_5) &= e, &\quad  \mbox{} &\quad 
\end{alignat*}
where $r \neq 0$ as $x_3 \not \in Z(L)$. We show that we can further refine the basis so that $a= b=c=d=e=f=g=h=k=0$. To start with replace
$x_1, x_2, y_3$ by $x_1 - (a/r)x_3, x_2 - (b/r) x_3, y_3 + (a/r) y_1 + (b/r) y_2$
and we see that we can assume that
\begin{equation}\label{eq:32I-II.2} \begin{aligned}
(x_1 y_4, y_5) = (x_2 y_4, y_5) =0. 
\end{aligned} \end{equation}
These changes do not affect \eqref{eq:32I-II.1}. Next we replace $y_1, y_2, y_3$ by $y_1 - c x_2, y_2 - c x_1 - d x_2 - (e/r) x_3, y_3 - (e/r) x_2$ that gives us
\begin{equation}\label{eq:32I-II.3} \begin{aligned}
(y_2 y_3, y_4) = (y_2 y_3, y_5) = (y_2 y_4, y_5) = 0, 
\end{aligned} \end{equation}
while these changes have no affect on \eqref{eq:32I-II.1}, \eqref{eq:32I-II.2}. Now replace $x_2, y_4, y_5$ by $x_2 - f x_4 - g x_5, y_4 + f y_2, y_5 + g y_2$ and we can now assume that
\begin{equation}\label{eq:32I-II.4} \begin{aligned}
(y_1 y_3, y_4) = (y_1 y_3, y_5) = 0. 
\end{aligned} \end{equation}
As before this has no effect on the previous established equations. Finally we replace $y_1, y_3$ by $y_1 - (h/r) x_3, y_3 - (h/r) x_1 - (k/r) x_3$ and one sees readily that we can furthermore assume that $(y_1 y_4, y_5) = (y_3 y_4, y_5) =0$. We thus arrive at a family of algebras given by the presentation ${\mathcal P}_{10}^{(2,2)}$ given in the next Proposition.
\begin{Proposition}\label{pro33I} Let $L$ be a nilpotent SAA of dimension $10$ with an isotropic center of dimension $2$ with the further properties that $L$ is of nilpotence class $6$ and $L^3L^3 \nleq Z(L)$. Then $L$ has a presentation of the form 
\begin{align*}
{\mathcal P}_{10}^{(2,2)}(r): \ \  (x_3 y_4, y_5)=r,\  (x_2 y_3, y_5)=1,\  (x_1 y_3, y_4)=1,\  (y_1y_2,y_3)=1,
\end{align*} 
where $r \neq 0$. Furthermore the presentations ${\mathcal P}_{10}^{(2,2)}(r)$ and ${\mathcal P}_{10}^{(2,2)}(s)$
 describe the same algebra if and only if $s/r \in ({\mathbb F}^{*})^4$.
\end{Proposition}
\begin{proof}
We have already seen that all such algebras have a presentation of the form ${\mathcal P}_{10}^{(2,2)}(r)$ for some $0 \neq r \in {\mathbb F}$. Straightforward calculations show that conversely any algebra with such a presentation has the properties stated in the Proposition. It remains to prove the isomorphism property. To see that it is sufficient, suppose that we have an algebra $L$ with presentation 
${\mathcal P}_{10}^{(2,2)}(r)$ with respect to some given standard basis.
% $x_1, y_1, x_2, y_2$$, x_3, y_3, x_4$$, y_4,x_5, y_5$. 
Let $s$ by any element in ${\mathbb F}^*$ such that $s/r = b^4 \in ({\mathbb F}^*)^4$. Replace the basis for $L$ with a new standard basis $\tilde{x_{1}}, \cdots, \tilde{y_{5}}$ where $\tilde{x_{1}} = x_1 $, $\tilde{y_{1}} = y_1 $, $\tilde{x_{2}}=(1/b)x_2, \tilde{y_{2}}=b y_2, \tilde{x_{3}} = bx_3, \tilde{y_{3}}=(1/b) y_3, \tilde{x_{4}}=(1/b) x_4, \tilde{y_{4}}=
b y_4, \tilde{x_5}=(1/b^2) x_5$ and $\tilde{y_5}=b^2 y_5$. Direct calculations show that $L$ has the presentation ${\mathcal P}_{10}^{(2,2)}(s)$ with respect to this new basis.\\ \\
It remains to see that the condition is necessary. Consider again an algebra $L$ with presentation ${\mathcal P}_{10}^{(2,2)}(r)$
and suppose that $L$ has also a presentation ${\mathcal P}_{10}^{(2,2)}(s)$ with respect to some other standard basis $\tilde{x_{1}}, \cdots, \tilde{y_{5}}$. We want to show that $s/r \in  ({\mathbb F}^{*})^4$. We know that $L = {\mathbb F} \tilde{y_{5}} + {\mathbb F} \tilde{y_{4}} + L^2 = {\mathbb F} y_{5} + {\mathbb F} y_{4} + L^2 $. Thus
\begin{eqnarray*}
\tilde{y_4}&=& ay_4+by_5+u_4 \\
\tilde{y_5}&=& cy_4+dy_5+u_5,
\end{eqnarray*}
for some $u_4, u_5 \in L^2$ and $a, b, c, d \in {\mathbb F}$ where $ad-bc \neq 0$. We know that $L^2L^2 \leq L^4$ and thus
\begin{eqnarray*}
\tilde{y_5}\tilde{y_4}\tilde{y_4}&=& (cy_4 + dy_5) (ay_4+by_5) (ay_4+by_5) + w\\
\tilde{y_4}\tilde{y_5}\tilde{y_5} &=& (ay_4+by_5) (cy_4 + dy_5) (cy_4 + dy_5) + z,
\end{eqnarray*}
where $w, z \in L^4$. We use the fact that $L^7=0$ and $L^3L^3 \leq L^5$ in the following calculation. We have
 $$s^3 = (\tilde{y_4}\tilde{y_5}\tilde{y_5} \cdot (\tilde{y_4}\tilde{y_5}),\tilde{y_5}\tilde{y_4}^2)=  (ad-bc)^4 r^3. $$
Hence $s/r \in ({\mathbb F}^*)^4$.
\end{proof}
\begin{Remark}
$(1)$ It thus depends on the field ${\mathbb F}$, how many algebras there are of this type. When $({\mathbb F}^*)^4 = {\mathbb F}^*$ there is just one algebra. This includes the case when ${\mathbb F}$ is algebraically closed or finite field of characteristic $2$.\\ \\
$(2)$ Let ${\mathbb F}$ be a finite field of order $p^n$ where $p$ is an odd prime. If $p \equiv 1\, (\mbox{mod\,}4)$ then there are $4$ algebras and if $p\equiv -1\,(\mbox{mod\,}4)$ then there are $4$ algebras when $n$ is even and $2$ algebras when $n$ is odd.\\ \\
$(3)$ For ${\mathbb F} =  {\mathbb R}$ there are two algebras, one for $r < 0$ and one for $r > 0$. For  ${\mathbb F} = {\mathbb Q}$ there are infinitely many algebras.
\end{Remark}
%
%++-------------------------++++--------++++--------------------------++++-------++--------------------------+++----------+++---------------------------------+++------+++
%
\section{The algebras of class $7$}
Here we are dealing with algebras of maximal class and thus we can make use of the general theory concerning these. In particular we know that we can choose our standard basis such that
\begin{eqnarray*}
L^7=Z(L) &= &{\mathbb F} x_5+ {\mathbb F} x_4, \\
\mbox{}L^6 = Z_2(L) & = & {\mathbb F} x_5+ {\mathbb F} x_4 + {\mathbb F} x_3, \\
\mbox{}L^5 = Z_3(L) & = & {\mathbb F} x_5+ {\mathbb F} x_4 + {\mathbb F} x_3 + {\mathbb F} x_2.
\end{eqnarray*}
We also know that $L^4 = Z_4(L) = (L^5)^\perp$, $L^3 = Z_5(L) = (L^6)^\perp $ and $L^2 = Z_6(L) = (L^7)^\perp$. Furthermore we know that we can also get characteristic ideals of dimension $1, 5$ and $9$ in the following way.
\\ \\
Firstly, we know from the general theory that $x_3y_4, x_2y_3 \neq 0$. As a result $L^5L^2 = {\mathbb F} x_2y_3 \neq 0$. This gives us a characteristic ideal of dimension $1$ and then $(L^5L^2)^\perp$ is a characteristic ideal of dimension $9$.\\ \\
We now turn to the description of a characteristic ideal of dimension $5$. From the general theory we also know that $x_1y_2, y_1y_2$ are linearly independent. Thus $L^4L^3 = {\mathbb F} x_1y_2 + {\mathbb F} y_1y_2$ is a $2$-dimensional characteristic subspace of $L^6$.
Let $I_1 = L^5L^2, I_2 = L^7$ and $I_3 = L^6$. Let $k$ be smallest such that $I_k \cap L^4L^3 \neq \{0\}$. Then
\[ U = \{ x \in L^4:\, xL^3 \leq I_k \} \]
is a characteristic ideal of dimension $5$. We can thus further refine our basis such that we have the following situation. %\\ \\
%
%&&&&&&&&&&&&&&&&&&&&&&&& Pictuer &&&&&&&&&&&&&&&&&&&&&&&&&&&&&&
\begin{alignat*}{2}
\color{magenta}L^5L^2\ &\boxed{
\begin{matrix}
\color{green} x_5
\end{matrix}
}\ \ 
\begin{matrix}
y_5
\end{matrix} \\
\color{cyan}L^7=Z(L) \ &\boxed{
\begin{matrix}
\color{green} x_4
\end{matrix}
}
\boxed{\begin{matrix}
\color{green}y_4
\end{matrix}} \ \color{magenta}(L^5L^2)^\perp\\
\color{cyan}L^6=Z_2(L)\ &\boxed{
\begin{matrix}
\color{green}x_3
\end{matrix}
}
\boxed{
\begin{matrix}
 \color{green}y_3
\end{matrix}} \ \color{cyan}L^2=Z_6(L)\\
\color{cyan}L^5=Z_3(L)\ &\boxed{
\begin{matrix}
\color{blue} x_2
\end{matrix}}
\boxed{
\begin{matrix}
\color{blue}y_2
\end{matrix}
}\ \color{cyan}L^3=Z_5(L)\\
\color{magenta}U\ &\boxed{
\begin{matrix}
\color{blue} x_1
\end{matrix}}
\boxed{
\begin{matrix}
\color{blue}y_1
\end{matrix}
}\ \color{cyan}L^4=Z_4(L)
\end{alignat*}
There are now few separate cases to consider according to whether $L^4L^3 = L^7$ or $L^4L^3 \neq L^7$ and whether or not $L^4L^3 \cap L^5L^2 \neq \{0\}$.
\subsection{Algebras where $L^4L^3 = L^7 $}
In this case we have 
\[ {\mathbb F} x_2y_3 = {\mathbb F} x_1y_2 = {\mathbb F} x_5,\ 
{\mathbb F} x_5 + {\mathbb F} y_1y_2 = {\mathbb F} x_5 + {\mathbb F} x_4.\]
Now consider the characteristic subspace $U L^2 = {\mathbb F} x_5 + {\mathbb F} x_1y_3$. There are again two subcases to consider
as either $UL^2$ has dimension $1$ or $2$.
%
%+++++++-----------------++++++++--------------------------++++++++---------------------++++++++++++------------------------------++++++++++++
%
\subsubsection{I. Algebras where $ UL^2$ is $1$-dimensional }
In this case we have that $x_1y_3 \in {\mathbb F} x_5$ and it follows that $L^4L^2 = UL^2 + y_1 L^2 = {\mathbb F} x_5 + y_1L^2 = Z(L)$. Consider the characteristic subspace
\[ V = \{ x \in L^2:\, L^4x\leq L^5L^2 \}.\]
Then $V$ is of codimension $1$ in $L^2$ and $L^4 \leq V$. Also $y_2 \not \in V$. One sees readily that we can refine our choice of basis further such that
\[ V = L^4 + {\mathbb F} y_3. \]
In particular $y_1y_3 \leq {\mathbb F} x_5$. Next consider the characteristic subspace
\[ W = \{ x \in L^4:\, xV =0 \}.\]
We have that $L^6 \leq W$ and that $x_2 \not \in W$. Also $W$ is the kernel of the surjective linear map $L^4 \rightarrow L^5L^2,\,x\mapsto xy_3$ and thus of codimension $1$ in $L^4$. We can now pick our basis further such that
\[ W = L^6 + {\mathbb F} x_1 + {\mathbb F} y_1 .\]
It is not difficult to see that such a choice is compatible to what we have done so far. Notice that it follows that $y_1y_3 = x_1y_3 =0$. Next one notices that $L^3V = Z(L)$ and that $L^3U \leq L^5L^2$. Let
\[ Z = \{ x \in V:\, L^3x \leq L^5L^2 \}.\]
Then $Z$ is of codimension $1$ in $V$ and $y_1 \not \in Z$. We can now further refine the basis such that
\[ Z = U +  {\mathbb F} y_3.\]
The reader can convince himself that this is compatible to our choice so far. In particular $y_2y_3 \in {\mathbb F} x_5$. Replacing $y_2$ by a suitable $y_2 - \alpha x_2$, we can furthermore assume that $y_2y_3=0$.
With this choice of basis we thus have $x_1y_3 = y_1y_3=y_2y_3=0$ as well as ${\mathbb F} x_1y_2 = {\mathbb F} x_2y_3 = {\mathbb F} x_5$ and ${\mathbb F} x_5 + {\mathbb F} x_4 = {\mathbb F} x_5 + {\mathbb F} y_1y_2$. It is not difficult to see that we can further refine our basis such that
\begin{equation}\label{eq:32II-I.10} \begin{aligned} x_1y_3 = y_1y_3 = y_2y_3= 0,\  x_1y_2 =  x_2y_3 = x_5,\ y_1y_2= x_4.
\end{aligned} \end{equation}
We are then only left with the following triple values
\begin{alignat*}{3}
	(x_{1}y_{4},y_{5})&=a,&\quad  (y_{1}y_{4},y_{5})&=c,&\quad  (y_3 y_4, y_5)&=e,\\
\mbox{} (x_{2}y_{4},y_{5})&=b,&\quad  (y_2 y_4, y_5) &=d,&\quad  (x_{3}y_{4},y_{5})&=r,
\end{alignat*}
where $r \neq 0$. Replacing $x_3, x_2, y_1, y_3, y_4, y_5$ by $x_3+ b x_4, x_2+a x_4 - c x_5, y_1+d x_4, y_3 -(e/r) x_3 - (be/r) x_4, y_4 = y_4 -b y_3 -a y_2 + d x_1, y_5+ cy_2$, gives us a new standard basis where we can assume that $a= b= c= d= e =0$. The reader can check that \eqref{eq:32II-I.10} is not affected by these changes. Finally by replacing $x_1, x_2, \ldots, y_5$ by $(1/r^3)x_1, (1/r)x_2, r x_3,$
$r^4 x_4, (1/r^2) x_5,$ 
$r^3 y_1,$
$r y_2, (1/r) y_3,$
$(1/r^4) y_4, r^2 y_5,$
we can furthermore assume that $r=1$. We thus arrive at a unique presentation for $L$.
\begin{Proposition}\label{pro33I}
There is a unique nilpotent SAA $L$ of dimension $10$ that has isotropic center of dimension $2$ with the further properties that the class is $7$, $L^4L^3 = L^7$ and $\mbox{dim\,}UL^2=1$. This algebra can be given by the presentation
\begin{align*}
{\mathcal P}_{10}^{(2,3)}: \ \  (x_3 y_4, y_5)=1,\  (x_2 y_3, y_5)=1,\  (x_1y_2,y_5)=1,\  (y_1y_2,y_4)=1. 
\end{align*} 
\end{Proposition} \noindent
Direct calculations show that the algebra with this presentation has the properties stated.
%
%+++-----------------------------+++--------------------------------+++----------------------------------+++------------------------------+++
%
\subsubsection{II. Algebras where $UL^2$ is $2$-dimensional }
In this case ${\mathbb F}x_2y_3 = {\mathbb F}x_1y_2 = {\mathbb F}x_5$ and ${\mathbb F}y_1y_2 + {\mathbb F}x_5 = {\mathbb F}x_1y_3 + {\mathbb F}x_5 = {\mathbb F}x_4 + {\mathbb F}x_5$. It is not difficult to see that we can choose our standard basis such that
\begin{equation}\label{eq:32II-I.1} \begin{aligned}
  x_1 y_2= x_2y_3=x_5,\  y_1y_2=x_4. 
\end{aligned} \end{equation}
Now $y_1y_3 = a x_5 + b x_4$ for some $a, b \in {\mathbb F}$. Replacing $x_2, y_1, y_2, y_3$ by $x_2 + b x_3, y_1 - a x_2, y_2 - a x_1, y_3 - b y_2 + ab x_1$ gives
\begin{equation}\label{eq:32II-I.2} \begin{aligned}
  y_1y_3=0, 
\end{aligned} \end{equation}
and the changes do not affect \eqref{eq:32II-I.1}. Next consider $y_2y_3 = a x_4 + b x_5$ and replace $x_1, y_2, y_3$ by $x_1 - a x_3,
 y_2  - b x_2, y_3 + a y_1$. These changes imply that we can assume furthermore that
\begin{equation}\label{eq:32II-I.3} \begin{aligned}
  y_2y_3=0.
\end{aligned} \end{equation}
Now consider $x_1y_3 = a x_5 + b x_4$ (where $b \neq 0$ by our assumptions). Replacing $x_1, y_2$ by $x_1 - a x_2, y_2 + a y_1$ we can assume that $a=0$. Then replace $x_1, \ldots, y_5$ by $(1/b)x_1, (1/b^2)x_2,$
$(1/b^3) x_3, b^3  x_4, b x_5$
$, b y_1, b^2 y_2, b^3 y_3,$
$ (1/b^3) y_4, (1/b) y_5$ 
we can assume that $b=1$. Thus
\begin{equation}\label{eq:32II-I.4} \begin{aligned}
  x_1y_3=x_4.
\end{aligned} \end{equation}
This leaves us with the following triples.
\begin{alignat*}{3}
	(x_{1}y_{4},y_{5})&=a,&\quad (y_{1}y_{4},y_{5})&=c,&\quad  (y_3 y_4, y_5)&=e,\\
\mbox{} (x_{2}y_{4},y_{5})&=b,&\quad  (y_2 y_4, y_5) &=d,&\quad  (x_{3}y_{4},y_{5})&=r,
\end{alignat*}
First replace $x_3, x_2, y_1, y_4$ by $x_3 + b x_4, x_2 + a x_4, y_1 + d x_4, y_4- a y_2 + d x_1 - b y_3$. These changes imply that we can assume that
\begin{equation}\label{eq:32II-I.5} \begin{aligned}
 (x_{1}y_{4},y_{5})  = (x_{2}y_{4},y_{5}) = (y_2 y_4, y_5) =0 .
\end{aligned} \end{equation}
Finally replace $x_2, y_3, y_5$ by $x_2 - c x_5, y_3 -(e/r)x_3, y_5  + c y_2$ and we furthermore assume that
\begin{equation}\label{eq:32II-I.6} \begin{aligned}
 (y_{1}y_{4},y_{5})  = (y_{3}y_{4},y_{5}) =0 .
\end{aligned} \end{equation}
Thus $L$ has a presentation of the form ${\mathcal P}_{10}^{(2,4)}(r)$ as described in the next proposition.
\begin{Proposition}\label{pro33Ia}
Let $L$ be a nilpotent SAA of dimension $10$ with an isotropic center of dimension $2$ that is of class $7$  and has the further properties that $L^4L^3 =  L^7$ and $\mbox{dim\,} UL^2=2$. This algebra can be given by a presentation of the form
\begin{align*}
{\mathcal P}_{10}^{(2,4)}(r): \ \  (x_3 y_4, y_5)&=r,\ (x_2 y_3, y_5)=1,\ (x_1y_2,y_5)=1,\ (x_1 y_3, y_4)=1,\\
					  (y_1y_2,y_4)&=1,
\end{align*} 
where $r \neq 0$. Furthermore two such presentations ${\mathcal P}_{10}^{(2,4)}(r)$ and ${\mathcal P}_{10}^{(2,4)}(s)$ describe the same algebra if and only if $s/r \in ({\mathbb F}^{*})^{11}$.
\end{Proposition}
\begin{proof}
We have already seen that any such algebra has such a presentation. Direct calculations show that an algebra with a presentation ${\mathcal P}_{10}^{(2,4)}(r)$ has the properties stated. We turn to the isomorphism property. To see that the condition is sufficient, suppose we have an algebra $L$ with a presentation 
${\mathcal P}_{10}^{(2,4)}(r)$ with respect to some standard basis $x_1, y_1, \ldots, x_5, y_5$.
Suppose that $s/r=a^{11}$ for some $a \in {\mathbb F}^*$. Consider a new standard basis 
 $\tilde{x_{1}}=ax_1,\ \tilde{y_{1}}=(1/a)y_1,\ \tilde{x_{2}}=a^3 x_2,\ \tilde{y_{2}}=(1/a^3) y_2,\ \tilde{x_{3}}=a^5 x_3,\  \tilde{y_{3}}=(1/a^5) y_3,\ \tilde{x_{4}}=(1/a^4)x_4,\ \tilde{y_{4}}=a^4 y_4,\ \tilde{x_5}=(1/a^2) x_5,\ \tilde{y_5}=a^2 y_5$. 
Calculations show that $L$ has then presentation ${\mathcal P}_{10}^{(2,4)}(s)$ with respect to the new standard basis.\\\\
It is only remains now to see that the conditions is also necessary. Consider an algebra $L$ with presentation ${\mathcal P}_{10}^{(2,4)}(r)$ with respect to some standard basis $x_1, y_1, \ldots, x_5, y_5$. Take some arbitrary new standard basis $\tilde{x_1}, \tilde{y_1}, \ldots, \tilde{x_5}, \tilde{y_5}$ such that $L$ satisfies the presentation ${\mathcal P}_{10}^{(2,4)}(s)$ with respect to the new basis.
Using the fact that we have an ascending chain of characteristic ideals we know that
\begin{align*}
\tilde{y_1}&=ay_1+\beta_{11}x_1+\cdots+\beta_{15}x_5,\\
\tilde{y_2}&=by_2+\alpha_{21}y_1+\beta_{21}x_1+ \cdots +\beta_{25}x_5, \\
\tilde{y_3}&=cy_3+\alpha_{32}y_2+\alpha_{31}y_1+\beta_{31}x_1+ \cdots +\beta_{35}x_5,\\
\tilde{y_4}&=dy_4+\alpha_{43}y_3+\alpha_{42}y_2+\alpha_{41}y_1+\beta_{41}x_1+\cdots+ \beta_{45}x_5,\\
\tilde{y_5}&=ey_5+\alpha_{54}y_4+ \cdots +\alpha_{51}y_1+\beta_{51}x_1+\cdots +\beta_{55}x_5, \\
\tilde{x_1}&=(1/a)x_1+\gamma_{12}x_2+ \cdots +\gamma_{15}x_5,\\
\tilde{x_2}&=(1/b)x_2+\gamma_{23}x_3+\gamma_{24}x_4+\gamma_{25}x_5,\\
\tilde{x_3}&=(1/c)x_3+\gamma_{34}x_4+\gamma_{35}x_5,\\
\tilde{x_4}&=(1/d)x_4+\gamma_{45}x_5,\\
\tilde{x_5}&=(1/e)x_5,
\end{align*}
for some $\alpha_{ij}, \beta_{ij}, \gamma_{ij}, a, b, c, d, e $ where $a, b, c, d, e \neq 0$. Direct calculations show that
\begin{eqnarray*}
		   1= (\tilde{x_1}\tilde{y_2},\tilde{y_5})=be/a      &  \Rightarrow & e=a/b \\
\mbox{} 1= (\tilde{y_1}\tilde{y_2},\tilde{y_4})  =abd      & \Rightarrow &  d=1/(ab)\\
\mbox{} 1= (\tilde{x_1} \tilde{y_3},\tilde{y_4}) =cd/a     & \Rightarrow & c=a^2b\\
\mbox{} 1= (\tilde{x_2}\tilde{y_3},\tilde{y_5})   =ce/b     & \Rightarrow &  b= a^3.
\end{eqnarray*}
Thus $b= a^3$, $c = a^5$, $d = 1/a^4$, $e = 1/a^2$ and it follows that
\[ s = (\tilde{x_3}\tilde{y_4},\tilde{y_5})=(de/c) r = (1/a)^{11}r. \]
Hence $s/r \in ({\mathbb F}^{*})^{11}$.
\end{proof}
\begin{Remark} It follows that if $({\mathbb F}^{*})^{11} = {\mathbb F}^{*}$ then there is only one algebra of this type. This includes any algebraically closed field and ${\mathbb R}$. If ${\mathbb F}$ is a finite field of order $p^n$, then the number of algebras is either $11$ or $1$ according to whether $11$ divides $p^n -1$ or not. Notice also that there are infinitely many algebras over ${\mathbb Q}$.
\end{Remark}
\subsection{Algebras where $ L^4L^3 \neq L^7 $ and $L^5 L^2 \leq L^4L^3$}
Here we pick our standard basis such that
\[ L^4L^3 = {\mathbb F} x_5+ {\mathbb F} x_3.\]
Notice also that as before $U = \{ x \in L^4:\, xL^3 \leq L^5L^2 \}$ and thus again $x_1y_2 \in {\mathbb F} x_5$. Notice also that
\[(L^4L^3)^\perp = {\mathbb F} x_5 + \cdots +{\mathbb F} x_1 + {\mathbb F} y_1 + {\mathbb F} y_2+ {\mathbb F} y_4.\]
Then $L^5 (L^4L^3)^\perp = ({\mathbb F} x_3 + {\mathbb F} x_2)y_4 = L^5L^2$. Consider the characteristic subspace
\[ V = \{ x \in L^5:\, x(L^4L^3)^\perp=0 \}.\]
Here $x_3y_4 \neq 0$ and thus $V$ is the kernel of a surjective linear map $L^5 \rightarrow L^5L^2,\, x \mapsto xy_4$ and has codimension $1$ in $L^5$. We pick our standard basis such that
\[ V = {\mathbb F} x_5 + {\mathbb F} x_4 + {\mathbb F} x_2.\]
In particular 
\begin{equation}\label{eq:32II-I.ab.1} \begin{aligned} x_2y_4=0. \end{aligned}\end{equation}
Here we have again $UL^2 = {\mathbb F} x_5+ {\mathbb F} x_1y_3$ and thus either the dimension of $UL^2$ is $1$ or $2$. We consider these cases separately.
%
%+++++++-----------------++++++++--------------------------++++++++---------------------++++++++++++------------------------------++++++++++++
%
\subsubsection{I. Algebras where $ UL^2$ is $1$-dimensional}
Notice that
\[V^\perp = {\mathbb F} x_5 + \cdots + {\mathbb F} x_1 +{\mathbb F}y_{1}+ {\mathbb F} y_3.\]
and that $UV^\perp =  {\mathbb F} x_5 = L^5L^2$. Let 
\[ W = \{ x \in U:\, xV^\perp =0 \}.\]
Here $x_2y_3 \neq 0$ and $W$ is the kernel of the surjective linear map $U \rightarrow L^5L^2,\  x\mapsto xy_3$. We choose our standard basis further such that
\[ W = {\mathbb F} x_5 + {\mathbb F} x_4 + {\mathbb F} x_3 + {\mathbb F} x_1 .\]
In particular 
\begin{equation}\label{eq:32II-I.b.01} \begin{aligned} x_1y_3=0. \end{aligned} \end{equation}
Next look at $L^4V^\perp =  {\mathbb F} x_5 +  {\mathbb F} y_1y_3$. Notice that $y_1y_3 \in V$ and that $(y_1y_3, y_2) \neq 0$ (as $ {\mathbb F} y_1y_2  + {\mathbb F} x_5=  {\mathbb F} x_3 +  {\mathbb F} x_5$). We choose our basis further such that
\[ L^4V^\perp =  {\mathbb F} x_5 +  {\mathbb F} x_2.\]
In particular 
\begin{equation}\label{eq:32II-I.b.02} \begin{aligned} (y_1y_3, y_4)=0. \end{aligned} \end{equation}
Now consider the characteristic subspace
\[ T = L^4V^\perp + L^4L^3  = {\mathbb F} x_5 + {\mathbb F} x_3 + {\mathbb F} x_2.\]
Notice that $T^\perp = L^4 + {\mathbb F} y_4$ and $W T^\perp = {\mathbb F} x_3y_4 + {\mathbb F} x_1y_4$. Let
\[ R = \{x\in W:\, xT^\perp = 0\}.\]
We have $x_3y_4 \neq 0$ and $R$ is the kernel of the surjective linear map $W \rightarrow L^5L^2,\, x\mapsto xy_4$. We now refine our basis further such that
\[R = {\mathbb F} x_5  + {\mathbb F} x_4 +  {\mathbb F} x_1 .\]
In particular $x_1y_4=0$. We have thus got a basis where $x_1y_3 = x_1y_4 = x_2y_4=0$ and where $(y_1y_3, y_4)=0$. It is not difficult that we can furthermore assume that
\begin{alignat*}{4}
 x_3 y_4 &= x_5, &\quad x_2y_3 &= r x_5, &\quad x_2y_4&=0, &\quad x_1y_2&=x_5,\\
 x_1y_3 &= 0, &\quad x_1y_4 &= 0, &\quad y_1y_2&=x_3, & \quad (y_1y_3, y_4) &= 0. 
 \end{alignat*}
We still need to consider the following triples.
\begin{alignat*}{3}
(y_3y_4, y_5)&=a, &\quad (y_1y_4,y_5)&=c, &\quad  (y_2y_3, y_5)&=e,\\
(y_1y_3, y_5)&=b, &\quad (y_2y_3, y_4)&=d, &\quad  (y_2y_4, y_5)&=f,\\
(x_2y_3, y_5)&=r,&\quad  & \quad 
 \end{alignat*}
We start by replacing $x_2$, $x_1,$ $y_4$, $y_5$ by $x_2 - bx_5$, $x_1+dx_4$, $y_4-dy_1$, $y_5+ by_2$ and 
\begin{equation}\label{eq:32II-I.b.03} \begin{aligned} (y_1y_3, y_5) = (y_2y_3, y_4)=0. \end{aligned} \end{equation}
Then replace $y_1, y_2, y_3$ with $y_1-c x_3$, $y_2-[(e-c)/r]x_2-f x_3$, $y_3-cx_1-f x_2-ax_3$ and we can furthermore assume that
\begin{equation}\label{eq:32II-I.b.04} \begin{aligned}
(y_1y_4, y_5) = (y_2y_4, y_5)=(y_2y_3, y_5)=(y_3y_4, y_5)=0.
 \end{aligned} \end{equation}
It follows that $L$ has a presentation of the form ${\mathcal P}_{10}^{(2,5)}(r)$ as in the following proposition.
\begin{Proposition}\label{pro33Ia}
Let $L$ be a nilpotent SAA of dimension $10$ with an isotropic center of dimension $2$ that is of class $7$ and the further properties that $L^4L^3 \neq L^7$, $L^5L^2 \leq L^4L^3$ and $UL^2$ is $1$-dimensional. This algebra can be given by a presentation of the form
\begin{align*}
{\mathcal P}_{10}^{(2,5)}(r):\ \ (x_2 y_3, y_5)=r,\  (x_3 y_4, y_5)=1\ , (x_1y_2,y_5)=1,\  (y_1y_2,y_3)=1,\  
\end{align*} 
where $r \neq 0$. Furthermore two such presentations ${\mathcal P}_{10}^{(2,5)}(r)$ and ${\mathcal P}_{10}^{(2,5)}(s)$ describe the same algebra if and only if $s/r \in ({\mathbb F}^{*})^3$.
\end{Proposition}
\begin{proof}
We have already seen that any such algebra has a presentation of this form. Conversely, direct calculations show that any algebra with a 
presentation ${\mathcal P}_{10}^{(2,5)}(r)$ satisfies the properties stated. We turn to the isomorphism property. To see that the condition is sufficient, suppose we have an algebra $L$ with presentation 
${\mathcal P}_{10}^{(2,5)}(r)$ with respect to some standard basis $x_1, y_1, \ldots, x_5, y_5$. Suppose that $s/r=a^{3}$ for some $a \in {\mathbb F}^*$.
Consider a new standard basis 
$\tilde{x_{1}}=x_1,\ \tilde{y_{1}}=y_1,\ \tilde{x_{2}}=ax_2,\ \tilde{y_{2}}=(1/a)y_2,\ \tilde{x_{3}}=(1/a)x_3,\ \tilde{y_{3}}=ay_3,\ \tilde{x_{4}}=x_4,\ \tilde{y_{4}}=y_4,\ \tilde{x_5}=(1/a)x_5,\ \tilde{y_5}=ay_5$. 
Calculations show that $L$ has then the presentation ${\mathcal P}_{10}^{(2,5)}(s)$ with respect to the new basis.\\\\
It only remains to see that the condition is also necessary. Consider an algebra $L$ with presentation ${\mathcal P}_{10}^{(2,5)}(r)$
with respect to some standard basis $x_1, y_1, \ldots, x_5, y_5$. Take some arbitrary new standard basis 
$\tilde{x_1}, \tilde{y_1}, \ldots, \tilde{x_5}, \tilde{y_5}$
such that $L$ also satisfies the presentation ${\mathcal P}_{10}^{(2,5)}(s)$ with respect to the new basis. 
Using the fact that we have an ascending chain of characteristic ideals as well as the fact that 
$L^4L^3 = {\mathbb F} x_5  + {\mathbb F} x_3$, 
$L^4V^\perp = {\mathbb F} x_5  + {\mathbb F} x_2$, 
$R ={\mathbb F} x_5  + {\mathbb F} x_4 + {\mathbb F} x_1$, 
$W^\perp ={\mathbb F} y_2  + U$, 
$(R + L^4V^\perp)^\perp ={\mathbb F} y_3  + U$
and
$T^\perp ={\mathbb F} y_4  +{\mathbb F} y_1  + U$
are characteristic subspaces, we know that
\begin{align*}
\tilde{y_1}&=(1/a)y_1+\beta_{11}x_1+\cdots+\beta_{15}x_5,\\
\tilde{y_2}&=(1/b)y_2+\beta_{21}x_1+ \cdots +\beta_{25}x_5, \\
\tilde{y_3}&=(1/c)y_3+\beta_{31}x_1+ \cdots +\beta_{35}x_5,\\
\tilde{y_4}&=(1/d)y_4 +\alpha_{41}y_1+\beta_{41}x_1+\cdots+ \beta_{45}x_5,\\
\tilde{y_5}&=(1/e)y_5+\alpha_{54}y_4+ \cdots +\alpha_{51}y_1+\beta_{51}x_1+\cdots +\beta_{55}x_5, \\
\tilde{x_1}&=ax_1+\gamma_{14}x_4 +\gamma_{15}x_5,\\
\tilde{x_2}&=b x_2+\gamma_{25}x_5,\\
\tilde{x_3}&=cx_3+\gamma_{35}x_5,\\
\tilde{x_4}&=d x_4+\gamma_{45}x_5,\\
\tilde{x_5}&=ex_5,
\end{align*}
for some $ a, b, c, d, e, \alpha_{ij}, \beta_{ij}, \gamma_{ij}$ where $a, b, c, d, e \neq 0$. 
It follows that
\begin{align*}
1&= ( \tilde{x_3} \tilde{y_4}, \tilde{y_5})=c/(de)                   \\
\mbox{} 1&=( \tilde{x_1} \tilde{y_2}, \tilde{y_5})  =a/(be)     \\
\mbox{} 1&= ( \tilde{y_1} \tilde{y_2}, \tilde{y_3})=1/(abc).
%
%\mbox{} 1= (\tilde{x_2}\tilde{y_3},\tilde{y_5})   =ce/b     & \Rightarrow &  b= a^3
%
\end{align*}
This gives $c=1/(ab)$, $e=a/b$, $d=1/a^2$ and then
\[ s = (\tilde{x_2}\tilde{y_3},\tilde{y_5})=br/(ce) =b^3 r. \]
This finishes the proof.
\end{proof}
\begin{Remark}
 Again we just got one algebra if $({\mathbb F}^*)^{3} = {\mathbb F}^*$. 
This includes all fields that are algebraically closed as well as ${\mathbb R}$.
For a finite field of order $p^n$ there are $3$ algebras if $3 | p^n-1$ but otherwise one.
For ${\mathbb Q}$ there are infinitely many algebras.
\end{Remark}
%
%+++++++-----------------++++++++--------------------------++++++++---------------------++++++++++++------------------------------++++++++++++
%
\subsubsection{II. Algebras where $UL^2$ is $2$-dimensional }
Recall that $UL^2 =  {\mathbb F}  x_5 +  {\mathbb F}  x_1y_3$, $L^4 L^3  = {\mathbb F}  x_5 + {\mathbb F}  x_3$ and $x_2y_4=0$.
It is not difficult to see that one can further refine the basis such that
\begin{equation}\label{eq:32II-I.b.1} \begin{aligned}
x_3y_4= \alpha x_5,\ x_2y_3=rx_5,\ x_2y_4=0,\ x_1y_2=x_5,\ x_1y_3=x_4,\ y_1y_2=x_3.
\end{aligned} \end{equation}
Replacing $x_1, y_1, \ldots, x_5, y_5$ by $\alpha x_1$, $(1/\alpha)y_1$, $(1/\alpha^3)x_2$, $\alpha^3 y_2$, $\alpha^2 x_3$,
 $(1/\alpha^2)y_3$, $(1/\alpha) x_4$, $\alpha y_4$, $\alpha^4 x_5$, $(1/\alpha^4)y_5$, implies that we can furthermore assume that
 $\alpha=1$. We have also the following triples to sort out.
\begin{alignat*}{3}
(y_1y_3, y_4)&=a, &\quad  (y_2y_3,y_4)&=d, &\quad (x_1y_4, y_5)&=g,\\
(y_1y_3, y_5)&=b, &\quad  (y_2y_3, y_5)&=e, &\quad (y_3y_4, y_5)&=h.\\
(y_1y_4, y_5)&=c, &\quad (y_2y_4, y_5)&=f, &\quad
\end{alignat*}
First we replace $x_2, x_1, y_4, y_5$ by $x_2 - a x_4 - b x_5,$ $x_1+dx_4$, $y_4+ay_2-d y_1$, $y_5+by_2$ and we see that we can assume that
\begin{equation}\label{eq:32II-I.b.2} \begin{aligned}
(y_1y_3, y_4)=(y_1y_3, y_5)= (y_2y_3,y_4)=0.
\end{aligned} \end{equation}
Next replace $y_1, y_2, y_3$ by $y_1-c x_3$, $y_2-f x_3$, $y_3-c x_1-f x_2$ and we can also assume that
\begin{equation}\label{eq:32II-I.b.3} \begin{aligned}
(y_1y_4, y_5)=(y_2y_4, y_5)=0.
\end{aligned} \end{equation}
Then replace $x_1, y_5$ by $x_1+ e x_5$, $y_5-ey_1$ and we see now add
\begin{equation}\label{eq:32II-I.b.4} \begin{aligned}
(y_2y_3, y_5)=0.
\end{aligned} \end{equation}
Finally replace $x_1, y_3$ by $x_1 - g x_3$, $y_3+g y_1 - h x_3$ and we see that we can now add
\begin{equation}\label{eq:32II-I.b.5} \begin{aligned}
(x_1y_4, y_5) = (y_3y_4, y_5)=0.
\end{aligned} \end{equation}
We have thus see that $L$ has a presentation ${\mathcal P}_{10}^{(2,6)}(r)$ as in the following proposition.
\begin{Proposition}\label{pro33Ia}
Let $L$ be a nilpotent SAA of dimension $10$ with an isotropic center of dimension $2$ that is of class $7$ and has further properties that $L^4L^3 \neq L^7 $, $L^5L^2 \leq L^4L^3$ and $UL^2$ is $2$-dimensional. This algebra can be given by a presentation of the form
\begin{align*}
{\mathcal P}_{10}^{(2,6)}(r): \ \ (x_2 y_3, y_5)&=r,\ (x_3 y_4, y_5)=1,\  (x_1y_2,y_5)=1,\ (x_1 y_3, y_4)=1,\\
					 (y_1y_2,y_3)  &=1
\end{align*} 
where $r \neq 0$. Furthermore two such presentations ${\mathcal P}_{10}^{(2,6)}(r)$ and ${\mathcal P}_{10}^{(2,6)}(s)$ describe the same algebra if and only if $s/r \in ({\mathbb F}^{*})^{12}$.
\end{Proposition}
\begin{proof}
We have already seen that any such algebra has a presentation of this form. Conversely, direct calculations show that any algebra with a presentation ${\mathcal P}_{10}^{(2,6)}(r)$ satisfies the properties stated. We turn to the isomorphism property. To see the condition is sufficient, suppose we have an algebra $L$ with presentation ${\mathcal P}_{10}^{(2,5)}(r)$ with respect to some standard basis $x_1, y_1, \ldots, x_5, y_5$. Suppose that $s/r=a^{12}$ for some $a \in {\mathbb F}^*$.
Consider a new standard basis  
$\tilde{x_{1}}=(1/a)x_1,\ \tilde{y_{1}}=ay_1,\ \tilde{x_{2}}=a^4x_2,\ \tilde{y_{2}}=(1/a^4)y_2,\ 
\tilde{x_{3}}=(1/a^3)x_3,\ \tilde{y_{3}}=a^3y_3,\ \tilde{x_{4}}=a^2x_4,\ \tilde{y_{4}}= (1/a^2)y_4,\ 
\tilde{x_5}=(1/a^5)x_5,\ \tilde{y_5}=a^5y_5$.
Calculations show that $L$ has then the presentation ${\mathcal P}_{10}^{(2,6)}(s)$ with respect to the new basis. \\\\
It only remains to see that the condition is also necessary. Consider an algebra $L$ with presentation ${\mathcal P}_{10}^{(2,6)}(r)$ with respect to some standard basis $x_1, y_1, \ldots, x_5, y_5$. Take some arbitrary new standard basis
$\tilde{x_1}, \tilde{y_1}, \ldots, \tilde{x_5}, \tilde{y_5}$
such that $L$ also satisfies the presentation ${\mathcal P}_{10}^{(2,6)}(s)$ with respect to the new basis.
Using the fact that
we have an ascending chain of characteristic ideals as well as the fact that $L^4L^3 = {\mathbb F} x_5 + {\mathbb F} x_3$, $V = {\mathbb F} x_5 + {\mathbb F} x_4 + {\mathbb F} x_2$, are characteristic subspaces, we know that
\begin{align*}
\tilde{y_1}&=(1/a)y_1+\beta_{11}x_1+\cdots+\beta_{15}x_5,\\
\tilde{y_2}&=(1/b)y_2+\alpha_{21}y_1+\beta_{21}x_1+ \cdots +\beta_{25}x_5, \\
\tilde{y_3}&=(1/c)y_3+\alpha_{31}y_1+\beta_{31}x_1+ \cdots +\beta_{35}x_5,\\
\tilde{y_4}&=(1/d)y_4 +\alpha_{42}y_2+\alpha_{41}y_1+\beta_{41}x_1+\cdots+ \beta_{45}x_5,\\
\tilde{y_5}&=(1/e)y_5+\alpha_{54}y_4+ \cdots +\alpha_{51}y_1+\beta_{51}x_1+\cdots +\beta_{55}x_5, \\
\tilde{x_1}&=ax_1+\gamma_{12}x_2 + \ldots +\gamma_{15}x_5,\\
\tilde{x_2}&=b x_2+\gamma_{24}x_4+\gamma_{25}x_5,\\
\tilde{x_3}&=cx_3+\gamma_{35}x_5,\\
\tilde{x_4}&=d x_4+\gamma_{45}x_5,\\
\tilde{x_5}&=ex_5,
\end{align*}
for some $ a, b, c, d, e, \alpha_{ij}, \beta_{ij}, \gamma_{ij} $ where $a, b, c, d, e \neq 0$. 
It follows that
\begin{align*}
1 &= ( \tilde{x_3} \tilde{y_4}, \tilde{y_5})=c/(de)                   \\
\mbox{} 1 &=( \tilde{x_1} \tilde{y_2}, \tilde{y_5})  =a/(be)     \\
\mbox{} 1&= ( \tilde{y_1} \tilde{y_2}, \tilde{y_3})=1/(abc) \\
\mbox{} 1&= (\tilde{x_1}\tilde{y_3},\tilde{y_4})   =a/(cd).
\end{align*}
This gives $b = (1/a^4)$, $c=a^3$, $d=(1/a^2)$, $e=a^5$ and then
\[ s = (\tilde{x_2}\tilde{y_3},\tilde{y_5})=br/(ce) =(1/a^{12}) r. \]
This finishes the proof.
\end{proof}
\begin{Remark}
The number of algebras depends thus again on the underling field. In particular there is one algebra over the field ${\mathbb C}$,
two algebras over ${\mathbb R}$ and infinitely many over ${\mathbb Q}$. When ${\mathbb F}$ is a finite field of order $p^n$ then the number can be $12$, $6$, $4$, $3$, $2$ or $1$ depending on what the value of $p^n$ is modulo $12$.
\end{Remark}
%+++------------------------------+++----------------------------------------+++------------------------------------+++---------------------------------+++
\subsection{Algebras where $L^{4}L^{3} \neq L^{7}$ and $L^{5}L^{2} \nleq L^{4}L^{3}$}
First we pick our standard basis such that 
\begin{equation*}\label{eq:32II-III.1} \begin{aligned}  
L^4L^3 \cap L^7 ={\mathbb F}x_4,\  L^4L^3 ={\mathbb F}x_4 + {\mathbb F}x_3. 
\end{aligned} \end{equation*}
Thus in particular ${\mathbb F} x_1y_2 ={\mathbb F} x_4 $. 
From this one sees that $U (L^5L^2)^\perp = L^7 + {\mathbb F} x_1y_4$ where $(x_1y_4, y_2) = - (x_1y_2, y_4) \neq 0$.
We further refine our basis such that 
\[ U \cdot (L^5L^2)^\perp = L^7 + {\mathbb F} x_2 = {\mathbb F} x_5 + {\mathbb F} x_4 + {\mathbb F} x_2.\]
Then notice that $ (U (L^5L^2)^\perp) L =L^5L^2 +  {\mathbb F} x_2y_5$, where $(x_2y_5, y_3) = - (x_2y_3, y_5) \neq 0$.
We refine our basis further such that
$$(U(L^5L^2)^\perp)L = L^5L^2+ {\mathbb F} x_3 = {\mathbb F} x_5 + {\mathbb F} x_3.$$
Notice that in particular $x_2y_5 \in {\mathbb F} x_5 + {\mathbb F} x_3$ and thus $(x_2y_4, y_5) = - (x_2y_5, y_4) =0$. We have as well $(x_2 y_4, y_3) = - (x_2y_3, y_4) =0$ and thus 
\begin{equation}\label{eq:32II-III.2} \begin{aligned} x_2y_4 = 0. \end{aligned} \end{equation}
We also have 
\[{\mathbb F} x_3 = ( {\mathbb F} x_4 + {\mathbb F} x_3) \cap ({\mathbb F} x_5 + {\mathbb F} x_3) = (L^4L^3) \cap ((U (L^5L^2)^\perp) L ).\]
Next consider the characteristic subspace 
\[ V = \{ x \in L^4:\, xL^3 \leq (L^4L^3) \cap ((U (L^5L^2)^\perp) L )\}.\]
Notice that $x_1y_2 \neq 0$ and thus $V$ is the kernel of a surjective linear map $L^4 \rightarrow L^4L^3/{\mathbb F} x_3,\, x\mapsto xy_2+ {\mathbb F} x_3$ and thus of codimension $1$ in $L^4$. Notice also that $L^5 \leq V$. We refine our basis further such that
\[ V = L^5 + {\mathbb F} y_1 = {\mathbb F} x_5 + \cdots + {\mathbb F} x_2 + {\mathbb F} y_1.\]
It follows in particular that 
\begin{equation}\label{eq:32II-III.3} \begin{aligned} {\mathbb F} y_1y_2= {\mathbb F} x_3. \end{aligned} \end{equation}
Notice that
\[ (U (L^5L^2)^\perp)^\perp = {\mathbb F} x_5 + \ldots + {\mathbb F} x_1 + {\mathbb F} y_1 + {\mathbb F} y_3\]
and then $U \cdot (U (L^5L^2)^\perp)^\perp = L^5L^2 + {\mathbb F} x_1y_3$. Now $(x_1y_3, y_2) = - (x_1y_2, y_3) = 0$ but also
$x_1y_4 \in U(L^5L^2)^\perp = {\mathbb F} x_5 + {\mathbb F} x_4 + {\mathbb F} x_2$ and thus $(x_1y_3, y_4) = - (x_1 y_4, y_3)=0$.
It follows that $x_1y_3 \in {\mathbb F} x_5 = L^5L^2$. It follows that $U \cdot (U (L^5L^2)^\perp)^\perp = L^5L^2$.
Consider the characteristic subspace
\[R = \{ x \in U:\, x (U (L^5L^2)^\perp)^\perp = 0 \}.\]
We have that $x_2y_3 \neq 0$ and thus $R$ is the kernel of the surjective linear map $U \rightarrow L^5L^2,\, x\mapsto xy_3$.
Thus $R$ is of codimension $1$ in $U$ and contains $L^6$. Now choose our basis further such that
\[ R = L^6 + {\mathbb F} x_1 = {\mathbb F} x_5 + {\mathbb F} x_4 + {\mathbb F} x_3 + {\mathbb F} x_1.\]
In particular 
\begin{equation}\label{eq:32II-III.4} \begin{aligned} x_1y_3=0. \end{aligned} \end{equation}
Next consider $L^4 \cdot (U (L^5L^2)^\perp)^\perp = L^5L^2 + {\mathbb F} y_1y_3$. Notice that $(y_1y_3, y_2) = - (y_1y_2, y_3) \neq 0$. We can refine our basis further such that
\[ L^4 \cdot (U (L^5L^2)^\perp)^\perp = {\mathbb F} x_5 + {\mathbb F} x_2.\]
In particular 
\begin{equation}\label{eq:32II-III.5} \begin{aligned} (y_1y_3, y_4)=0. \end{aligned} \end{equation}
It is not difficult to see that we can now choose our basis such that
\begin{alignat*}{4}
x_3y_4 &= x_5, &\quad x_2y_3 &= r x_ 5, &\quad x_2y_4 &= 0, &\quad &\\
x_1 y_2 &= x_4,&\quad x_1y_3 &=0, &\quad y_1y_2 &= x_3, &\quad (y_1 y_3, y_4) &=0.
\end{alignat*}
Replacing $x_1, y_1, \ldots, x_5, y_5$ by $(1/r) x_1$, $ry_1$, $rx_2$, $(1/r)y_2$, $x_3$, $y_3$, $(1/r^2)x_4$, $r^2y_4$, $r^2 x_5$, $(1/r^2) y_5$, we see that we can furthermore assume that $r=1$.
We are now left with the triples
\begin{alignat*}{3}
(x_1y_4, y_5)&=a, &\quad (y_2y_3,y_4)&=d, &\quad  (y_3y_4, y_5)&=h,\\
(y_1y_3, y_5)&=b, &\quad (y_2y_3, y_5)&=e, &\quad  \\
(y_1y_4, y_5)&=c, &\quad  (y_2y_4, y_5)&=f, &\quad 
 \end{alignat*}
We now show that we can further refine the basis such that all these values are zero. First replace
$x_2, x_1, y_5$ by $x_2 -  b x_5$, $x_1 + e x_5$, $y_5  + b y_2 -  e y_1$
and we can assume that
\begin{equation}\label{eq:32II-III.6} \begin{aligned} (y_1y_3, y_5)= (y_2y_3, y_5)=0. \end{aligned} \end{equation}
Then replace $x_1, y_1, y_2, y_3$ by $x_1 - a x_3$, $y_1 - c x_3$, $y_2 - f x_3$, $y_3 + a y_1 - c x_1 - f x_2$
and we can furthermore assume 
\begin{equation}\label{eq:32II-III.7} \begin{aligned} (x_1y_4, y_5)=(y_1y_4, y_5)= (y_2y_4, y_5)=0. \end{aligned} \end{equation}
Finally replace $x_1, y_3, y_4$ by $x_1+ d x_4, y_3 - g x_3, y_4 - d y_1$ and we also have
\begin{equation}\label{eq:32II-III.8} \begin{aligned} (y_2y_3, y_4)= (y_3y_4, y_5)=0. \end{aligned} \end{equation}
We have thus arrived at a unique presentation.
\begin{Proposition}\label{pro33Ia}
There is a unique nilpotent SAA of dimension $10$ with an isotropic center of dimension $2$ that is of class $7$ and has the further properties that $L^4L^3 \neq L^7 $, $L^5L^2 \not \leq L^4L^3$. This algebra can be given by the presentation 
\begin{align*}
{\mathcal P}_{10}^{(2,7)}: \ \ (x_2 y_3, y_5)=1,\ (x_3 y_4, y_5)=1,\  (x_1y_2,y_4)=1,\ (y_1y_2,y_3)=1.
\end{align*} 
\end{Proposition}
\begin{proof} We have already seen that any such algebra must have a presentation of this form and conversely direct calculations show that the algebra with the given presentation satisfies all the properties stated.
\end{proof}

% Part III

\part{Powerfully nilpotent groups and powerfully soluble groups}
\noindent

\chapter*{INTRODUCTION}
\vspace{-2.1cm}
\hspace{500cm}
\noindent
\noindent

This part will be mostly about the connection between Part II and Part I.
Thus throughout this part we will be working with an arbitrary field $\mbox{GF\,}(3)$. We also assume that the symplectic alternating algebras is furthermore nilpotent unless stated otherwise.\\\\
The study in Part II reveals some new classes of groups that we call\emph{ powerfully nilpotent group},\emph{ powerfully perfect nilpotent group} and\emph{ powerfully soluble group}. \\\\
As we have seen before, nilpotent symplectic alternating algebras over the field $\mbox{GF\,}(3)$ have a $1$-$1$ correspondence with a class $\mathcal{C}$ of powerful $2$-Engel $3$-groups. \\\\
Our classification of algebras of nilpotent symplectic alternating algebras of dimension up to $10$ over $\mbox{GF\,}(3)$  yields that there are $25$ powerfully nilpotent group over class $\mathcal{C}$ of rank at most $11$.\\\\
In general the situation here like a more further work to be done, which is currently under consideration.
We will only here first consider some general definitions and then in particular describe briefly what is it happening in class $\mathcal{C}$ based on the new language revealed from classification in Part II. \\\\

\thispagestyle{empty}
\chapter{Contributions in Group Theory and Further work}
Here we discuss some further directions that would be interesting and compatible with our work in this thesis. One natural consequence step would be is to investigate the groups that correspond to nilpotent SAA's in class $\mathcal{C}$. Also what is it happening there when the SAA is further nil-algebras over $\mbox{GF\,}(3)$. For instance we know such a correspondence exist.\\\\
A different direction is having the nil SAA's are classified up to $\mbox{dim\,}8$ which are exactly the same as the nilpotent SAA's of $\mbox{dim\,}$ up to $8$. One could investigate some general properties that holds for nil SAA's. Also what could we read more from the additional algebras that we have got as we know that there are nil SAA's, are they all nil-4 SAA's and bounded by nil-degree $4$. \\\\
We however in the rest of this part we continue describing some few things that is just a starting point for an ongoing new work.
\section{Contributions in Group Theory}
As we said before this is a more general setting now that is open and under investigation for a further work to be done. We however describe only what a powerfully nilpotent group is and in particular the connection between the previous parts of this thesis. Similarly we define a powerfully soluble group.\\\\
Recall that there is a one-to-one correspondence between symplectic alternating algebras over the field of three elements and a certain rich class \( \mathcal{C}\) of powerful 2-Engel 3-group of exponent 27. Namely, the groups form a class \( \mathcal{C}\) consist of all powerful 2-Engel 3-groups \(G\) with the following extra properties:
{
\begin{itemize}
    \item[1)] \( G = \langle x, H \rangle \), where \(H  = \{ g \in G \colon g^9=1 \} \) and \( Z(G) = \langle x \rangle \) with \( O(x)=27\),
    \item[2)] \(G\) is of rank \(2r+1\) and has order \(3^{3+4r}\). 
\end{itemize}}
\noindent
The associated symplectic alternating algebra $L(G)$ is constructed as follows. First we consider $L(G) = H/G^3$ as a vector space over GF($3$). To this we associate a bilinear alternating form (,) and an alternating binary multiplication as follows: for any $\bar{a} = a G^3, \bar{b} = b G^3$ and $\bar{c} = c G^3 $ in $L(G)$,
\[ [a,b]^3 = x^{9 (\bar{a}, \bar{b})} \]
\[ \bar{a} \cdot \bar{b} = \bar{c} \mbox{  where  } [a,b] Z(G) = c^3 Z(G).\]

\noindent
Next we identify the subclass of \( \mathcal{C}\) that consists of all groups in \( \mathcal{C}\) that has an extra group theoretical property that we call {\it powerfully nilpotent}. This subclass corresponds to nilpotent SAA's.\\
\begin{defn} A finite $p$-group $G$ is said to be a \emph{powerfully nilpotent group} if there is an ascending chain of powerfully embedded subgroups 
$ H_0, H_1, \cdots, H_n$ such that
\[\{1\}=H_0 \leq H_1=Z(G) \leq H_2  \leq H_3 \leq \cdots \leq H_{n} = G \]
and $[H_{i}, G] = H_{i-1}^p$ for $i = 1, \cdots, n$. 
We refer to such a chain as a powerfully central chain and $n$ is the length of the chain. Further, if $G$ is powerfully nilpotent then the smallest possible length of a powerfully central chain for $G$ is called its \emph{powerful nilpotence class}. 
\end{defn}
\noindent
\begin{Remark} Notice that if $G$ is powerfully nilpotent group, then the nilpotent class of the group $G$ is different from the Powerfully nilpotent class of $G$.\\
\end{Remark}
\noindent
We then define the upper powerfully nilpotent series, lower powerfully nilpotent series and the derived powerfully nilpotent series as usual like in Group.
\noindent
\section{Connection between nilpotent symplectic alternating algebras and powerfully nilpotent groups}
Let $L$ be a nilpotent SAA over $\mbox{GF\,}(3)$ of dimension $2n \geq 6$. We have seen from the general theory that $L$ has a 
 the central ascending chain 
\[\{0\} < I_2 < I_3 < \cdots < I_{n-1} < I_{n-1}^\perp < I_{n-2}^\perp  < \cdots< I_2^\perp < L \]
such that $\mbox{dim\,} I_r = r$ for $r = 2, \ldots, n-1$.  
In particular $L$ is nilpotent of class at most $2n-3$.\\\\
We now move to groups that consist $\mathcal{C}$. Let $G=<x,H>$ be such group.\\\\
%++------------------------------------------------------------------------------------------------------
%
Here the groups consist $\mathcal{C}$ are $2$-Engel, thus they are nilpotent of class at most $3$. We however here interest of those that are of class $3$.\\\\
From our previous work we know that if the SAA $L$ is abelian then the corresponding group $G(L)$ in class $\mathcal{C}$ would have a nilpotent class exactly $2$. This is because $G(L)$ being $2$-Engel has class at most $3$ and that the derived group $[G,G] \neq 1$ since we have a non-degenerate alternating form. It is however clear that the groups that are of class $3$ corresponds to non-abelian SAA over $\mbox{GF\,}(3)$. The reason for that is the latter must have some non-zero triples, say $(uv, w) =1$ but the corresponding commutator would be, say $[a,b,c]=1$ and hence $G(L)$ has class $3$.\\\\
We first start by the following remark that set up the connection between the SAA over $\mbox{GF\,}(3)$ and groups of class $\mathcal{C}$ that is of class $3$.\\ 
\begin{Lemma}[\cite{4}] Let $G(L)$ be a group of class $\mathcal{C}$ that has nilpotent class $3$. $G(L)$ is isomorphic to $G(K)$ if and only if $L$ is isomorphic to $K$.
\end{Lemma}
\noindent
%++-----------------------------------------------------------------------------------------------------------
Recall that ideals in SAA correspond to powerfully embedded subgroups and subalgebra correspond to powerfully subgroup.
We also have that $[a Z(G), bZ(G)] = c^3 Z(G)$ for any $a,b,c \in H$. We next lists few simple consequences.\\\\

\noindent
Let $G$ be a group from $\mathcal{C}$ of rank $2n+1$ that has nilpotent class $3$. Then

\begin{itemize}
    \item[1)] $G$ is powerfully nilpotent group if and only if $L(G)$ is a nilpotent SAA.
\item[2)]  The  powerfully nilpotent class of $G$ is the same as the nilpotent class of $L(G)$.
\item[3)]  the center $Z(G)$ of $G$ and $Z_2(G)$ correspond to the isotropic subspace of $L$ and the center of $L$ consequently.
\item[4)]  $\mbox{rank\,} G = \mbox{dim\,} Z_2(G)/Z(G)$.
\item[5)]  The structure of $G$ is fully determined by the structure of $L$.\\
\end{itemize}

\noindent
Now let $G$ be in the class $\mathcal{C}$ of powerful $2$-Engel $3$-groups. For any $K$ such that $G^3 \leq K \leq G$ we let $\bar{K} = K/G^3$. Notice that
\[ \bar{A} \cdot L(G) \leq \bar{B} \mbox{ if and only if } [ \langle A, x \rangle, G] \leq \langle B, x \rangle^3 .\]
Thus if $G^3 \leq H_i$ for $i=1, \cdots, n$, then
\[ \{0\} = \bar{H_0} \leq \bar{H_1} \leq \cdots \leq \bar{H_n}=L(G) \]
is a central chain of ideals in $L(G)$ if and only if 
\[ \{1\} \leq \langle x \rangle \leq \langle H_0, x \rangle \leq \cdots \leq \langle H_n, x \rangle = G \]
is a powerfully central chain. The classification of the nilpotent symplectic alternating algebras of dimension $10$ over $\mbox{GF}(3)$ gives us thus the classification for the powerfully nilpotent groups in $\mathcal{C}$ that are of rank $11$. The classification reveals that there are $25$ such groups.\\

\noindent
It is also extremely interesting to see that there is a class of {\it maximal powerfully nilpotent groups} consists of powerfully nilpotent groups in $\mathcal{C}$ that corresponds to the maximal class of nilpotent SAA's of class $2n-3$. We seems to have also a dual class of {\it minimal powerfully nilpotent groups} too.\\

\bibliographystyle{plain}
\markboth{Bibliography}{Bibliography}
\addcontentsline{toc}{chapter}{Bibliography}
\bibliography{refrences.bib}
\end{document}